\documentclass[12pt,reqno]{amsart}
\usepackage{amssymb,amsmath,dsfont}
\usepackage{pgfplots}
\pgfplotsset{compat=1.14}
\usepackage{enumitem}
\usepackage{calc}

\newtheorem{thm}{Theorem}[section]
\newtheorem{lemma}[thm]{Lemma}
\newtheorem{prop}[thm]{Proposition}
\newtheorem{definition}[thm]{Definition}
\newtheorem{cor}[thm]{Corollary}
\newtheorem{claim}[thm]{Claim}

\theoremstyle{definition}
\newtheorem{remark}[thm]{Remark}

\def\E{\mathbb{E}}
\def\N{\mathbb{N}}
\def\R{\mathbb{R}}
\def\S{\mathcal{S}}
\def\Z{\mathbb{Z}}
\def\C{\mathbb{C}}

\DeclareMathOperator\real{Re}
\DeclareMathOperator\imag{Im}
\DeclareMathOperator\var{Var}
\DeclareMathOperator\cov{Cov}

\DeclareMathOperator\gama{gamma}
\DeclareMathOperator\supp{Supp}

\DeclareMathOperator\vol{Vol}

\DeclareMathOperator\poisson{Poisson}
\DeclareMathOperator\PD{PD}
\DeclareMathOperator\betadist{beta}

\newenvironment{enumerate-math}
{\begin{enumerate} \addtolength{\itemsep}{5pt}
}
{\end{enumerate}}

\title{Limit distributions for Euclidean random permutations}
\date{\today}

\thanks{Research of D.E. and R.P. supported by ISF grant~861/15 and by ERC starting grant 678520 (LocalOrder). School of Mathematical Sciences, Tel Aviv
University, Tel Aviv 69978, Israel.\\Emails: dorelbom@mail.tau.ac.il,
peledron@post.tau.ac.il}

\author{Dor Elboim}
\address{Dor Elboim\hfill\break
    Tel Aviv University\\
    School of Mathematical Sciences}
\email{dorelbom@mail.tau.ac.il}

\author{Ron Peled}
\address{Ron Peled\hfill\break
    Tel Aviv University\\
    School of Mathematical Sciences\\
    Tel Aviv, 69978, Israel.}
\email{peledron@post.tau.ac.il}
\urladdr{http://www.math.tau.ac.il/~peledron}

\begin{document}
\begin{abstract}
We study the length of cycles in the model of spatial random permutations in Euclidean space. In this model, for given length $L$, density $\rho$, dimension $d$ and jump density $\varphi$, one samples $\rho L^d$ particles in a $d$-dimensional torus of side length $L$, and a permutation $\pi$ of the particles, with probability density proportional to the product of values of $\varphi$ at the differences between a particle and its image under $\pi$. The distribution may be further weighted by a factor of $\theta$ to the number of cycles in $\pi$. Following Matsubara and Feynman, the emergence of macroscopic cycles in $\pi$ at high density $\rho$ has been related to the phenomenon of Bose-Einstein condensation. For each dimension $d\ge 1$, we identify sub-critical, critical and super-critical regimes for $\rho$ and find the limiting distribution of cycle lengths in these regimes. The results extend the work of Betz and Ueltschi. Our main technical tools are saddle-point and singularity analysis of suitable generating functions following the analysis by Bogachev and Zeindler of a related surrogate-spatial model.
\end{abstract}
\maketitle

\section{Introduction}\label{sec:introduction}
We study the model of spatial random permutations in Euclidean space, initially proposed by Matsubara \cite{matsubara1951quantum} and Feynman \cite{feynman1953atomic} in relation to the phenomenon of Bose-Einstein condensation. The model's background and its relations with other models of statistical physics are discussed in Section~\ref{sec:background} below. We proceed here to define the model and describe our main results.

Let $d\ge 1, N\ge 1$ be integers and $L>0, \theta>0$ be
real numbers. Let $X$ be a random variable taking values in $\R^d$.
We assume that $X$ is absolutely continuous with respect to Lebesgue
measure with a density which we denote by~$\varphi$. Throughout the
paper we make the following assumptions on $X$ and $\varphi$,
\begin{equation}\label{eq:assumptions_on_phi}
  \mathbb{E}(X) = 0\quad\text{and}\quad\text{$\varphi$ is a Schwartz
  function},
\end{equation}
where we recall that a Schwartz function is a smooth function whose
derivatives (of any order) decay faster than any polynomial (see Section~\ref{sec:preliminaries}).

Define the domain $\Lambda$ by
\begin{equation*}
  \Lambda:=\left\{x = (x_1,\ldots, x_d)\colon 0\le x_j < L\right\} = \left[0,L\right)^d.
\end{equation*}
We `wrap' $X$ in $\Lambda$ to obtain a periodic density function
$\varphi_\Lambda$ defined by

\begin{equation}\label{eq:varphi_Lambda_def}
\varphi_\Lambda (x):=\sum_{k\in\Z^d}\varphi(x + L\cdot k),\quad
x\in\R^d.
\end{equation}
The spatial random permutation model in $\Lambda$ (with periodic boundary conditions) is a joint probability distribution on $N$ particles in $\Lambda$ and a permutation in $\S_N$, the permutation group on $\{1,2,\ldots,N\}$, in which the density of a
pair $(\mathbf{x}, \pi)\in \Lambda^N\times\S_N$, with respect to the Lebesgue
measure on $\Lambda^N$ times the counting measure on $\S_N$, is
proportional to
\begin{equation}\label{eq:model_density}
  \theta^{C(\pi)}\cdot\prod_{i=1}^N \varphi_\Lambda \! \left( x_{\pi(i)}-x_i\right),
\end{equation}
where $\mathbf{x} = (x_1,\ldots, x_N)$ and where we write $C(\pi)$ for the total number of cycles in $\pi$. The model may be ill-defined for small values of $N$, as the expression \eqref{eq:model_density} may equal zero for all pairs $(\mathbf{x}, \pi)\in \Lambda^N\times\S_N$, but our proofs imply that this can only happen for a finite number, depending on $\varphi$, of values of $N$ (e.g., the model is ill-defined for $N=1$ when $\varphi_\Lambda(0)=0$ and ill-defined for $N=2$ when $\varphi_{\Lambda}(x)\varphi_{\Lambda}(-x)\equiv 0$; see also Proposition~\ref{prop:product of weights}). As we shall see, many properties of the model are governed by the \emph{density} of particles in $\Lambda$, given by
\begin{equation}\label{eq:def of rho}
  \rho:=\frac{N}{L^d}.
\end{equation}

Our main object of study is the limiting distribution of the cycle lengths in
$\pi$, when the pair $(\mathbf{x},\pi)$ is sampled from the density \eqref{eq:model_density} and $N$ and $L$ are taken to infinity. We will see that several asymptotic regimes arise according to the dimension $d$ and the limiting behavior of the density $\rho$. We consider separately the asymptotic behavior when $d=1$, $d=2$ and $d\ge 3$. In dimensions $d=1,2$ we allow $\rho$ to change with $N$ whereas in
dimensions $d\ge 3$ we keep it fixed as $N$ increases.

We proceed to introduce the notation required for stating our main theorems.
We order the cycles of $\pi$ according to the minimal index appearing in the cycle, so that the first cycle is the cycle containing $1$, the second cycle (when $C(\pi)\ge 2$) is the cycle containing the minimal $2\le i\le N$ which is not in the first cycle, and so on. Let $L_j(\pi)$, $1\le j\le C(\pi)$, be the length of the $j$'th cycle in this order and set $L_j(\pi)=0$ for $j>C(\pi)$. We let $\left(\ell_j(\pi )\right)_{j=1}^\infty$ be the sequence of cycle lengths $(L_j(\pi))_{j=1}^\infty$ arranged in non-increasing order. We often write $L_j$ and $\ell_j$ instead of $L_j(\pi)$ and $\ell_j(\pi)$ when $\pi$ is clear from the context.

We consider the limit $N\to\infty$, where we allow the density $\rho$ (and hence $L$) to vary with $N$ in a way to be prescribed, and define the limiting fraction of points in macroscopic cycles,
\begin{equation}\label{eq:def of nu}
\nu := \lim _{\epsilon \rightarrow 0} \liminf
_{N\rightarrow \infty } \mathbb E \bigg(\frac{1}{N}\sum _{j=1}^{C(\pi)} L_j(\pi) \cdot \mathds{1}_{\{L_j(\pi) \ge \epsilon N\}}\bigg) = \lim _{\epsilon \rightarrow 0} \liminf
_{N\rightarrow \infty } \mathbb P(L_1(\pi)\ge \epsilon N),
\end{equation}
where the pair $(\mathbf{x},\pi)$ is sampled from the density \eqref{eq:model_density}, we write $\mathds{1}_A$ for the indicator random variable of the event $A$ and where the second equality follows by symmetry (as the distribution of $\pi$ is invariant under relabeling of $\{1,\ldots, N\}$, see Proposition~\ref{prop:product of weights}).

Denote by $\PD(\theta)$ the Poisson-Dirichlet distribution with parameter $\theta$. This distribution is the limiting distribution of $\frac{1}{N}(\ell_1(\pi),\ell_2(\pi),...)$ when $\pi$ has the Ewens distribution with parameter $\theta$ (that is, when $\mathbb{P}(\pi = \pi_0)$ is proportional to $\theta^{C(\pi_0)}$ for $\pi_0\in \S_N$. The case $\theta=1$ corresponds to a uniform permutation); see \cite{feng2010poisson} for further background.

Denote by $\gama(\alpha ,\beta )$ the $\gama $ distribution with shape parameter $\alpha >0$ and rate parameter $\beta >0$, which is supported on $[0,\infty)$ with density ($\Gamma$ stands for the gamma function)
\[\frac{\beta ^\alpha }{\Gamma (\alpha)}r^{\alpha -1}e^{-\beta r},\quad r>0.\]

We write $\var(X)$ for the variance of $X$ in dimension $d=1$, $\cov(X)$
for the covariance matrix of $X$ in dimensions $d\ge 1$ and $\det(A)$ for the determinant of a matrix $A$. Finally, we write $\varphi^{*j}$ to denote the
$j$-fold convolution $\varphi * \varphi * \cdots * \varphi$ (see Section~\ref{sec:preliminaries}).

\begin{figure}[!tbp]
  \centering
  \begin{minipage}[b]{0.3283\textwidth}
    \includegraphics[width=\textwidth]{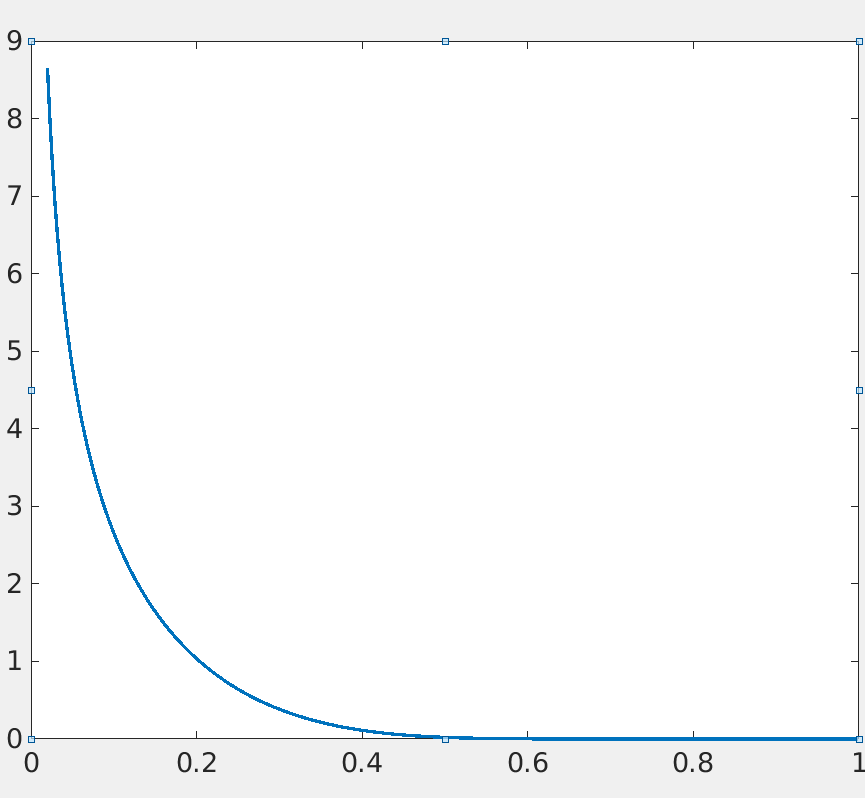}
  \end{minipage}
  \hfill
  \begin{minipage}[b]{0.3283\textwidth}
    \includegraphics[width=\textwidth]{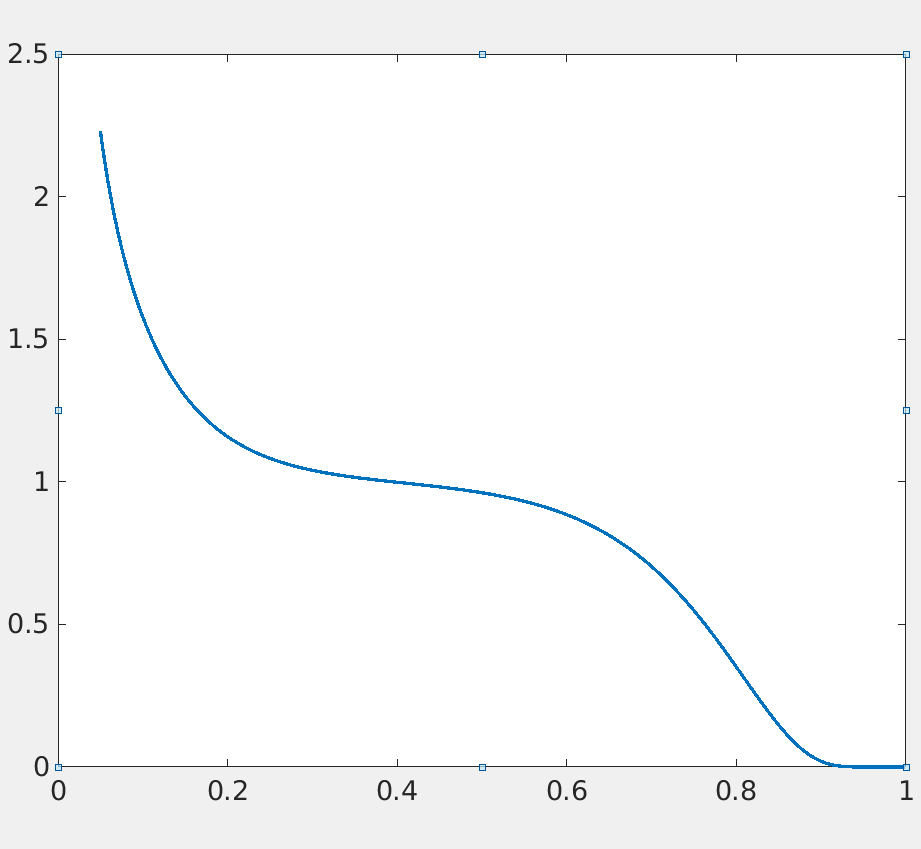}
  \end{minipage}
  \hfill
  \begin{minipage}[b]{0.3283\textwidth}
    \includegraphics[width=\textwidth]{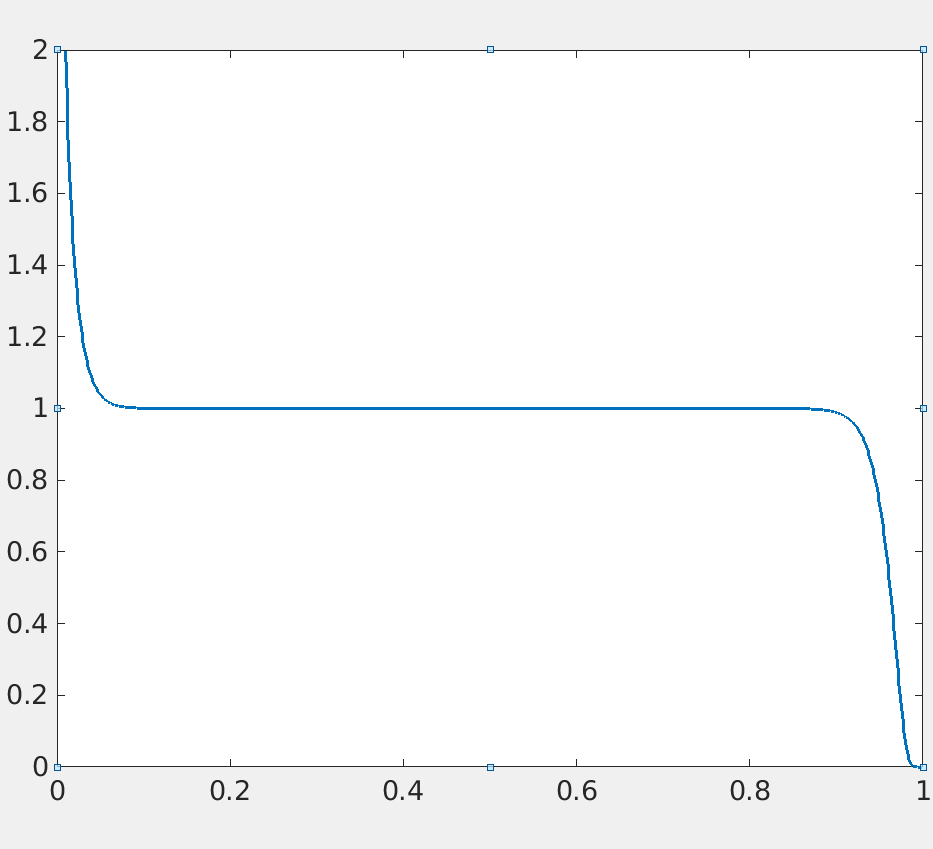}
  \end{minipage}
      \caption{The density function in the 1D critical regime: the limiting density of $L_1/N$ with $\theta =1$ and $\rho /\sqrt{N}\to \alpha $ for $\alpha =0.3$, $\alpha =0.8$ and $\alpha =2$ respectively. By~\ref{item:sub-critical in d=1} and~\ref{item:super-critical in d=1} of Theorem~\ref{thm:d=1}, when $\rho /\sqrt{N}\to 0$ then $L_1 /N \to 0$ and when $\rho / \sqrt{N} \to \infty$ then $L_1/N$ converges in distribution to $U[0,1]$ (see Remark~\ref{remark:connection}). }
\label{critical1Dfigure}
\end{figure}

\begin{thm}[One dimension]\label{thm:d=1}
Let $(\mathbf{x},\pi)$ be randomly sampled from the density
\eqref{eq:model_density} with $d=1$.
Assume that $X$ satisfies the assumptions in \eqref{eq:assumptions_on_phi}. Then, as $N\to \infty$:
\begin{enumerate}[label=(\roman{*})]
\item(Sub-critical I)\label{item:constant in d=1}
Suppose the density $\rho >0$ is fixed as $N\to \infty$. Then $\nu =0$ and
\[L_1\overset{d}{\longrightarrow }Y,\]
where $Y$ is the integer-valued random variable defined by
\begin{equation}\label{eq:Y_thm_def}
  \mathbb P \left(Y=j\right)=\theta \rho ^{-1}\varphi ^{*j}(0) r_*^j,\quad j\ge 1
\end{equation}
and
\begin{equation}\label{eq:r_*_thm_def}
  \text{$r_*$ is the unique number satisfying $0<r_*<1$ and $\sum_{j=1}^\infty \varphi ^{*j}(0) r_*^j=\rho \theta^{-1}$.}
\end{equation}
\item(Sub-critical II)\label{item:sub-critical in d=1}
Suppose the density $\rho$ satisfies $\rho \to \infty$ and $\rho =o(\sqrt{N})$. Then $\nu=0$ and
\[\frac{\theta ^2L_1}{2\var (X)\rho ^2} \overset{d}{\longrightarrow } \gama \left(\frac{1}{2},1\right).\]
\item(Critical)\label{item:critical in d=1}
Suppose that the density $\rho$ satisfies $\frac{\rho}{\sqrt{N}} \to \alpha\in(0,\infty)$. Then $\nu=1$ and
\[\frac{L_1}{N}\overset{d}{\longrightarrow}\mu,\]
where $\mu$ is the probability measure on $\left(0,1\right)$ whose density function is
\begin{equation}\label{eq:complicated distribution}
\frac{1}{Z}\left( \sum _{m\in \Z} e^{-2\pi ^2 \sigma ^2 \alpha ^2 m^2 x}  \right)\left(1-x\right)^{-\frac{3}{2}} \sum _{n=0}^{\infty } (-1)^n {{-2\theta }\choose{n}} (\theta +n)\exp \left(\frac{-(\theta +n)^2}{2 \alpha ^2\sigma ^2 \left(1-x\right)}\right),
\end{equation}
for $x\in (0,1)$, where
\begin{equation*}
Z=\frac{1}{\theta } \sum _{n=0}^{\infty } (-1)^n {{-2\theta }\choose{n}} (\theta +n)\exp \left(\frac{-(\theta +n)^2}{2 \alpha ^2\sigma ^2 }\right), \quad \sigma ^2:=\var (X),
\end{equation*}
and where we set ${t\choose{n}}:=\frac{1}{n!}t(t-1)\cdots(t-n+1)$. See~Figure~\ref{critical1Dfigure}.
\item(Super-critical)\label{item:super-critical in d=1}
Suppose the density $\rho $ satisfies $\rho =\omega (\sqrt{N})$ and $\rho\le N$. Then $\nu =1$ and
\[\frac{1}{N}(\ell_1,\ell_2,...)\overset{d}{\longrightarrow } \PD(\theta ).\]
\end{enumerate}
\end{thm}

\begin{thm}[Two dimensions]\label{thm:d=2}
Let $(\mathbf{x},\pi)$ be randomly sampled from the density
\eqref{eq:model_density} with $d=2$. Let
\begin{equation}\label{eq:def of alpha _c }
\alpha _c:=\frac{\theta }{2\pi \sqrt{\det \left( \cov(X)\right)}}.
\end{equation}
Assume that $X$ satisfies the assumptions in \eqref{eq:assumptions_on_phi}. Then as $N\to \infty$:
\begin{enumerate}[label=(\roman{*})]
\item(Sub-critical I)\label{item:constant in d=2}
Suppose the density $\rho >0$ is fixed as $N\to \infty$. Then $\nu =0$ and
\[L_1\overset{d}{\longrightarrow }Y,\]
where $Y$ is the integer-valued random variable given by \eqref{eq:Y_thm_def} and \eqref{eq:r_*_thm_def}.
\item(Sub-critical II, critical)\label{item:sub-critical and critical in d=2}
Suppose the density $\rho $ satisfies $\rho \rightarrow \infty$ and $\frac{\rho }{\log N}\rightarrow \alpha \in \left[0,\alpha _c\right]$. Then $\nu =0$ and
\[\frac{\alpha _c\log L_1}{\rho }\overset{d}{\longrightarrow } U[0,1].\]
\item(Super-critical)\label{item:super-critical in d=2}
Suppose the density $\rho $ satisfies $\frac{\rho }{\log N}\rightarrow \alpha\in(\alpha_c,\infty]$ and $\rho\le N$. Then $\nu =1-\frac{\alpha _c}{\alpha}$,
\begin{equation*}
  \frac{1}{\nu N}(\ell_1,\ell_2,...)\overset{d}{\longrightarrow } \PD(\theta)\quad\text{and}\quad \frac{\log L_1}{\log N }\overset{d}{\longrightarrow } \mu,
\end{equation*}
where $\mu $ is a probability measure on $[0,1]$ with an atom of mass $1-\frac{\alpha_c}{\alpha}$ at the point $1$ and constant density on $(0,1)$, and we set $\frac{\alpha_c}{\alpha}:=0$ when $\alpha=\infty$.
\end{enumerate}
\end{thm}

\begin{thm}[Dimension $d\ge 3$]\label{thm:d=3}
Let $(\mathbf{x},\pi)$ be randomly sampled from the density
\eqref{eq:model_density} with $d\ge 3$. Suppose that the density $\rho >0$ is fixed as $N\to \infty$. Let
\begin{equation}\label{eq:rho_c_def}
\rho _c:=\theta \sum _{j=1}^\infty  \varphi^{*j}(0).
\end{equation}
Assume that $X$ satisfies the assumptions in \eqref{eq:assumptions_on_phi}. Then as $N\to \infty$:
\begin{enumerate}[label=(\roman{*})]
\item(Sub-critical)\label{item:sub-critical in d=3}
If $\rho <\rho _c $ then $\nu =0$ and
\[L_1\overset{d}{\longrightarrow }Y,\]
where $Y$ is the integer-valued random variable given by \eqref{eq:Y_thm_def} and \eqref{eq:r_*_thm_def}.
\item(Critical)\label{item:critical in d=3}
If $\rho =\rho_c $ then $\nu =0$ and
\[L_1\overset{d}{\longrightarrow }Y,\]
where $Y$ is the integer-valued random variable given by \eqref{eq:Y_thm_def} with $\rho=\rho_c$ and $r_*=1$. Thus,
\begin{equation}\label{eq:asymptotic of Y}
\mathbb P\left(Y=j\right)\sim \frac{\theta}{\rho _c\left(2\pi \right)^{\frac{d}{2}}\sqrt{\det (\cov (X))}}\cdot j^{-\frac{d}{2}},\quad j\to \infty,
\end{equation}
where $\sim$ denotes that the left-hand side is asymptotic to the right-hand side as $j\to\infty$.
\item(Super-critical)\label{item:super-critical in d=3}
If $\rho >\rho _c $ then $\nu =1-\frac{\rho _c}{\rho }$,
\[\frac{1}{\nu N}(\ell_1,\ell_2,...)\overset{d}{\longrightarrow } \PD(\theta )\]
and
\begin{equation}\label{eq:small cycles in supr critical}
\mathbb P (L_1=j)\to \theta \rho ^{-1}\varphi ^{*j}(0),\quad j\ge 1.
\end{equation}
\end{enumerate}
\end{thm}
\subsection{Extensions and remarks} The above theorems identify sub-critical, critical and super-critical regimes of density governing the asymptotic distribution of cycles lengths in $\pi$ in each dimension $d\ge 1$. The limiting distribution of the length of the cycle containing $1$ is determined in all regimes and a Poisson-Dirichlet limit law is proved for the joint distribution of cycle lengths in the super-critical regimes. We make several remarks concerning these statements and additional results which may be deduced with the techniques of this paper:
\begin{itemize}[wide=0pt, leftmargin=*]
  \item Integrability: Our proofs hinge on an integrability property of the model, that the marginal probability
of the permutation $\pi$ has a representation as a product of
weights depending only on the length of the cycles in $\pi$. See Proposition~\ref{prop:product of weights} below for the exact statement. This
fact was also central in previous works on the model (see, e.g.,
\cite[Proposition 3.1]{betz2011spatial}).
  \item Joint distribution of cycle lengths: In the sub-critical cases in all dimensions $d\ge 1$ and the critical cases in dimensions $d\ge 2$, one may extend the above limit laws to apply to the joint distribution of any fixed number of $L_1, L_2, \ldots$, yielding that the cycle lengths become asymptotically independent and identically distributed (see Remark~\ref{remark:asymptotically_IID} and Remark~\ref{remark:asymptotically_IID_critical_d_ge_5}). More precisely, when $\rho$ is fixed as $N\to\infty$, with $\rho\le \rho_c$ if $d\ge 3$, one obtains that for any fixed $k$,
      \begin{equation*}
        (L_1, L_2,\ldots, L_k) \overset{d}{\longrightarrow }(Y_1, Y_2,\ldots, Y_k)
      \end{equation*}
      where the $(Y_j)$ are independent, each having the distribution of the corresponding $Y$ variable in the above theorems. Similarly, in the sub-critical II regime in one dimension it holds that
      \begin{equation*}
        \left(\frac{\theta ^2L_1}{2\var (X)\rho ^2}, \frac{\theta ^2L_2}{2\var (X)\rho ^2},\ldots, \frac{\theta ^2L_k}{2\var (X)\rho ^2}\right) \overset{d}{\longrightarrow } \left(Y_1,Y_2,\ldots, Y_k\right)
      \end{equation*}
      where the $(Y_j)$ are independent, each having the $\gama \left(\frac{1}{2},1\right)$ distribution. The analogous statement holds in the sub-critical II and critical regimes in two dimensions.
  \item Joint distribution of $(\mathbf{x}, \pi)$: Examination of the joint density \eqref{eq:model_density} of $(\mathbf{x}, \pi)$ reveals that, after conditioning on $\pi$, the distribution of $\mathbf{x}$ becomes a collection of independent random walk bridges on the torus $\Lambda$, with jump density $\varphi$ and uniform starting points, whose lengths are the lengths of the cycles in $\pi$. Thus, the marginal distribution of $\pi$ determines the joint distribution of $(\mathbf{x}, \pi)$ in a simple manner.
  \item Number of cycles of given length: Denote by $C_j$ the number of cycles of length $j$ in the random permutation $\pi$. One can relate the moments of $C_j$ and the joint distribution of the $(L_j)$, as discussed in Remark~\ref{remark:connection}. For instance, $\mathbb E\left(C_j/N\right) = \frac{1}{j}\mathbb P(L_1=j)$. Consequently $\mathbb{E}(C(\pi))=N\cdot \E(\frac{1}{L_1})$. Another simple consequence of these relations, deduced by bounding the variance of $C_j$, is that in the regime where $\rho$ and $j$ are fixed as $N\to\infty$, with $\rho\le \rho_c$ if $d\ge 3$, then $C_j/N$ converges in probability to $\mathbb E\left(C_j/N\right)$. We do not study the distribution of the $(C_j)$ further in this work.
  \item Number of cycles: Our techniques provide further information on the distribution of the number of cycles $C(\pi)$. Specifically, they provide access to the asymptotics of $\mathbb{E}(t^{C(\pi)})$ as $N\to\infty$ in the above regimes, for most values of $t>0$. This is further explained in Remark~\ref{remark:generating_function_of_number_of_cycles} but is not developed in this work.
  \item Density in critical regime in one dimension: Probability laws involving the Jacobi theta function have appeared in several works; see \cite{biane2001probability} for a survey. The limiting density \eqref{eq:complicated distribution} obtained in the critical regime in one dimension is of this kind, though we have not seen its exact expression in previous works (a similar expression is in \cite[equation (3.11)]{biane2001probability}).
  \item Schwartz assumption: Our theorems are proved under the assumptions in~\eqref{eq:assumptions_on_phi}. While the assumption that $\mathbb E(X)=0$ is essential to the results, the assumption that $\varphi$ is a Schwartz function may be weakened, requiring that $\varphi$ has sufficiently many derivatives and these decay sufficiently fast. We have not attempted to keep track of the minimal assumptions used in the proof of each result.
  \item The requirement $\rho\le N$: This assumption is imposed in the super-critical parts of Theorem~\ref{thm:d=1} and Theorem~\ref{thm:d=2} but should not be necessary for the results to hold. The assumption is equivalent to requiring that $L\ge 1$ and is made in order to simplify certain technical points in our argument, bearing in mind that most interest is in the case that $L$ tends to infinity with $N$.
\end{itemize}

\subsection{Physics background, previous results and related models}\label{sec:background}
\subsubsection{Background and previous results}
One of the main motivations for studying the spatial random permutation model \eqref{eq:model_density} comes from physics, where it was proposed by Matsubara \cite{matsubara1951quantum} and Feynman \cite{feynman1953atomic} to express quantities arising in the interacting Bose gas as expectation values in a model of random permutations. It was observed that the ideal (non-interacting) Bose gas gives rise to the spatial random permutation model studied in this work with $\theta=1$ and Gaussian $\varphi$ (this is sometimes called the Feynman-Kac representation of the ideal Bose gas). With this link, the phenomenon of Bose-Einstein condensation was related to the appearance of macroscopic cycles in the random permutation. S{\"u}t{\H o} \cite{suto1993percolation, suto2002percolation} further elucidated this link, combining exact calculations and certain approximations, by showing that Bose-Einstein condensation in the ideal Bose gas occurs exactly when macroscopic cycles arise in the corresponding spatial random permutation model and by deriving the limiting distribution of $L_1$ in both the sub-critical and super-critical regimes in dimensions $d\ge 3$.

A mathematical investigation of the spatial random permutation model \eqref{eq:model_density} was conducted by Betz and Ueltschi \cite{betz2009spatial, betz2011spatial}. They studied the model for the class of jump densities $\varphi$ having a non-negative Fourier transform, in particular, having $\varphi(x)=\varphi(-x)$ for all $x$, and satisfying certain additional technical assumptions. Their results include the formula \eqref{eq:rho_c_def} for $\rho_c$, a proof that the fraction of points in macroscopic cycles (equivalently, of `super-constant' size) equals $\min(0,1-\frac{\rho_c}{\rho})$ at each fixed density $\rho$ and the Poisson-Dirichlet limit law for the length of macroscopic cycles in the super-critical regime. Their results apply in dimensions $d\ge 3$ and also in dimensions $d\in\{1,2\}$ when the jump density $\varphi$ is such that $\rho_c$ is finite there (this can occur when $\varphi$ has heavy tails). Their analysis includes the parameter $\theta$ and, in fact, allows more general cycle weights converging to $\theta$.

Our results compare with those of Betz and Ueltschi as follows. Our assumption~\eqref{eq:assumptions_on_phi} yields a class of jump densities $\varphi$ which is neither wider nor narrower than that of \cite{betz2009spatial, betz2011spatial}, as there exist slowly-decaying functions with non-negative Fourier transform (e.g., $(1+|x|)^{-\gamma}$, $1<\gamma<2$, in dimension $d=1$) and there exist Schwartz functions with zero center of mass having complex, or real but sometimes negative, Fourier transform (such as $\exp(-x^4)$ or $\exp(-(x-2)^2) + 2\exp(-(x+1)^2)$ in dimension $d=1$). In the intersection of the two classes, the results of \cite{betz2009spatial, betz2011spatial} apply only in dimensions $d\ge 3$ and yield there the fraction of points in cycles of macroscopic size and the Poisson-Dirichlet limit law. Our work extends the analysis in dimensions $d\ge 3$ by further providing the limiting distribution of cycle lengths for all values of $\rho$. We further analyze the distribution of cycle lengths in dimensions $1$ and $2$.

\subsubsection{Surrogate-spatial model}
Bogachev and Zeindler \cite{bogachev2015asymptotic} studied a related model for random permutations which they term the \emph{surrogate-spatial model}. In this model the probability measure on permutations is given by cycle weights, with the formula
\begin{equation}\label{eq:surrogate_spatial_model}
  \mathbb{P}(\pi = \sigma) = \frac{1}{Z}\prod_{j=1}^N (\theta_j + N\kappa_j)^{C_j(\sigma)},\quad \sigma\in \S_N,
\end{equation}
where $C_j(\sigma)$ is the number of cycles of length $j$ in $\sigma$ and $(\theta_j)$, $(\kappa_j)$ are sequences independent of $N$ for which various behaviors are allowed. The spatial random permutations model \eqref{eq:model_density} is also of this form with the sequence $(\theta_j + N\kappa_j)$ replaced by a weight sequence $(W_{L,j})$ given by~\eqref{eq:W_L_j_def} which takes into account the density $\varphi$ and the geometry; see Proposition~\ref{prop:product of weights}. Very roughly, we may say that the spatial random permutations model corresponds to the choice
\begin{equation}\label{eq:theta_j_and_kappa_j}
\theta_j = \theta\quad\text{and}\quad \kappa_j=\frac{1}{\rho j^{d/2}},
\end{equation}
see Corollary~\ref{cor:corollary on W} and Lemma~\ref{lem:asymptotics of kappa _j}. Indeed, the spatial random permutation model served as the inspiration for the surrogate-spatial model \cite{bogachev2015asymptotic} and the two models are close relatives of each other with the choice \eqref{eq:theta_j_and_kappa_j} when $\rho$ is fixed. The impressive work of Bogachev and Zeindler provides a very detailed analysis of the surrogate-spatial permutation model including identifying sub-critical, critical and super-critical regimes, finding the limiting distributions of cycle lengths and establishing a Poisson-Dirichlet limit law in the super-critical regime. Their proofs rely on saddle point and singularity analysis of suitable generating functions. Our approach is inspired by their analysis, adapting the techniques to the model here and augmenting them with additional tools as necessary.

Let us discuss the main differences between the analysis of~\cite{bogachev2015asymptotic} and the analysis here:
\begin{itemize}[wide=0pt, leftmargin=*]
  \item The generating function: The analysis of~\cite{bogachev2015asymptotic} proceeds via singularity analysis of a generating function $G(z):=\sum_{j=1}^\infty \frac{\theta_j + N\kappa_j}{j}z^j$. This idea also forms the basis of our approach with $(\theta_j + N\kappa_j)$ replaced by the sequence $(W_{L,j})$. The analysis in~\cite{bogachev2015asymptotic} makes several assmptions on the analytic properties of $G$. These include having a singularity on the positive half-line at its radius of convergence, having an analytic continuation to a larger domain and input on the derivatives of $G$ near the singularity. In the model discussed here these properties need to be derived for the generating function defined via the sequence $(W_{L,j})$ and a thorough investigation of $(W_{L,j})$ itself is required. This analysis is performed in Section~\ref{sec:properties} and has no counterpart in~\cite{bogachev2015asymptotic}.
  \item Dimensions $d=1,2$: The analysis of~\cite{bogachev2015asymptotic} is well-suited to the case of a fixed sequence $(\kappa_j)$ as $N$ tends to infinity. In our analysis of dimensions $d=1, 2$ we focus on the cases where $\rho$ tends to infinity with $N$, whence the analogous factor to $\kappa_j$ depends significantly on $N$ (in fact, the analogous factor depends on $N$ also when $\rho$ is fixed, but this dependence is milder leading mostly to technically more complicated arguments). Our analysis in these dimensions is thus more subtle, relying on precise control of the generating function near its singularity at $z=1$. The resulting scaling of the variable $L_1$ and limiting distributions differ from the ones in~\cite{bogachev2015asymptotic}.
  \item Critical cases: The case that $\rho=\rho_c$, the critical point, is analyzed in~\cite{bogachev2015asymptotic} for several choices of $(\theta_j + N\kappa_j)$ via singularity analysis. This method is in principle applicable to the critical spatial random permutations model in dimensions $d\ge 3$. However, as the sequence $(W_{L,j})$ seems to have a more complicated behavior around $j\approx L^2$, which is the region relevant to the critical case, the required estimates appear quite involved in dimensions $d=3,4$. We have thus implemented this approach only in dimensions $d\ge 5$ (which is simplified by having a finite second derivative at $z=1$ for a relevant generating function). For the critical cases in dimensions $d=2,3,4$ we use a seemingly novel approach, making use of recursion relations for the partition function and bootstrapping the estimates proved for the sub-critical cases. The analysis in dimension $d=1$ proceeds via singularity analysis but leads to different expressions than those discussed in~\cite{bogachev2015asymptotic}.
\end{itemize}

\subsubsection{Random permutations with cycle weights}
There has been significant activity \cite{betz2011random,betz2011spatial2,cipriani2015limit,benaych2007cycles,yakymiv2007random,timashov2008random,maples2012number,nikeghbali2013generalized,ercolani2014cycle,lugo2009profiles,barbour2005random,dereich2015cycle} in recent years in studying random permutations in which the probability of a permutation is proportional to a product over its cycles of a weight depending on the length of the cycle. The spatial random permutation model has this general form, see Proposition~\ref{prop:product of weights}, but differs from the models studied in this literature in that the weight assigned to a cycle depends both on the length of the cycle and on the size of the permutation.

\subsubsection{Spatial random permutations}
The model studied in this work, of spatial random permutations in Euclidean space, may be thought of as a particular case of a wider class. Informally, one may think of a spatial random permutation as a random permutation which is biased towards the
identity in some underlying geometry. This broad idea covers
many specific cases including the well-studied interchange model \cite{diaconis1981generating,toth1993improved,angel2003random,schramm2005compositions,berestycki2011emergence,alon2013probability,hammond2013infinite,hammond2015sharp,biskup2015gibbs,kotecky2016random,bjornberg2015large}  and the Mallows model (defined in dimension 1) \cite{mallows1957non,starr2009thermodynamic,borodin2010adding,gnedin2010q, gnedin2012two, mueller2013length,bhatnagar2015lengths,mukherjee2016fixed,gladkich2018cycle, angel2018mallows}. The study of the cycle structure is of great interest in such models as well. In particular, the famous T{\'o}th conjecture~\cite{toth1993improved} states that for $d\ge 3$, infinite orbits arise in the interchange process on $\Z^d$ run for a sufficiently long time, whereas for $d\in\{1,2\}$ no infinite orbits arise after any fixed amount of time. This is easy to establish in dimension $1$ (where further quantitative results are announced by Kozma and Sidoravicius) but remains wide open in dimensions $d\ge 2$ (the $d\ge3$ case is closely related to the long-standing open question of establishing a phase transition for the quantum Heisenberg ferromagnet \cite{toth1993improved}; see also Section~\ref{sec:quantum_Heisenberg_model} below). In this regard, the possibility of \emph{universality} is especially intriguing, that the results of this paper, such as the limiting distributions obtained in dimensions $1$ and $2$ when $\rho$ grows with $N$ or the general features of the sub-critical, critical and super-critical behavior in dimensions $d\ge 3$, are valid also for the other spatial permutation models. We illustrate this possibility with specific analog statements on the interchange and Mallows models, after describing these models briefly.

Let $G = (V,E)$ be a finite or infinite bounded-degree graph. The
\emph{interchange model} \cite{toth1993improved} (also called the \emph{stirring model}
in some of the literature) gives a dynamics on permutations in
$\S_V$, one-to-one and onto functions $\pi\colon V\to V$, which is
associated to the structure of the graph. Each edge of the graph is
endowed with an independent Poisson process of rate $1$. An edge is
said to \emph{ring} at time $t$ if an event of its Poisson process
occurs at that time. Starting from the identity permutation
$\pi^0\in \S_V$, the interchange process, introduced by T\'oth
\cite{toth1993improved}, is the permutation-valued stochastic process $(\pi^t)$
obtained by performing a transposition along each edge at each time
that it rings. The interchange model on the subset $\{0,\ldots, M-1\}^d\subset\Z^d$ equipped with periodic boundary conditions bears formal similarity with the model of Euclidean random permutations studied in this work, under a suitable translation of the parameters. In the interchange model, the location of a particle at time $t\ll M^2$ has roughly a centered Gaussian distribution with covariance matrix $2t$ times the identity. The two models are thus similar after rescaling space by a factor of $\sqrt{t}$, taking $L = \frac{M}{\sqrt{t}}$ and $N = M^d$, so that $\rho=\frac{N}{L^d}=t^{d/2}$.

The Mallows distribution \cite{mallows1957non} on $\S_n$ with parameter $q > 0$ gives probability proportional to $q^{\mathop{inv}(\pi)}$ to each permutation $\pi$, where $\mathop{inv}(\pi) := |{(s, t)\,\colon\,s < t\text{ and }\pi_s > \pi_t}|$ counts the number of inversions in $\pi$. For $0<q<1$, the model may also be defined on the set of all integers~$\Z$~\cite{gnedin2012two}. Focusing on the case $0<q<1$, an important feature of the Mallows model is the typical displacement of points in the permutation, with $|\pi(s) - s|$ typically being of order $\min(\frac{q}{1-q}, n)$ \cite{braverman2009sorting, gnedin2012two, bhatnagar2015lengths}. Similarly to the interchange model, this suggests a similarity of the Mallows model with the Euclidean random permutation model in one dimension, when the parameters are chosen as $L=n/\min\left(\frac{q}{1-q}, n\right)$ and $N = n$, so that $\rho=\min(\frac{q}{1-q}, n)$. The threshold for the emergence of macroscopic cycles, and the Poisson-Dirichlet distribution, in the Mallows model was found in \cite{gladkich2018cycle} and is in qualitative correspondence with the results of Theorem~\ref{thm:d=1}.

The following statements are analogous to the sub-critical results of Theorem~\ref{thm:d=1} and Theorem~\ref{thm:d=2} (in the infinite-volume limit):
\begin{itemize}
  \item Let $L_1(t)$ be the length of the cycle containing the origin in the interchange process on $\Z ^d$ at time $t$. Then, there are constants $c_1,c_2>0$ such that when $d=1$,
\begin{equation}\label{eq:Interchange_universality_1d}
\frac{L_1(t)}{c_1\cdot t}\overset{d}{\longrightarrow } \gama \left(\frac{1}{2},1\right),\quad t\to \infty,
\end{equation}
and when $d=2$,
\begin{equation}\label{eq:Interchange_universality_2d}
\frac{\log L_1(t)}{c_2 \cdot t}\overset{d}{\longrightarrow } U[0,1],\quad t\to \infty.
\end{equation}
  \item Let $0<q<1$ and consider the Mallows model on $\Z$ \cite{gnedin2012two}. Let $L_1(q)$ be the length of the cycle containing $0$. Then there is a constant $c>0$ such that
\begin{equation}\label{eq:Mallows_universality}
c\cdot L_1(q) (1-q)^2\overset{d}{\longrightarrow } \gama \left(\frac{1}{2},1\right)\quad \text{as }q\uparrow 1.
\end{equation}
\end{itemize}
However, analytical calculations of Angel and Hutchcroft~\cite{Angel2018privatecommunication} imply that the statement~\eqref{eq:Mallows_universality} does not hold. This is also apparent in simulations, see~Figure~\ref{mallowsfigure}. On the positive side, the simulations seem to indicate that the limit in \eqref{eq:Mallows_universality} should at least exist and share some general features with the gamma distribution. Given that the Mallows statement~\eqref{eq:Mallows_universality} does not hold, we also do not expect the corresponding statement for the 1D interchange model to hold exactly as in~\eqref{eq:Interchange_universality_1d}, though, again, the limit there is still likely to exist and bear similarities with the gamma distribution. The situation for the 2D statement~\eqref{eq:Interchange_universality_2d} is less clear. One may be more hopeful that it holds as written, as the presence of the logarithm on the left-hand side renders the statement more robust to small deviations in the distribution of $L_1(t)$ itself.

\begin{figure}[!tbp]
  \centering
  \begin{minipage}[b]{0.49635\textwidth}
    \includegraphics[width=\textwidth]{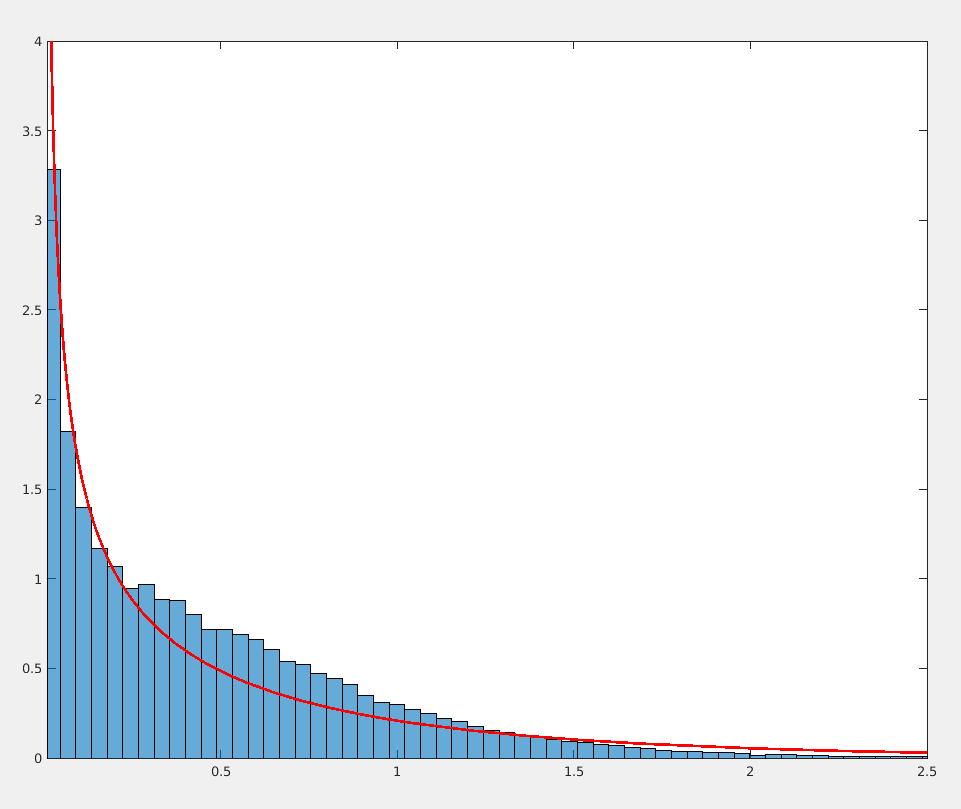}
  \end{minipage}
  \hfill
  \begin{minipage}[b]{0.49635\textwidth}
    \includegraphics[width=\textwidth]{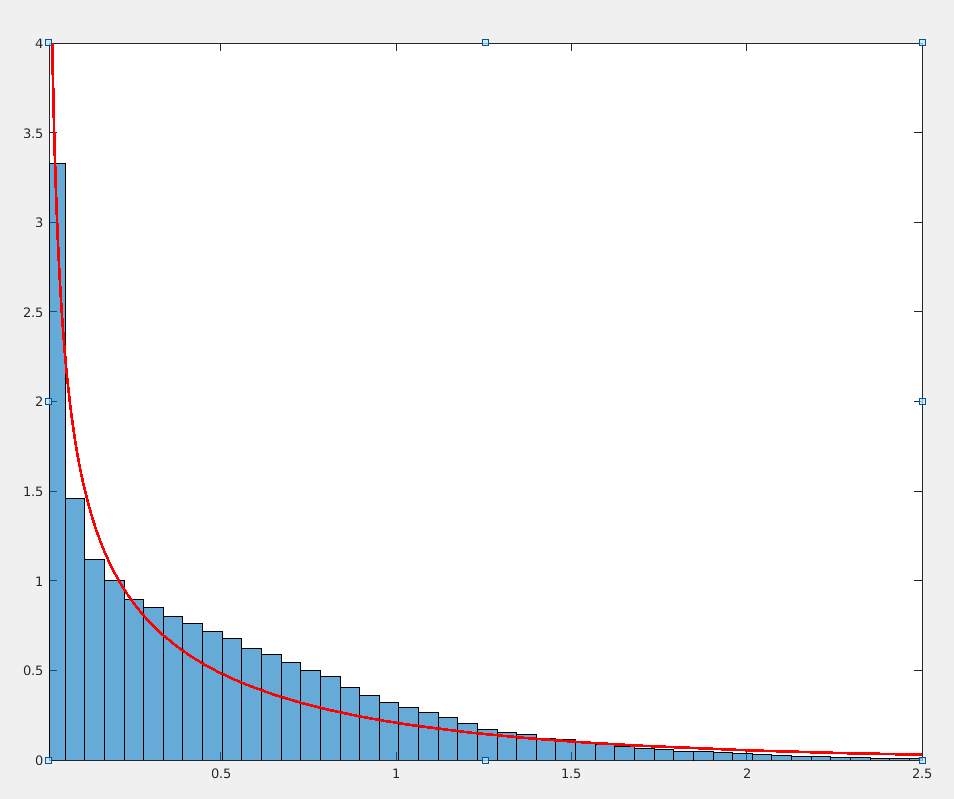}
  \end{minipage}
      \caption{The empirical density of $0.3\cdot L_1(q)(1-q)^2$ as obtained from the simulation (in blue) for $q=0.9$ and $q=0.975$ respectively, and the density of $\gama(\frac{1}{2},1)$ distribution (in red). The simulations were sampled using the algorithm described in Section~3.5 of \cite{gladkich2018cycle} with 100000 iterations.}
\label{mallowsfigure}
\end{figure}

\subsubsection{Quantum Heisenberg model in two dimensions}\label{sec:quantum_Heisenberg_model}
Continuing further with the analogy to the interchange process, one may also speculate that the relation~\eqref{eq:Interchange_universality_2d} in $d=2$ continues to hold also for the interchange model which is tilted by a factor of $2$ to the power of the number of cycles (in analogy with setting $\theta=2$ in for the model studied in this work). It is known from the work of T{\'o}th \cite{toth1993improved} that this tilted model is in direct correspondence with the quantum Heisenberg model, at the temperature $T$ which is the inverse of the time $t$ to which the interchange model is run. Precisely, the spin-spin correlation between sites $x$ and $y$ in the quantum Heisenberg model equals a constant times the probability that $x$ and $y$ are in the same cycle \cite{toth1993improved} (see also \cite{ueltschi2013random}). Thus, the length of the cycle containing a given site $x$ equals the sum of spin-spin correlations between $x$ and the other sites. A relation of the form~\eqref{eq:Interchange_universality_2d} (accompanied with appropriate integrability conditions) would then imply that the sum of spin-spin correlations is of the order $\exp(\frac{c}{T})$ for some constant~$c>0$. This is in accordance with the predicted behavior of the Heisenberg ferromagnet, which states that spin-spin correlations decay exponentially at every positive temperature, with a correlation length which is exponential in the inverse temperature. Such predictions, going back to Polyakov~\cite{Pol75}, remain wide open in the mathematical literature.

\subsubsection{Band matrices}
We mention that the results regarding the emergence of macroscopic cycles in one dimension bear formal similarity with a conjectured localization transition for random band matrices. This similarity is detailed in \cite[Section 1.2.2]{gladkich2018cycle} in the context of the Mallows measure on permutations. One may also define random band matrices in dimensions $d\ge 2$ where they are rather poorly understood. Indeed, even taking the bandwidth to be one (when the band matrix corresponds to a random Schr\"odinger operator) leads to many unsolved problems around the famous Anderson localization phenomenon (foremost among these is the question of delocalization of random Schr\"odinger operators at low disorder in dimensions $d\ge 3$. Also conjectured but unresolved is the fact that two-dimensional random band matrices exhibit localization for any fixed bandwidth). Does the similarity between the emergence of macroscopic cycles and the localization properties of random band matrices extend to dimensions $d\ge 2$? If true, such similarity would predict that in two dimensions, the critical bandwidth for delocalization in a random band matrix on the box $[-L,L]^2$ is of order $\sqrt{\log L}$.

\subsection{Acknowledgements}
We thank Daniel Ueltschi for helpful advice and for referring us to the paper of Bogachev and Zeindler \cite{bogachev2015asymptotic}. We thank Volker Betz, Gady Kozma, Mikhail Sodin and Elad Zelingher for useful discussions. We thank Xiaolin Zeng for helpful comments on an earlier version of this work. We thank Omer Angel and Tom Hutchcroft for considering the validity of the statement~\eqref{eq:Mallows_universality} on the Mallows model and letting us know the conclusion of their calculations. We are grateful to two anonymous referees whose thoughtful comments helped to elucidate the presentation of the results.

\section{Preliminaries}\label{sec:preliminaries}

{\bf Constant policy:} Throughout the paper we regard the dimension $d\ge 1$, the density function $\varphi$ (satisfying the assumptions~\eqref{eq:assumptions_on_phi}) and the real $\theta>0$ as fixed and our emphasis is on the behavior of the various quantities of interest as the parameters $N$ and $L$ (or the density $\rho=\frac{N}{L^d}$) change. Constants such as $C, c, \epsilon, \delta$ denote positive numerical values which may depend on $d, \varphi$ and $\theta$ but are independent of all other parameters (in particular, of $N$ and $L$). When the constant depends on an additional parameter this will be noted explicitly, writing, for instance, $C_n$ for a value which depends on $d, \varphi,\theta$ and $n$. The constants $C,c$, or their counterparts depending on additional parameters, are regarded as generic constants in the sense that their value may change from one appearance to the next, with the value of $C$ increasing and the value of $c$ decreasing. However, constants labeled with a fixed number, such as $C_0, c_0$, have a fixed value throughout the section that they appear in.

{\bf Oh notation:} For two functions $f,g$, possibly depending on many parameters, we write $f = O(g)$ to denote that $\frac{|f|}{|g|}\le C$ where $C$ is as above, that is, independent of all parameters besides $d, \varphi$ and $\theta$. We use a similar notation when the constant may depend on additional parameters writing, for instance, $f = O_n(g)$ to denote that $\frac{|f|}{|g|}\le C_n$. We write $f = o(g)$ as $k\to k_0$ (where $k_0$ may be infinity) to denote that $\lim_{k\to k_0} \frac{|f|}{|g|}=0$. If we write, in addition, that the little Oh is uniform in $m\in I_k$ (where $I_k$ may or may not depend on $k$) we mean that $\lim_{k\to k_0} \sup_{m\in I_k} \frac{|f|}{|g|}=0$. The notation $f\sim g$ as $k\to k_0$ means that $f = (1+o(1))g$ as $k\to k_0$ and we may again add a uniformity requirement.

{\bf Notation:} We write $\mathbb{N}:=\{1,2,3,\ldots\}$ for the set of positive integers.

We use the Pochhammer symbol $(x)_n$ defined by
\[(x)_n:=x\left(x-1\right)\cdots \left(x-n+1\right),\quad x\in \R,\, n\in \mathbb N\]
and $(x)_0:=1$.

A smooth function $f:\R^d \to \mathbb C$ is called
\emph{Schwartz} if
\[\underset{x\in \R ^d}{\sup }\left|x_1^{\alpha _1}\cdots x_d^{\alpha _d} \cdot \frac{\partial ^{\beta _1+\dots +\beta _d}f}{\partial x_1^{\beta _1}\cdots \partial x_d^{\beta _d}}\right|<\infty\quad\text{for any integers $\alpha _1,\dots ,\alpha
_d,\beta _1,\dots \beta _d\ge 0$}.\]
We use the shorthand $f^{(n)}$ to denote the $n$-th derivative of a function $f$ and write $\left[z^n\right]f(z)$ to denote the coefficient of $z^n$ in the power series of $f(z)$. The convolution $f * g$ of integrable $f,g:\R^d\to\mathbb C$ is defined, as usual, by $(f * g)(x):=\int f(y)g(x-y)dy$. We write $f^{*j}$ to denote the $j$-fold convolution $f * f * \cdots * f$.

We denote the closure of a set $\Omega \subseteq \C$ by $\overline{\Omega }$ and let $\mathbb {D}$ denote the open unit disc,
\[\mathbb{D}:=\{z\in \C \ :\ |z|< 1 \}.\]

We use the standard branches of the argument, logarithm and power functions on the complex plane. That is, we take $\arg(z)\in (-\pi ,\pi ]$ for $z\in\mathbb C$ and consider $\log z:=\log|z| + i\arg(z)$ and $z^\alpha:=e^{\alpha\log z}$ for $z\in\mathbb C\setminus (-\infty,0]$.

We write $\|v\|$ for the Euclidean norm of a vector $v\in\R^d$. Such vectors are thought of as column vectors for purposes of matrix multiplication and we write $v^T$ for the row vector obtained after transposition. Similarly, for a matrix $B$ we write $B^T$ for the transposed matrix.

We write $\mathds{1}_A$ for the indicator random variable of an event $A$.

We denote by $N\left(\mu, \Sigma \right) $ the multivariate normal distribution with mean vector $\mu$ and covariance matrix $\Sigma $.

In several places in the paper it is convenient to discuss the `square root' matrix of the covariance matrix $\cov(X)$ of $X$. To this end we point out that $\cov(X)$ is (symmetric) positive definite as $X$ has a density. In the sequel we
\begin{equation}\label{eq:def of A}
  \text{fix $A$ to be a (symmetric) positive definite matrix satisfying $A^2 = \cov(X)$}.
\end{equation}

{\bf Hierarchy:} The environments Theorem, Proposition, Lemma and Claim all have the same formal meaning in our paper. Informally, we have tried to use the title to indicate a level in the hierarchy - propositions are used to prove theorems, lemmas are used to prove propositions, etc.

\section{Exact expression for the distribution of $\pi$}\label{sec:combinatorics}

In this section, following Betz and Ueltschi \cite{betz2011spatial} and Bogachev and Zeindler \cite{bogachev2015asymptotic}, we start by proving that the marginal probability of the permutation in the spatial random permutation model has a representation as a product of cycle weights. We then find a convenient generating function for the partition functions arising in this representation and express the main statistics of interest to us (such as the distribution of $L_1$) in terms of the partition functions.

We remind that the integer $d\ge 1$, density function $\varphi$ (satisfying the assumptions~\eqref{eq:assumptions_on_phi}) and real $\theta>0$ are fixed. Throughout the section we fix also the integer $N\ge 1$ and real $L>0$ and let $(\mathbf{x},\pi )$ be randomly sampled from the density \eqref{eq:model_density}.

\subsection{Marginal distribution of $\pi $ and generating function}\label{sec:convenient generating function}
We again denote by $C_j(\sigma )$ the number of cycles of length $j$ of a permutation $\sigma$.
\begin{prop}\label{prop:product of weights}
The marginal distribution of the permutation $\pi$ is given by
\begin{equation}\label{eq:probability of a permutation}
\mathbb P(\pi=\sigma )=\frac{1}{N!H_N(L)}\prod_{C}W_{L,|C|}=\frac{1}{N!H_N(L)}\prod _{j=1}^N\left(W_{L,j}\right)^{C_j(\sigma
)},\quad \sigma \in \S_N
\end{equation}
where the first product runs over all cycles $C$ of $\sigma$, $|C|$ denotes the length of the cycle $C$,
\begin{equation}\label{eq:W_L_j_def}
W_{L,j}:=\theta L^d\sum _{k\in \Z^d}\varphi^{*j}(Lk)
\end{equation}
and
\begin{equation}\label{eq:def of partition function}
H_N(L):=\frac{1}{N!}\sum _{\sigma  \in \S_N} \prod _{j=1}^N\left(W_{L,j}\right)^{C_j(\sigma )}
\end{equation}
is the partition function.
\end{prop}
We include the factor $N!$ in \eqref{eq:probability of a
permutation} for consistency with the notation in \cite{betz2011spatial}
and as it simplifies some of the resulting generating function
formulas below (e.g., equality \eqref{eq:generating_h(v)}).

In the proof, we write $x\pmod{L\Z^d}$, where $x\in\R^d$, to denote
the unique point $y\in\Lambda$ satisfying that $x-y \in L\Z^d$.  We also require the notion of the $j$-fold convolution of $\varphi_\Lambda$ with itself \emph{on the torus $\Lambda$}. Define $\varphi_\Lambda^{\circledast 1}:=\varphi_\Lambda $ and set, inductively, for $j\ge 1$,
\begin{equation}\label{eq:def of convolution on the torus}
\varphi_\Lambda^{\circledast (j+1)}(y):=\intop _\Lambda \varphi_\Lambda^{\circledast j}(x)\varphi_\Lambda(y-x)dx,\quad y\in \R ^d.
\end{equation}

\begin{claim}\label{claim: torus convolution}
For any $j\ge 1$ we have
\begin{equation}\label{eq:convolution_equality}
\varphi_\Lambda ^{\circledast j}(y)=\sum_{k\in \Z
^d}\varphi^{*j}(y+Lk), \quad y\in\R ^d.
\end{equation}
Equivalently, using the notation $(\cdot )_ \Lambda $ as in \eqref{eq:varphi_Lambda_def},
\begin{equation}\label{eq:convolution and projection}
(\varphi_\Lambda ) ^{\circledast j}= (\varphi ^{*j})_{\Lambda}.
\end{equation}
\end{claim}

\begin{proof}
Intuitively, the identity \eqref{eq:convolution and projection} can be understood as follows, the left-hand side is the density of
the sum on the torus $\Lambda$ of $j$ independent copies of
$X\pmod{L\Z ^d}$, whereas the right-hand side is the density of
the projection to the torus of the sum in $\R^d$ of $j$ independent copies of $X$. We turn to prove \eqref{eq:convolution_equality} formally by induction. The case $j=1$ is exactly the definition of $\varphi_\Lambda$. Suppose \eqref{eq:convolution_equality} holds for some $j\ge 1$. Then
\begin{equation*}
\begin{split}
\varphi_\Lambda^{\circledast (j+1)}(y)&=\intop _\Lambda \varphi_\Lambda^{\circledast j}(x)\varphi_\Lambda(y-x)dx\\
&=\intop _\Lambda \sum_{k_1,k_2\in \Z^d}\varphi^{*j}(x+Lk_1)\varphi(y-x+Lk_2)dx\\
&=\sum_{k,k_1\in \Z^d}\intop _\Lambda \varphi^{*j}(x+Lk_1)\varphi(y-x+Lk-Lk_1)dx\\
&=\sum_{k\in \Z^d}\intop _{\R ^d}\varphi^{*j}(x)\varphi(y+Lk-x)dx=\sum_{k\in \Z^d}\varphi^{*(j+1)}(y+Lk). \qedhere
\end{split}
\end{equation*}
\end{proof}

\begin{proof}[Proof of Proposition~\ref{prop:product of weights}]
The marginal probability on permutations is
\begin{equation}
\mathbb P (\pi =\sigma )=\frac{\theta^{C(\sigma )}}{Z_N}\intop_{\Lambda
^N}\prod_{i=1}^N \varphi_\Lambda \left(
x_{\sigma (i)}-x_i\right)dx_1\cdots dx_N,\quad \sigma \in \S_N,
\end{equation}
where $Z_N$ is the appropriate normalization factor. It is
straightforward that the integral factorizes according to the cycles
in the permutation $\pi$. The contribution of each fixed point equals
\[\theta \intop _\Lambda \varphi_\Lambda (0)dy=\theta |\Lambda |\sum _{k\in \Z^d}\varphi (Lk)=W_{L,1}.\]
The contribution of each cycle $\left(y_1,\dots ,y_j\right)$ of length
$j\ge 2$ is (with the convention that $y_{j+1}$ is $y_1$)
\begin{equation}\label{eq:contribution_of_length_j_cycle}
\theta \intop _{\Lambda ^j}\prod_{i=1}^j \varphi_\Lambda \left(
y_{i+1}-y_i\right)dy_1\cdots dy_j=\theta \intop _{\Lambda }
\left(\;\,\intop _{\Lambda ^{j-1}}\prod_{i=1}^j \varphi_\Lambda
\left( y_{i+1}-y_i\right)dy_2\cdots dy_j\right)dy_1.
\end{equation}
By considering the change of variables $y_1\mapsto y_1$ and $y_i\mapsto y_i+y_1
\pmod{L\Z^d}$ for $2\le i\le j$ and using the the fact that $\varphi_\Lambda $ is periodic, we see that the inner integral does
not depend on the value of $y_1$.
Thus, the expression \eqref{eq:contribution_of_length_j_cycle}
becomes
\begin{equation}\label{eq:convolution}
\begin{split}
 \theta\left|\Lambda \right|\intop _{\Lambda ^{j-1}}\varphi_\Lambda (y_2) \varphi_\Lambda \left( y_3-y_2\right)\cdots\varphi_\Lambda &\left( y_j-y_{j-1}\right)\varphi_\Lambda (-y_j)\,dy_2\cdots dy_j\\
&=\theta \left|\Lambda \right|\varphi_\Lambda^{\circledast j}(0)=\theta L^d \sum_{k\in \Z ^d}\varphi^{*j}(Lk)=W_{L,j},
\end{split}
\end{equation}
where in the first equality we used \eqref{eq:def of convolution on the torus} and in the second equality we appealed to Claim~\ref{claim: torus convolution}.
\end{proof}

The next well-known identity
is a special case of the enumeration theorem of \emph{P\'olya}.
\begin{lemma}\label{lem:cycle_index_theorem}
Let $(a_j)_{j\in \mathbb N}$ be a sequence of complex numbers. Then the following formal power series expansion holds
\begin{equation}\label{eq:symm_fkt}
\exp\left(\sum_{j=1}^{\infty}\frac{a_jz^j}{j}\right)
=\sum_{n=0}^\infty\frac{z^n}{n!}\sum_{\sigma \in\S _n}\prod_{j=1}^{n}
a_{j}^{C_j(\sigma )}.
\end{equation}
\end{lemma}
\begin{proof}
The proof is a relatively straightforward calculation, presented below with a probabilistic flavor.

We assume that $a_j>0$ and that $\sum _{j=1}^\infty a_j<\infty$. This is without loss of generality since if \eqref{eq:symm_fkt} holds for such sequences then it holds in a formal sense for all sequences. Recall that a random variable $Y$ has the $\poisson(\lambda)$ distribution if $\mathbb{P}(Y=n)=\exp(-\lambda)\frac{\lambda^n}{n!}$ for integers $n\ge 0$. Note that it has the probability generating function $\mathbb{E}(z^Y) = e^{\lambda (z-1)}$. Define an infinite sequence of independent Poisson random variables,
\[X_j\sim \poisson\left(\frac{a_j}{j}\right),\quad j\ge 1\]
and set
\begin{equation*}
  X := \sum_{j=1}^\infty j X_j.
\end{equation*}
Note that $X$ is almost surely finite as it has finite expectation by our assumption on $(a_j)$. The equality \eqref{eq:symm_fkt} follows by calculating $\mathbb{E}(z^X)$ in two ways. On the one hand,
\[\mathbb{E}(z^X)=\mathbb{E}\left(\prod _{j=1}^\infty z^{jX_j}\right)=\prod _{j=1}^\infty \mathbb{E}(z^{jX_j})=\prod _{j=1}^\infty \exp \left(\frac{a_j\left(z^j-1\right)}{j}\right)=\exp \left(\sum _{j=1}^\infty \frac{a_j\left(z^j-1\right)}{j}\right),\]
On the other hand,
\begin{equation*}
\begin{split}
\mathbb{E}(z^X)&=\sum _{n=0}^\infty z^n\cdot \mathbb P \left(X=n\right)=\sum _{n=0}^\infty z^n \sum _{\substack{ \left(c_1,c_2,\dots \right): \\ \sum _{j=1}^\infty jc_j=n}} \mathbb P \left(\forall j,\, X_j=c_j\right) \\
&=\exp \left(-\sum _{j=1}^\infty \frac{a_j}{j}\right)\sum _{n=0}^\infty z^n \sum _{\substack{ \left(c_1,\dots c_n\right): \\ \sum _{j=1}^njc_j=n}} \prod _{j=1}^n \frac{a_j^{c_j}}{j^{c_j}c_j!}=\exp \left(-\sum _{j=1}^\infty \frac{a_j}{j}\right)\sum _{n=0}^\infty \frac{z^n}{n!} \sum _{\sigma \in \S _n} \prod _{j=1}^n a_j^{C_j(\sigma )},
\end{split}
\end{equation*}
where the last equality follows as the number of permutations $\sigma \in \S _n$ with cycle structure $\left(c_1,\dots ,c_n\right)$ is $\frac{n!}{\prod_{j=1}^nj^{c_j}c_j !}$.
\end{proof}

Lemma~\ref{lem:cycle_index_theorem} can be used to obtain a
convenient expression for the normalization constant $H_N(L)$ involved
in the marginal probability distribution \eqref{eq:probability of a permutation}. To this end set
$H_0(L):=1$ and define the function
\begin{equation}\label{eq:def of G_L(z)}
G_L(z):=\sum _{j=1}^\infty \frac{W_{L,j}}{j}z^j.
\end{equation}
By Lemma~\ref{lem:cycle_index_theorem} (with $a_j=W_{L,j}$), the generating function of the sequence
$\left(H_n(L)\right)_{n\ge 0}$ is given by
\begin{equation}\label{eq:generating_h(v)}
\sum_{n=0}^\infty H_n(L) z^n =\sum_{n=0}^\infty\frac{z^n}{n!}\sum_{\sigma\in \S _n} \prod_{j=1}^n (W_{L,j})^{C_j(\sigma )}= \exp\left(\sum_{j=1}^\infty
\frac{W_{L,j}}{j}\,z^j\right)=e^{G_L(z)}.
\end{equation}
The equality holds in a formal sense, as we relied on Lemma~\ref{lem:cycle_index_theorem}. However, one may check that the non-negative sequence $(W_{L,j})_j$ is bounded, whence the sum in \eqref{eq:def of G_L(z)} converges for $|z|<1$, implying the same for the Taylor series \eqref{eq:generating_h(v)} of its exponential. The boundedness of $(W_{L,j})_j$ follows from its definition \eqref{eq:W_L_j_def} as the density of a convolution on the torus (see \eqref{eq:def of convolution on the torus} and \eqref{eq:convolution_equality}) and is also a consequence of our subsequent Corollary~\ref{cor:corollary on W}.

\subsection{The distribution of the cycle lengths through the generating function}\label{sec: the disribution of C_j and L-1}
Recall the definitions of $L_1,L_2,\dots $ appearing before \eqref{eq:def of nu}.

\begin{lemma}\label{lem:distribution_of_ell_1,ell_2}
For any integers $m\ge 1$ and $j_1,\dots,j_m\ge 1$ with $j_1+\dots +j_m\le N$ we have
\begin{equation}\label{eq:L=l}
\mathbb P\left(L_1 = j_1,\dots, L_m = j_m\right)=
\frac{H_{N-j_1-\dots-j_m}(L)}{H_N(L)}\cdot \prod_{k=1}^m
\frac{W_{L,j_k}}{N-
j_1-\dots-j_{k-1}}.
\end{equation}
In particular, for $m=1$,
\begin{equation}\label{eq:dist_L1}
\mathbb P\left(L_1 = j\right)= \frac{W_{L,j}}{N} \cdot
\frac{H_{N-j}(L)}{H_N(L)},\quad 1\le j\le N.
\end{equation}
\end{lemma}

\begin{proof} Note that there are $(N-1)\cdots(N-j+1)=(N-1)_{j-1}$
possibilities for a cycle of length $j$ containing the element~$1$. Thus, using \eqref{eq:probability of a permutation} and \eqref{eq:def of partition function}, we get
\begin{equation*}
\mathbb P\left(L_1 = j\right) = \frac{  (N-1)_{j-1} \cdot W_{L,j}}{N! H_N(L)}\sum _{\sigma  \in \S_{N-j}} \prod _{j=1}^{N-j}\left(W_{L,j}\right)^{C_j(\sigma )}= \frac{W_{L,j}}{N} \cdot
\frac{H_{N-j}(L)}{H_N(L)},
\end{equation*}
which proves the lemma for $m=1$.
Similarly, for $m=2$,
\begin{equation}\label{eq:N-L1}
\begin{split}
\mathbb P & \left(L_1 = j_1,L_2=j_2\right) \\
&=\frac{(N-1)_{j_1-1} \cdot W_{L,j_1}\cdot (N-j_1-1)_{j_2-1}
W_{L,j_2} \cdot  (N-j_1-j_2)!\cdot H_{N-j_1-j_2}(L)}{N! H_N(L)}\\
& = \frac{W_{L,j_1}W_{L,j_2}}{N(N-j_1)} \cdot
\frac{H_{N-j_1-j_2}(L)}{H_N(L)}.
\end{split}
\end{equation}
The general case $m\ge 1$ is handled in the
same manner.
\end{proof}

\begin{remark}\label{remark:connection}
The relation between the distribution of $L_1$ and the $(C_j)$ is as follows,
\begin{equation}\label{eq:C_j_expectation}
  \mathbb P \left(L_1=j\right)=\mathbb E \left( \frac{jC_j}{N}\right),\quad 1\le j\le N,
\end{equation}
since
\begin{equation*}
\mathbb P \left(L_1=j\right) =\sum_{\substack{k\ge 0 \mbox{ s.t. }\\ \mathbb P \left(C_j=k\right)>0}} \mathbb P \left(L_1=j\, \mid \, C_j=k\right) \cdot \mathbb P \left(C_j=k\right)
=\sum_{k=0}^\infty \frac{jk}{N} \cdot \mathbb P \left( C_j=k\right)=\mathbb E \left( \frac{jC_j}{N}\right),
\end{equation*}
where the second equality follows as there are $jk$ elements in cycles of length $j$ and any $1\le i\le N$ is equally likely to be one of them. In a similar manner one checks that
\begin{equation}\label{eq:C_j_squared_expectation}
  \mathbb{P}(L_1 = j, L_2 = j) = \mathbb E\left(\frac{j C_j}{N}\cdot\frac{j(C_j - 1)}{N-j}\right),\quad 1\le j\le N.
\end{equation}
One may similarly obtain expressions involving higher moments of $C_j$. A useful expression for the variance of $C_j$ is obtained by combining \eqref{eq:C_j_expectation} and \eqref{eq:C_j_squared_expectation},
\begin{equation*}
  \var\left(\frac{C_j}{N}\right) = \frac{N-j}{N}\cdot\frac{\mathbb{P}(L_1 = j, L_2 = j)}{j^2} + \frac{\mathbb P(L_1=j)}{Nj} - \left(\frac{\mathbb P(L_1 = j)}{j}\right)^2.
\end{equation*}
\end{remark}

\begin{remark} \label{remark:generating_function_of_number_of_cycles}
Our analysis of the asymptotic distribution of cycle lengths proceeds by determining the asymptotics of the partition function $H_n(L)$, in various asymptotic regimes of $n$ and $L$, and then applying Lemma~\ref{lem:distribution_of_ell_1,ell_2}. Further information on the random permutation $\pi$ may be obtained by varying the various parameters it depends on. For instance, the number of cycles $C(\pi)$ may be studied as follows. Let us write $H_n^{\theta}(L)$ instead of $H_n(L)$ to note explicitly the dependence of the partition function on $\theta$. Then,
\begin{equation*}
  \E(t^{C(\pi)}) = \frac{H_N^{t\cdot\theta}(L)}{H_N^{\theta}(L)}
\end{equation*}
as one immediately verifies using Proposition~\ref{prop:product of weights}. Thus, the asymptotic behavior of the partition function can be used to determine the probability generating function of the number of cycles. This direction is not developed in this work.
\end{remark}

\section{Basic properties of the generating function}\label{sec:properties}

In this section we state and prove some basic properties of the generating function $G_L$.

\subsection{Fourier transform} We frequently use the Fourier transform $\hat{f}$ (or $\mathcal{F}(f)$) of a function $f:\R ^d\to \mathbb C$, defined by
\begin{equation}\label{eq:fourier transform}
\hat{f}(t)=\mathcal{F}(f)(t):=\intop _{\R ^d}f(x)e^{-2\pi ix\cdot t}dx, \quad t\in \R ^d,
\end{equation}
where we write $x\cdot t$ for the standard scalar product in $\R ^d$.
In the following claims we collect some basic facts about the Fourier transform of a Schwartz function.

\begin{claim}\label{claim:basic facts on the fourier transform}
Let $f:\R ^d \to \mathbb C$ be a Schwartz function. Then:
\begin{enumerate}[label=(\roman{*})]
\item\label{item:1 in basic properties}
For any invertible linear transformation $B:\R^d\to \R ^d$,
\begin{equation}\label{eq:change variables in transform}
\mathcal{F}(f\circ B)=\frac{1}{\left|\det B\right|} \cdot \hat{f}\circ \left(B^T\right)^{-1}.
\end{equation}

\item\label{item:2 in basic properties}
The Fourier inversion theorem states that $\hat{\hat{f}}(x)=f(-x)$ for all $x\in \R ^d$.
\item\label{item:3 in basic properties}
For any $j\in  \N$, the functions $f^{*j}$ and $\hat{f}^j$ are Schwartz functions and we have $\widehat{f^{*j}}=\hat{f}^j$.
\item\label{item:4 in basic properties}
The Poisson summation formula holds:
\begin{equation}\label{eq:fourier sumation}
\sum _{k\in \Z ^d}f(k)=\sum _{m\in \Z ^d}\hat{f}( m).
\end{equation}
\end{enumerate}
\end{claim}

In the following claim we restrict ourselves to the density function $\varphi$.

\begin{claim}\label{claim:basic fourier 2}
The Fourier transform $\hat{\varphi}$ of $\varphi$ satisfies:
\begin{enumerate}[label=(\roman{*})]
\item\label{item:1 in basic properties2}
For every $t\in \R ^d$, $\left|\hat{\varphi}(t)\right|\le 1$ with equality if and only if $t=0$.
\item\label{item:2 in basic properties2}
We have
\begin{equation}\label{eq:taylor of phi}
\hat {\varphi}(t)=1-2\pi ^2t^T \cov(X) t+O(\| t\|^3)=1-2\pi ^2 \|At\|^2+O(\|t\|^3),\quad \|t\|\le 1
\end{equation}
where $A$ is given in \eqref{eq:def of A}. In particular, as $A$ is positive definite, there is a $c>0$ so that $\left|\hat{\varphi}(t)\right|\le 1-c\|t\|^2$ when $\|t\|\le 1 $.
\end{enumerate}
\end{claim}

The results in Claim~\ref{claim:basic facts on the fourier transform} and Claim~\ref{claim:basic fourier 2} are standard and we refer the reader to \cite{grafakos2004classical} for the proofs.

Throughout this section we denote by $\psi $ the density function of the Gaussian distribution $N(0,\cov (X))$, which is given by
\begin{equation}\label{eq:def of Gaussian_psi}
\psi (x):=\frac{1}{\sqrt{\left(2\pi \right)^{d}\det (\cov (X))}}e^{-\frac{1}{2}x^T \cov (X)^{-1}x}=\frac{1}{\left(2\pi \right)^{\frac{d}{2}}\det A}e^{-\frac{1}{2}\|A^{-1}x\|^2},\quad x\in \R ^d,
\end{equation}
where $A$ is defined in \eqref{eq:def of A}.
Observe that for any $j\in \N$ we have
\begin{equation}\label{eq:convolution of gaussian}
\psi ^{*j}(x)=\frac{1}{\left(2\pi j\right)^{\frac{d}{2}}\det A}e^{-\frac{1}{2j}\|A^{-1}x\|^2}, \quad x\in \R ^d
\end{equation}
and that
\begin{equation}\label{eq:fourier of psi}
\hat{\psi }(t)=e^{-2\pi ^2\|At\|^2},\quad t\in \R ^d,
\end{equation}
where we used part~\ref{item:1 in basic properties} of Claim~\ref{claim:basic facts on the fourier transform} and the fact that $\mathcal{F}(e^{-\|\ \cdot \ \|^2})=\pi ^{\frac{d}{2}}e^{-\pi ^2 \|\ \cdot  \ \|^2}$.

For later reference we write explicitly the following consequence of \eqref{eq:taylor of phi}: For any $M\ge 0$ there are $C,c_M>0$ so that
\begin{equation}\label{eq:phi minus psi}
| \hat{\varphi}^j(t)-\hat{\psi }^j(t) |\le C j\|t\|^3 e^{-c_M j\|t\|^2},\quad \|t\|\le M,\ j\in \N.
\end{equation}
The inequality follows since when $j\|t\|^3 \ge 1 $ and $\|t\|\le M$, using Claim~\ref{claim:basic facts on the fourier transform},
\begin{equation*}
| \hat{\varphi}^j(t)-\hat{\psi }^j(t) |\le | \hat{\varphi}^j(t)|+\hat{\psi }^j(t)\le 2\left(1-c_M\|t\|^2\right)^j\le 2e^{-c_Mj\|t\|^2}\le 2j\|t\|^3e^{-c_M j\|t\|^2}
\end{equation*}
 and when $j\|t\|^3\le 1$,
\begin{equation*}
\hat{\varphi}^j(t)=\left(1-2\pi ^2 \|At\|^2+O\left(\|t\|^3\right) \right)^j=e^{-2\pi ^2 j\|At\|^2+O\left(j\|t\|^3\right)}
=\left(1+O(j\|t\|^3)\right)\hat{\psi }^j(t).
\end{equation*}

\subsection{Asymptotics}\label{sec:asymptotic behavior of W_{L,j}}
Recall that the weights $(W_{L,j})$ are given in \eqref{eq:W_L_j_def} by the formula
\begin{equation}\label{eq:poisson summation for W}
W_{L,j}=\theta L^d\sum _{k\in \Z^d}\varphi^{*j}(Lk)=\theta \sum _{m\in \frac{1}{L} \mathbb Z ^d}\hat{\varphi}^j(m),
\end{equation}
where in the last equality we used the Poisson summation formula (see parts~\ref{item:1 in basic properties}, \ref{item:3 in basic properties} and \ref{item:4 in basic properties} of Claim~\ref{claim:basic facts on the fourier transform}), and that the function $G_L$ from \eqref{eq:def of G_L(z)} is given by
\begin{equation*}
  G_L(z)=\sum _{j=1}^\infty \frac{W_{L,j}}{j}z^j.
\end{equation*}
Define the function
\begin{equation}\label{eq:g_Taylor_expansion}
g(z):=\sum_{j=1}^\infty\frac{\theta \varphi^{*j}(0)}{j}z^j.
\end{equation}
We will show that $W_{L,j}$ approximately equals $\theta L^d\varphi^{*j}(0)$ for small $j$ and therefore, $G_L(z)$ approximately equals $L^dg(z)$ for small values of $|z|$. This will facilitate our analysis as the function $g(z)$ is independent of $L$. We will also see that $g$ determines the critical density in dimensions $d\ge 3$ as
\[g'(1)=\theta \sum _{j=1}^{\infty}\varphi^{*j}(0)=\rho _c.\]
The following lemma, a multidimensional, local central limit theorem, determines the asymptotic behavior of $\varphi^{*j}(0)$.

\begin{lemma}\label{lem:asymptotics of kappa _j}
For sufficiently large $j\in \N$,
\begin{equation}\label{eq:asymptotics of w_j}
\varphi^{*j}(0)=\left(1+O(j^{-\frac{1}{2}})\right)\psi ^{*j}(0)=  \frac{1+O(j^{-\frac{1}{2}})}{\sqrt{\det (\cov (X))\left( 2\pi \right)^d} }\cdot j^{-\frac{d}{2}}.
 \end{equation}
\end{lemma}

\begin{proof}
The second equality in \eqref{eq:asymptotics of w_j} is by \eqref{eq:convolution of gaussian}. We turn to prove the first one. Let $j\in \N$.  As $\varphi$ and $\psi $ are Schwartz functions, there exists a $C_0>0$ so that
\begin{equation}\label{eq:bounds on phi for large t}
\max \left\{\left|\hat{\varphi}(t)\right|,\hat{\psi }(t)\right\}\le \frac{C_0}{\|t\|^{d+1}}, \quad t\in \R ^d.
\end{equation}
Now by parts~\ref{item:2 in basic properties} and \ref{item:3 in basic properties} of Claim~\ref{claim:basic facts on the fourier transform}, we have
\begin{equation}\label{eq:estimate separately}
\begin{split}
|\varphi^{*j}(0)-\psi ^{*j}(0)| &= \left|\intop_{\R ^d} \hat \varphi^j(t)dt-\intop_{\R ^d}\hat{\psi }^j(t)dt\right| \\
&\le \intop _{ \|t\|\le 2C_0} |\hat {\varphi}^j(t)-\hat {\psi}^j(t)|dt + 2\intop _{ \|t\|\ge 2C_0} \max \left\{\left|\hat{\varphi}(t)\right|,\hat{\psi }(t)\right\}dt.
\end{split}
\end{equation}
We estimate each of the integrals separately. First, substituting \eqref{eq:phi minus psi} we obtain
\begin{equation}\label{eq:substituting in integral}
\intop _{ \|t\|\le 2C_0} |\hat {\varphi}^j(t)-\hat {\psi}^j(t)|dt\le Cj\intop _{ \R ^d}\|t\|^3 e^{-cj\|t\|^2}dt\le Cj^{-\frac{d+1}{2}}\intop _{\R ^d}\|s\|^3e^{-c\|s\|^2}ds\le Cj^{-\frac{d+1}{2}},
\end{equation}
where in the second inequality we changed variables to $s=\sqrt{j}\, t$. Second, using \eqref{eq:bounds on phi for large t},
\begin{equation}\label{eq:substituting in integral 2}
\intop _{ \|t\|\ge 2C_0} \max \left\{\left|\hat{\varphi}(t)\right|,\hat{\psi }(t)\right\}dt\le \intop _{\|t\|\ge 2C_0}\left(\frac{C_0}{\|t\|^{d+1}}\right)^jdt \le C e^{-cj}.
\end{equation}
Substituting \eqref{eq:substituting in integral} and \eqref{eq:substituting in integral 2} in \eqref{eq:estimate separately} yields the first equality in \eqref{eq:asymptotics of w_j}.
\end{proof}

The next pair of lemmas determine the asymptotic behavior of $(W_{L,j})$ in all regimes of $j$ and $L$. A change in behavior takes place when $j$ is approximately $L^2$.

\begin{lemma}\label{lem:j<L^(2-epsilon)}
For any $\epsilon >0$ there exists a $C_\epsilon >0$ such that
\begin{equation}
\left|W_{L,j}-\theta L^d\varphi ^{*j}(0)\right|\le C_\epsilon L^{-2}
\end{equation}
for all $L\ge 1$ and integer $j\le L^{2-\epsilon }$.
\end{lemma}

\begin{proof}
First recall that $\varphi^{*j}$ is the density function of $S_j:=\sum _{i=1}^jX_i$ where $(X_i), i\ge 1$, are independent copies of $X$. Intuitively, this implies that the sum $S_j$ for $j\le L^{2-\epsilon }$ is unlikely to have $\|S_j\|\ge L$. The lemma will follow from a quantitative estimate of this type together with control of the smoothness of $\varphi ^{*j}$.

Since $\varphi$ is Schwartz,
\[\underset{j> 1}{\sup}\ \underset{x\in \R ^d}{\sup }\left|\frac{\partial \varphi^{*j}}{\partial x_i}(x)\right|\le \underset{x\in \R ^d}{\sup }\left|\frac{\partial \varphi}{\partial x_i}(x)\right|\cdot \underset{j> 1}{\sup} \intop _{\R ^d}\varphi ^{*(j-1)}(t) dt<\infty, \quad 1\le i\le d,\]
as $\varphi^{*(j-1)}$ is a density function. Thus, there is a $c>0$ so that for any $k\in \Z ^d$ and $j\ge 1 $,
\[\varphi^{*j}(x)\ge \frac{1}{2}\varphi^{*j}(Lk)  ,\quad x\in B\left(Lk,c\varphi^{*j}(Lk)\right),\]
where $B(x,r)\subseteq \R ^d$ is the ball of radius $r$ centered at $x$.
We obtain that for every $k\in \mathbb Z ^d \setminus \{0\}$, on the one hand,
\begin{equation}\label{eq:1}
\begin{split}
\mathbb P (\| S_j\|\ge L\|k\|)&=\intop _{\| x\|\ge L\|k\|}\varphi^{*j}(x)dx\ge \\
\ge \frac{1}{2} \varphi^{*j}(Lk)&\cdot \vol \left[B\left(Lk,c\varphi^{*j}(Lk)\right)\cap \{x\in \R ^d:\|x\|\ge L\|k\|\}\right]
\ge c(\varphi^{*j}(Lk))^{d+1}.
\end{split}
\end{equation}
On the other hand, by Markov's inequality, for any even $n\ge 2$, there is a $C_n>0$ so that,
\begin{equation}\label{eq:2}
\mathbb P (\| S_j\|\ge \| k\|L)\le \frac{\mathbb E \| S_j\|^n}{\| k\|^nL^n}\le \frac{C_n j^{\frac{n}{2}}}{\| k\|^nL^n}=\frac{C_n}{\| k\|^nL^{\epsilon \frac{n}{2}}}.
\end{equation}
To see the second inequality in \eqref{eq:2}, note that if $d=1$ then, as $n$ is even, $\|S_j\|^n = \left(X_1+\dots +X_j\right)^n$ and there are at most $C_nj^{\frac{n}{2}}$ terms with non-zero expectation (and each term has expectation bounded by $C_n$). The case of higher dimensions may be reduced to the one-dimensional case by noting that, as $n$ is even, $\|S_j\|^n = (S_{j,1}^2 + \cdots + S_{j,d}^2)^{\frac{n}{2}} \le C_n (S_{j,1}^n + \cdots + S_{j,d}^n)$, where we write $S_j = (S_{j,1},\ldots, S_{j,d})$.

Thus, combining \eqref{eq:1} and \eqref{eq:2} and taking $n$ sufficiently large as a function $\epsilon$, we get
\[\varphi^{*j}(Lk)\le \frac{C_\epsilon}{L^{d+2}\|k\|^{d+1}},\quad k\in \Z^d\setminus \{0\},\]
and therefore
\[ W_{L,j}-\theta L^d\varphi^{*j}(0)=\theta L^d \sum_{k\in \mathbb Z ^d \setminus \{0\}}\varphi^{*j}(Lk)\le C_\epsilon L^{-2}\sum_{k\in \mathbb Z ^d \setminus \{0\}}\frac{1}{\|k\|^{d+1}}\le C_\epsilon L^{-2}.  \qedhere \]
\end{proof}

\begin{lemma}\label{lem:bound with fourier}
There are $C,c >0$ such that for all $L\ge 1$ and $j\in \N$,
\begin{equation}\label{eq:bound with fourier inequality}
\left|W_{L,j}-\theta \sum _{m\in \frac{1}{L} \mathbb Z ^d}\hat{\psi}^j(m)\right|\le CL^d j^{-\frac{d+1}{2}}e^{-c\frac{j}{L^2}}.
\end{equation}
\end{lemma}

\begin{proof}
The proof follows similar lines to the proof of Lemma~\ref{lem:asymptotics of kappa _j}.
Let $L\ge 1 $  and $j\in \N$. Since $\varphi$ and $\psi $ are Schwartz, there exists a $C_0>0$ so that
\begin{equation}\label{eq:C_0}
\max \left\{\left|\hat{\varphi}(t)\right|,\hat{\psi }(t)\right\}\le \frac{C_0}{\|t\|^{d+1}}, \quad t\in \R ^d.
\end{equation}
We have, by \eqref{eq:poisson summation for W},
\begin{equation}\label{eq:break the sum}
\begin{split}
\left|W_{L,j}-\theta \sum _{m\in \frac{1}{L} \mathbb Z ^d}\hat{\psi}^j(m)\right|&\le \theta \sum _{m\in \frac{1}{L} \mathbb Z ^d} |\hat{\varphi}^j(m)-\hat{\psi}^j(m)|\\
 \le \theta &\sum _{\substack{m\in \frac{1}{L} \mathbb Z ^d \\ \|m\|\le 2C_0}}|\hat{\varphi}^j(m)-\hat{\psi}^j(m)|+ 2\theta \sum _{\substack{m\in \frac{1}{L} \mathbb Z ^d \\ \|m\|\ge 2C_0}} \max \left\{\left|\hat{\varphi}(m)\right|,\hat{\psi }(m)\right\}.
\end{split}
\end{equation}
We estimate each of the sums separately. First, by \eqref{eq:phi minus psi}, we have
\begin{equation*}
\sum _{\substack{m\in \frac{1}{L} \mathbb Z ^d \\ \|m\|\le 2C_0}}|\hat{\varphi}^j(m)-\hat{\psi}^j(m)|\le Cj\sum _{m\in \frac{1}{L} \mathbb Z ^d}\|m\|^3e^{-cj\|m\|^2}=Cj^{-\frac{1}{2}}\sum _{n\in \frac{\sqrt{j}}{L} \mathbb Z ^d}\|n\|^3e^{-c\|n\|^2}.
\end{equation*}
When $j\ge L^2$, one can easily see that the last expression is at most $Cj^{-\frac{1}{2}}e^{-c\frac{j}{L^2}}\le CL^d j^{-\frac{d+1}{2}}e^{-c\frac{j}{L^2}}$. When $j\le L^2$ we have
\begin{equation*}
\sum _{n\in \frac{\sqrt{j}}{L} \mathbb Z ^d}\|n\|^3e^{-c\|n\|^2}\le \sum _{i=1}^{\infty} \sum_{i-1\le \|n\|< i} i^3e^{-c(i-1)^2}\le C \left(\frac{L}{\sqrt{j}}\right)^d \sum _{i=1}^\infty i^{d+2}e^{-c(i-1)^2} \le C \left(\frac{L}{\sqrt{j}}\right)^d.
\end{equation*}
Second, substituting \eqref{eq:C_0},
\begin{equation*}
\sum _{\substack{m\in \frac{1}{L} \mathbb Z ^d \\ \|m\|\ge 2C_0}} \max \left\{\left|\hat{\varphi}(m)\right|,\hat{\psi }(m)\right\}\le \sum _{\substack{m\in \frac{1}{L} \mathbb Z ^d \\ \|m\|\ge 2C_0}}\left(\frac{C_0}{\|m\|^{d+1}}\right)^j\le CL^d e^{-cj}.
\end{equation*}
Inequality \eqref{eq:bound with fourier inequality} follows by substituting all the bounds in \eqref{eq:break the sum}.
\end{proof}

\begin{cor}\label{cor:corollary on W}
There are $C,c>0$ so that for any $L\ge 1$ and $j\in \N$:
\begin{enumerate}[label=(\roman{*})]
\item\label{item:cor on W part 1}
If $j\le L^2$
\begin{equation}\label{eq: cor on W first inequality}
\left|W_{L,j}-\theta L^d \varphi ^{*j}(0)\right|\le C\left(L^{-\frac{1}{2}}+e^{-c\frac{L^2}{j}}\right).
\end{equation}
\item\label{item:cor on W part 2}
If $j\ge L^2$
\begin{equation}
\left|W_{L,j}-\theta \right|\le Ce^{-c\frac{j}{L^2}}.
\end{equation}
\end{enumerate}
\end{cor}

\begin{proof}
Part~\ref{item:cor on W part 2} is an immediate consequence of Lemma~\ref{lem:bound with fourier} as, when $j\ge L^2$,
\begin{equation*}
\left|W_{L,j}-\theta \right|\le Ce^{-c\frac{j}{L^2}}+ \theta \sum _{m\in \frac{1}{L} \mathbb Z ^d \setminus \{0\}}\hat{\psi}^j(m) \le Ce^{-c\frac{j}{L^2}}+\theta \sum _{m\in \frac{1}{L} \mathbb Z ^d \setminus \{0\}}e^{-cj\|m\|^2}\le Ce^{-c\frac{j}{L^2}}.
\end{equation*}
We turn to prove part~\ref{item:cor on W part 1}. Fix $\epsilon=\frac{1}{d+1}$. When $j\le L^{2-\epsilon }$ then \eqref{eq: cor on W first inequality} holds by Lemma~\ref{lem:j<L^(2-epsilon)}. When $L^{2-\epsilon }\le j\le L^2$,
\begin{equation*}
\left|W_{L,j}-\theta L^d \varphi ^{*j}(0)\right|\le \left|W_{L,j}-\theta L^d\sum _{k\in \Z^d}\psi ^{*j}(Lk)\right|+\theta L^d\sum _{k\in \Z^d \setminus \{0\}}\psi ^{*j}(Lk)+\theta L^d\left|\psi ^{*j}(0)-\varphi^{*j}(0)\right|.
\end{equation*}
We bound each term separately. First, as $L^d\sum _{k\in \Z^d}\psi ^{*j}(Lk) = \sum _{m\in \frac{1}{L} \mathbb Z ^d}\hat{\psi}^j(m)$ by the Poisson summation formula (as in \eqref{eq:poisson summation for W}), the first term is bounded by $CL^dj^{-\frac{d+1}{2}}\le CL^{-\frac{1}{2}}$ by the choice of $\epsilon$. Second, we use \eqref{eq:convolution of gaussian} to bound the second term,
\[L^d\sum _{k\in \Z^d \setminus \{0\}}\psi ^{*j}(Lk)\le C \left(\frac{L}{\sqrt{j}}\right)^d \sum _{k\in \Z ^d \setminus \{0\}} e^{-c\frac{L^2}{j}\|k\|^2}\le C \left(\frac{L}{\sqrt{j}}\right)^d e^{-c\frac{L^2}{j}} \le C e^{-c\frac{L^2}{j}}.\]
Lastly, by Lemma~\ref{lem:asymptotics of kappa _j}, the third term is at most $CL^dj^{-\frac{d+1}{2}}\le CL^{-\frac{1}{2}}$ by the choice of $\epsilon$.
\end{proof}

\begin{cor}\label{cor:G_L is close to L^dg}
There is a $C>0$ so that $\left|W_{L,j}-\theta L^d \varphi^{*j}(0)\right|\le C$ for any $L\ge1$ and $j\in \N$. Therefore, for any integer $n\in \N $ there is a $C_n>0$ for which
\[\left|G_L^{(n)}(z)-L^dg^{(n)}(z)\right|\le C_n\left(1-|z|\right)^{-n},\quad |z|<1,\]
where we recall that the functions $G_L$ and $g$ are defined in \eqref{eq:def of G_L(z)} and \eqref{eq:g_Taylor_expansion} respectively.
\end{cor}
\begin{proof}
  The bound $\left|W_{L,j}-\theta L^d\varphi ^{*j}(0)\right|\le C$ follows in the case $j\le L^2$ from Corollary~\ref{cor:corollary on W} and in the case $j\ge L^2$ from Lemma~\ref{lem:asymptotics of kappa _j} and Corollary~\ref{cor:corollary on W}. The bound on the generating functions is an immediate consequence, noting that $\frac{(n-1)!}{(1-z)^n} = \sum_{j=1}^\infty (j-1)\cdots(j-n+1)z^{j-n}$ for $n\in\N$ and $|z|<1$.
\end{proof}

\subsection{Analytic continuation} We shall view the generating functions $G_L$ and $g$,
given by the Taylor expansions \eqref{eq:def of G_L(z)} and
\eqref{eq:g_Taylor_expansion}, as functions of a complex variable.
By Lemma~\ref{lem:asymptotics of kappa _j} and
Corollary~\ref{cor:corollary on W} , these functions are
analytic in the open unit disk, as their Taylor expansions converge
there. In this section we prove that they can be continued analytically to a larger domain $\Delta_0$.

\begin{definition} \label{def:Delta0} For
$R>1$ and $0<\beta<\frac{\pi}{2}$, define an open domain
$\Delta(R,\beta)$ in the complex plane by (see Figure~\ref{fig_delta_0})
\begin{equation*}
\Delta(R,\beta): = \bigl\{z\in \mathbb C\setminus\{1\}\,\colon|z|<R,
\;\left|\arg(z-1)\right|>\beta \bigr\}.
\end{equation*}
\end{definition}
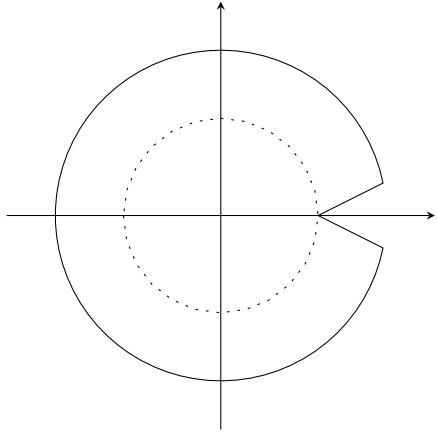
\begin{figure}
\begin{tikzpicture}
\begin{axis}[ticks=none,axis lines=middle,xmin=-2.2,xmax=2.2,ymin=-2.2,ymax=2.2,axis equal image]
\draw [dash pattern=on 1pt off 3pt, domain=-180:180] plot({cos(\x)},{sin(\x)});
\addplot[domain=1:1.668]{0.5*x-0.5};
\addplot[domain=1:1.668]{-0.5*x+0.5};
\draw [samples=70, domain=11.23:360-11.25] plot({1.7*cos(\x)},{1.7*sin(\x)});
\end{axis}
\end{tikzpicture}
\caption{The domain $\Delta(1.7,\frac{\pi }{8})$.}
 \label{fig_delta_0}
\end{figure}

The following technical claim is required to derive the analytic continuation of $G_L$ and $g$.
\begin{claim}\label{claim:1-z phi}
There exists $R_0>1$ such that the following holds for the domain
\begin{equation}\label{eq:Delta_0_def}
\Delta_0:=\Delta\! \left(R_0,\frac{\pi }{8}\right).
\end{equation}
\begin{enumerate}[label=(\roman{*})]
\item\label{item:part 1 of claim 1-z phi}
There is a $c>0$ such that for any $z\in \overline{\Delta_0}$,
\begin{equation}\label{eq:3}
\left|1-z\hat {\varphi}(t)\right|\ge c\min\left\{\|t\|^2 + \left|1-z\right|, 1\right\},\quad t\in \R ^d.
\end{equation}
\item
There is a $C>0$ such that for any $z\in \Delta _0$
\begin{equation}\label{eq:6}
\left| \frac{ 1-z+2\pi ^2 \|At\|^2}{1-z\hat {\varphi}(t)}-1 \right|\le C\|t\|, \quad \|t\|\le 1,
\end{equation}
where $A$ is given in \eqref{eq:def of A}.
\end{enumerate}
\end{claim}
We proceed to show that $G_L$ and $g$ can be continued analytically to $\Delta_0$.

\begin{lemma}\label{lem:analytic continuation}
The generating functions $G_L$, for any $L\ge 1$, and $g$ extend to analytic
functions in $\Delta_0$ (defined in \eqref{eq:Delta_0_def}), where they are given by
\begin{equation}\label{eq:analytic continuation form G_L}
G_L(z)=-\theta \sum _{m\in \frac{1}{L}\Z ^d} \log \left(1-z\hat {\varphi}(m)\right),
\end{equation}
\begin{equation}\label{eq:analytic continuation form g}
 g(z)=-\theta \intop _{\R ^d} \log \left(1-z\hat {\varphi}(t)\right)dt,
\end{equation}
with the sum and integral converging uniformly on compact subsets of $\Delta_0$.
\end{lemma}

We start with the proof of Claim~\ref{claim:1-z phi}.

\begin{proof}[Proof of Claim~\ref{claim:1-z phi}]
We start with the first part. We assume that $t\neq 0$ as the case $t=0$ is straightforward. By Claim~\ref{claim:basic fourier 2} there exists a $\delta >0$ such that for any $0<\|t\|\le \delta $,
\[|\arg (1-\hat{\varphi}(t))|\le \frac{\pi }{24}, \quad \left|1-\hat{\varphi}(t)\right|\ge c\|t\|^2,\quad  |\arg \hat{\varphi}(t)|\le \frac{\pi }{24} \ \mbox{ and } \ |\hat{\varphi }(t)|\ge \frac{1}{2}. \]
Thus, for $\|t\| \le \delta $ and $z\in \mathbb C\setminus \{1\}$ with $\arg (z-1)\ge \frac{\pi }{8}$ we have
\begin{equation}\label{eq:small t}
\begin{split}
|1-z\hat{\varphi}(t)|&= |(1-\hat{\varphi}(t))-(z-1)\hat{\varphi}(t)| \\
&\ge c \left(|1-\hat{\varphi }(t)|+ |1-z|\cdot|\hat{\varphi}(t)|\right)\ge c\left(\|t\|^2+|1-z|\right),
\end{split}
\end{equation}
where in the first inequality we used the fact that for any $w_1,w_2\in \C$ satisfying $|\arg w_1|\le \frac{\pi }{24}$ and $|\arg w_2|\ge \frac{\pi }{12}$ we have
\begin{equation*}
|w_1-w_2|\ge \sin \frac{\pi }{24}\cdot \max \left(|w_1|,|w_2|\right)\ge c \left(|w_1|+|w_2|\right).
\end{equation*}

Again, by Claim~\ref{claim:basic fourier 2}, there exists some $\epsilon>0$ such that $\left|\hat{\varphi}(t)\right|\le 1-\epsilon$ for $\|t\|\ge \delta $. Thus, for $\|t\|\ge \delta $ and $|z|\le 1+\epsilon $,
\begin{equation}\label{eq:for big t}
\left|1-z\hat {\varphi}(t)\right|\ge 1-\left|z\right| \cdot\left|\hat {\varphi}(t)\right|\ge c.
\end{equation}
The first part of the claim with $R_0=1+\epsilon $ follows from \eqref{eq:small t} and \eqref{eq:for big t}.

We continue with the second part. Fix $R_0$ and $\Delta _0$ such that the first part holds. Using the Taylor expansion of $\hat{\varphi}$ in \eqref{eq:taylor of phi}, we obtain
\begin{equation}
\begin{split}
\left|1-z\hat{\varphi}(t) -(1-z+2\pi ^2 \|At\|^2)\right|\le \left|(1-2\pi ^2\|At\|^2-\hat{\varphi}(t))+(z-1)(1-\hat{\varphi}(t))\right|\\
\le C\|t\|^3+C|1-z|\cdot \|t\|^2\le C\|t\|\cdot \left|1-z\hat{\varphi}(t)\right|,\quad 	z\in \Delta _0, \ \|t\|\le 1,
\end{split}
\end{equation}
where in the third inequality we used the first part of the claim. Dividing both sides by $\left|1-z\hat{\varphi}(t)\right|$ yields \eqref{eq:6}.
\end{proof}

We turn to prove Lemma~\ref{lem:analytic continuation}.

\begin{proof}[Proof of Lemma~\ref{lem:analytic continuation}]
We will start by proving that the identities in \eqref{eq:analytic continuation form G_L} and \eqref{eq:analytic continuation form g} hold when $\left|z\right|<1$. Then, we will show that the right-hand side of these identities define analytic functions in $\Delta _0$.

Fix $z$ with $|z|<1$. By \eqref{eq:def of G_L(z)} and \eqref{eq:poisson summation for W},
\begin{equation}
\begin{split}
G_L(z)=\sum _{j=1}^{\infty}\frac{W_{L,j}}{j}z^j&=\theta \sum _{j=1}^\infty \sum _{m\in \frac{1}{L} \mathbb Z ^d}\frac{1}{j}\hat {\varphi}^{j}(m)z^j \\
&=\theta \sum _{m\in \frac{1}{L} \mathbb Z ^d}\sum _{j=1}^{\infty}\frac{1}{j}(z\hat {\varphi}(m))^j=-\theta \sum _{m\in \frac{1}{L} \mathbb Z ^d}\log \left(1-z\hat {\varphi}(m)\right),
 \end{split}
 \end{equation}
where the last equality holds as $\left|z\hat{\varphi}(m)\right|\le |z|<1$. The change of order of summation in the third equality is justified as the sum converges absolutely. Indeed, since $\hat{\varphi}$ is a Schwartz function (part~\ref{item:3 in basic properties} of Claim~\ref{claim:basic facts on the fourier transform}) and $\left| \hat {\varphi}(m)\right|\le 1$,
\begin{equation}\label{eq:8}
\begin{split}
\sum _{j=1}^\infty \sum _{m\in \frac{1}{L} \mathbb Z ^d}\frac{1}{j}\left|\hat {\varphi}^{j}(m)z^j\right|\le \sum _{j=1}^\infty \left|z\right|^j\sum _{m\in \frac{1}{L} \mathbb Z ^d}\left|\hat {\varphi}(m)\right|
\le \frac{\left|z\right|}{1-\left|z\right|}\sum _{m\in \frac{1}{L} \mathbb Z ^d}\left|\hat {\varphi}(m)\right|<\infty .
\end{split}
\end{equation}
We turn to prove the identity for $g$. Using Fourier inversion theorem (Part~\ref{item:2 in basic properties} of Claim~\ref{claim:basic facts on the fourier transform}) we obtain,
\[ g(z)=\sum _{j=1}^\infty \frac{\theta \varphi ^{*j}(0)}{j}z^j=\theta \sum _{j=1}^\infty \frac{z^j}{j}\intop _{\R ^d}\hat{\varphi}^j(t)dt=\theta \intop _{\mathbb R ^d}\sum _{j=1}^\infty \frac{1}{j}(z\hat{\varphi}(t))^jdt=-\theta \intop _{\mathbb R ^d}\log \left(1-z\hat{\varphi}(t)\right) dt,\]
where the second inequality is justified by a calculation similar to \eqref{eq:8}.

Now, consider the right-hand side of the identities \eqref{eq:analytic continuation form G_L} and \eqref{eq:analytic continuation form g} for $z\in \Delta _0$. First note that the logarithms are well defined. Indeed, otherwise there are $z\in \Delta_0$ and $t\in \R ^d$ so that $z\hat{\varphi}(t)=R\ge 1$, but then $z'=\frac{z}{R}\in \Delta _0$ satisfies $1-z'\hat{\varphi}(t)=0$ contradicting the bound in Claim~\ref{claim:1-z phi}.

It suffices to prove that the sum and integral in \eqref{eq:analytic continuation form G_L} and \eqref{eq:analytic continuation form g} converge uniformly on compact subsets of $\Delta _0$.
Let $\Delta \subseteq \Delta _0$ be a compact subset. By the first part of Claim~\ref{claim:1-z phi}, for any $t\in \R ^d$ and $z\in \Delta $,  $\left|1-z\hat {\varphi}(t)\right|\ge c_\Delta  $. Thus, considering separately the cases of small and large $|\hat{\varphi}(t)|$, we obtain
\[\left|\log \left(1-z\hat{\varphi}(t)\right)\right|\le C_\Delta \left|\hat{\varphi}(t)\right|.\]
The uniform convergence follows as $\hat{\varphi}$ is a Schwartz function.
\end{proof}

In the following corollary we deduce explicit formulas for the derivatives of $G_L$ and $g$ in the domain $\Delta_0$ which will be of repeated use in the sequel.

\begin{cor}\label{cor:G_L^(n) and g^(n)}
For any $L\ge 1$ and $n\in \N $,
\begin{equation}\label{eq:analytic continuation for G_L'}
G_L^{(n)}(z)=\theta\left(n-1\right)! \sum_{m\in \frac{1}{L} \mathbb Z ^d}\frac{\hat {\varphi} ^n(m)}{\left(1-z\hat {\varphi}(m)\right)^n},\quad z\in \Delta _0.
\end{equation}
\begin{equation}\label{eq:analytic continuation for g'}
g^{(n)}(z)=\theta \left(n-1\right)! \intop _{\mathbb R ^d}\frac{\hat {\varphi}^n(t)}{\left(1-z\hat {\varphi}(t)\right)^n}dt,\quad z\in \Delta _0.
\end{equation}
\end{cor}

\begin{proof}
The formulas follow from Lemma~\ref{lem:analytic continuation} by differentiating the identities in \eqref{eq:analytic continuation form G_L} and \eqref{eq:analytic continuation form g} under the sum and integral signs, making use of the uniform convergence on compact subsets.
\end{proof}

\subsection{Near the singularity}
The following lemma determines the asymptotic behavior of the derivatives of the function $g$ near the potential singularity $z=1$. The derivative of $g$ is closely related to the polylogarithm function. Indeed, they differ merely by a constant multiple when the density function $\varphi$ is the Gaussian density $\psi$ (given by \eqref{eq:def of Gaussian_psi}). Our goal here is to show that, for general $\varphi$, $g$ shares the asymptotic behavior of the polylogarithm function near $z=1$.

\begin{lemma}\label{lem:asymptotics of g at 1}
The following holds for any integer $n\ge 1$:
\begin{enumerate}[label=(\roman{*})]
\item\label{item:d>2n}
If $d>2n$ then the function $g^{(n)}$ may be extended to a continuous function in $\overline{\Delta _0}$.
\item\label{item:d=2n}
If $d=2n$ then for $z\in \Delta _0$ we have
\[g^{(n)}(z)=\frac{\theta }{\left(2\pi \right)^{\frac{d}{2}}\sqrt{\det(\cov(X))}}\log \left(\frac{1}{1-z}\right)+O(1).\]
\item\label{item:d<2n}
If $d<2n$ then for $z\in \Delta _0$ we have
\[g^{(n)}(z)=\frac{\theta \cdot \Gamma \left(n-\frac{d}{2}\right)}{\left(2\pi \right)^{\frac{d}{2}}\sqrt{\det(\cov(X))}}\left(1-z\right)^{\frac{d}{2}-n} +O_n\left(\left|1-z\right|^{\frac{d}{2}-n+\frac{1}{2}}+\left| \log \left(1-z\right)\right|\right),\]
where $\Gamma $ is the gamma function.
\end{enumerate}
\end{lemma}

\begin{proof}
Inside the unit disc, the three parts follow in a straightforward manner from the definition of $g$ in \eqref{eq:g_Taylor_expansion} and the asymptotic behavior of $\varphi ^{*j}(0)$ given in Lemma~\ref{lem:asymptotics of kappa _j}. However, since we need the asymptotics outside the radius of convergence as well, we will derive it using the analytic continuation representation of $g^{(n)}$ given in \eqref{eq:analytic continuation for g'}.

We start with part~\ref{item:d>2n}. When $d>2n$, we have by \eqref{eq:3},
\[\underset{z\in \overline{\Delta _0}}{\sup}\ \intop _{\R ^d}\left|\frac{\hat{\varphi}^n(t)}{\left(1-z\hat{\varphi}(t)\right)^n}\right|dt\le C_n\intop_{\|t\|\le 1 } \frac{1}{\|t\|^{2n}}+C_n\intop_{\|t\|\ge 1 }\left|\hat{\varphi}(t)\right|dt<\infty\]
as $\hat{\varphi}$ is a Schwartz function. Thus, the integral on the right-hand side of \eqref{eq:analytic continuation for g'} converges uniformly in $z\in \overline{\Delta _0}$ and therefore it is continuous.

We turn to prove parts~\ref{item:d=2n} and \ref{item:d<2n}. Using \eqref{eq:analytic continuation for g'}, the Taylor expansion of $\hat{\varphi}$ in \eqref{eq:taylor of phi} and the bounds \eqref{eq:6} and \eqref{eq:3} we get
\[g^{(n)}(z)=O_n(1)+\theta \left(n-1\right)!\intop _{\|t\|\le 1}\frac{1+O_n(\|t\|)}{\left(1-z+2\pi ^2 \|At\|^2\right)^n}dt,\quad z\in \Delta _0.\]
We develop the above integral using a change of variables,
\begin{equation*}
\begin{split}
&\intop _{\|t\|\le 1}\frac{1+O_n(\|t\|)}{\left(1-z+2\pi ^2 \|At\|^2\right)^n}dt=O_n(1)+\frac{1}{\det (A)}\intop _{\|x\|\le 1 }\frac{1+O_n(\|x\|)}{\left(1-z+2\pi ^2 \|x\|^2\right)^n}dx\\
&=O_n(1)+\frac{\vol \left(\mathbb S^{d-1}\right)}{\det A}\intop _0^1 \frac{r^{d-1}\left(1+O_n(r)\right)}{\left(1-z+2\pi ^2r^2\right)^n}dr=O_n(1)+\frac{2\left(2\pi \right)^{-\frac{d}{2}}}{\Gamma \left(\frac{d}{2}\right)\det A}\intop _0^1\frac{s^{d-1}\left(1+O_n(s)\right)}{\left(1-z+s^2\right)^n}ds,
\end{split}
\end{equation*}
where we denote by $\vol \left(\mathbb S^{d-1}\right)=2\pi ^{\frac{d}{2}}(\Gamma \left(\frac{d}{2}\right))^{-1}$ the surface area of the $(d-1)$-dimensional unit sphere and where an additional $O_n(1)$ term is added due to the changes in the domain of integration. We set
\[I_{n,d}(\xi ):=\intop _0^1\frac{s^{d-1}ds}{\left(\xi +s^2\right)^n},\quad \xi \in \C \setminus \{0\},\ |\arg(\xi )|\le \frac{7\pi }{8}.\]
The last computation shows that, for $z\in \Delta _0$,
\begin{equation}\label{eq:substitute I}
\left| g^{(n)}(z)-\frac{2\theta \left(n-1\right)!\left(2\pi \right)^{-\frac{d}{2}}}{\Gamma \left(\frac{d}{2}\right)\det A}\cdot I_{n,d}(1-z)\right|\le C_n\intop _0^1\frac{s^dds}{\left|1-z+s^2\right|^n}\le C_nI_{n,d+1}(|1-z|).
\end{equation}
Thus is suffices to estimate $I_{n,d}(\xi )$ as $\xi \to 0$. To this end, note that $I_{n,d}(\xi)$ is bounded when $d>2n$ and that, using integration by parts,
\begin{equation}\label{eq:recursive formula}
I_{n,d}(\xi)=\intop _0^1\frac{s^{d-1}}{\left(\xi +s^2\right)^n}ds=\frac{s^d}{d\left(\xi +s^2\right)^n}\Bigg|_0^1+\frac{2n}{d}\intop _0^1 \frac{s^{d+1}}{\left(\xi +s^2\right)^{n+1}}=\frac{2n}{d}I_{n+1,d+2}(\xi )+O_n(1).
\end{equation}
As a result, it suffices to estimate $I_{k,1}$ and $I_{k,2}$ for $k\ge 1$. For any $\xi \in \C \setminus \{0\}$ with $|\arg (\xi )|\le \frac{7\pi}{8}$, we have
\[I_{k,2}(\xi)=\intop _0^1 \frac{s}{\left(\xi +s^2\right)^k}ds=\frac{1}{2}\intop _0^1 \frac{dy}{\left(\xi +y\right)^k} =O_k(1)+\left\{ \begin{array}{ll}
 -\frac{1}{2}\log \left(\xi \right) & \mbox{for } k=1 \\
\frac{1}{2}\left(k-1\right)^{-1}\xi ^{-\left(k-1\right)} & \mbox{for } k>1 \\
\end{array}
\right.\]
and
\[I_{k,1}(\xi )=\intop _0^1 \frac{ds}{\left(\xi+s^2\right)^n}=O_k(1)+\frac{1}{2}\intop _{-\infty}^\infty \frac{ds}{\left(\xi+s^2\right)^k}=\frac{\sqrt{\pi }\Gamma \left(k-\frac{1}{2}\right)\xi^{\frac{1}{2}-k}}{2\left(k-1\right)!}+O_k(1),\]
where the last equation follows using a standard residue argument. Indeed, it follows by integrating over the semi-circle contour noting that the integrand has a unique pole in  $\{ \imag(z)>0\}$, located at $i\sqrt{\xi}$ having order $k$. The details are left to the reader.

We conclude that, when $d=2n$,
\[I_{n,d}(\xi )= -\frac{1}{2}\log \xi +O(1)\]
and when $d<2n$,
\[I_{n,d}(\xi)=\frac{\Gamma \left(\frac{d}{2}\right)\Gamma \left(n-\frac{d}{2}\right)}{2(n-1)!}\xi ^{\frac{d}{2}-n}+O_n(1)\]
as one can check that the last expressions satisfy the recursion in \eqref{eq:recursive formula} and the cases $d=1,2$ (note that $\Gamma \left(\frac{1}{2}\right)=\sqrt{\pi}$).
Substituting the last estimates in \eqref{eq:substitute I} finishes the proof of the lemma.
\end{proof}

\section{The sub-critical regime}\label{sec:sub-critical}
In this section we establish parts~\ref{item:constant in d=1} and~\ref{item:sub-critical in d=1} of
Theorem~\ref{thm:d=1}, part~\ref{item:constant in d=2} and the sub-critical regime in part~\ref{item:sub-critical and critical in d=2} of Theorem~\ref{thm:d=2}
and part~\ref{item:sub-critical in d=3} of Theorem~\ref{thm:d=3}.

\setlist[description]{font=\normalfont}
Let $\rho = \rho(N)$ and consider several possibilities for the
dimension $d$ and the asymptotic regime of $\rho (N)$ as
$N\to\infty$ corresponding to the various statements in the
theorems. We give names to these cases to simplify later reference.
By the name (Sub-Critical) we refer collectively to any of the following
cases:
\begin{description}[leftmargin=!,labelwidth=75pt]
\item[(SubConst1)] Dimension $d=1$ and density $\rho$  satisfying $\rho \to \rho _* \in (0,\infty)$ as $N\to\infty$.
\item[(Sub1)] Dimension $d=1$ and density $\rho$ satisfying
$\rho\to\infty$ and $\rho=o(\sqrt{N})$.
\item[(SubConst2)]
Dimension $d=2$ and density $\rho$  satisfying $\rho \to \rho _* \in (0,\infty)$ as $N\to\infty$.
\item[(Sub2)]
Dimension $d=2$ and density $\rho$ satisfying $\rho\to\infty$ and
$\frac{\rho}{\log N}\rightarrow \alpha \in \left[0,\alpha
_c\right)$.
\item[(Sub3)] Dimension $d\ge 3$ and density $\rho$  satisfying $\rho \to \rho _* \in (0,\rho _c)$ as $N\to\infty$.
\end{description}

We remind the reader that the critical thresholds $\alpha_c$ and
$\rho_c$ are defined in \eqref{eq:def of alpha _c } and
\eqref{eq:rho_c_def}. We note explicitly that we allow $\rho$ to vary with $N$ in the cases (SubConst1), (SubConst2) and (Sub3), thus obtaining a somewhat stronger results than those stated in part~\ref{item:constant in d=1} of the main theorems. This additional flexibility will be used in our analysis of the critical regime in dimensions $d\ge3$ in Section~\ref{sec:critical in dimension 5}.

Our strategy, following \cite[Section~3.2]{bogachev2015asymptotic}, is to use saddle point analysis to estimate the Cauchy integral (recall \eqref{eq:generating_h(v)}),
\begin{equation}\label{eq:H_n_Cauchy_integral}
  H_n(L)=\left[z^n\right]e^{G_L(z)}=\frac{1}{2\pi i}\ointop _{\gamma }\frac{e^{G_L(z)}}{z^{n+1}}dz
\end{equation}
(see Theorem~\ref{thm:sub-critical integral}) and then apply the results of Section~\ref{sec:combinatorics} to connect $H_n(L)$ with the distribution of $L_1$ (Section~\ref{sec:proof of sub-critical parts of the main theorem}).

\subsection{Saddle point analysis}\label{sec:saddle point}
The integrand in \eqref{eq:H_n_Cauchy_integral} is $\exp(G_L(z) - (n+1)\log(z))$ and thus its saddle points, the critical points of the exponent, are the solutions $z$ to $z G_L'(z) = n+1$. Motivated by this and the fact that the values of $n$ in our analysis will be close to $N$, we define $r_N=r_{N,L}$ to be the unique $0<r<1$ satisfying
\begin{equation}\label{eq:def of r_N}
rG_L'(r)=N.
\end{equation}
This solution exists as, by Corollary~\ref{cor:corollary on W}, $G_L'(r)\overset{r\uparrow 1}{\longrightarrow }\infty$, for all $L\ge 1$. $r_N$ is unique since $G_L'$ has non-negative Taylor coefficients. The saddle point method suggests to take a contour of integration that passes through $r_N$.

In the following theorem we find the asymptotic behavior of $H_{N-j}(L)$. In the theorem and what follows we set, for $N\ge 1$,
\begin{equation}\label{eq:def of a_N_b_N}
a_N:=r_N G_L'(r_N)+r_N^2G_L''(r_N)\quad\text{and}\quad b_N:=r_N G_L'(r_N)+3r_N^2G_L''(r_N)+r_N^3G_L'''(r_N)
\end{equation}
and note that by \eqref{eq:def of r_N} and the fact that the Taylor coefficients of $G_L$ are non-negative we have $a_N, b_N\ge N$.
\begin{thm}\label{thm:sub-critical integral}
Let $j_N$ be a sequence of integers such that $1 \le j_N\le N$ and $j_N^2=o\left(a_N\right)$ as $N\to\infty$. Then, in each of the (Sub-Critical) asymptotic regimes,
\begin{equation}\label{eq:main subcritical h_N(N-j)}
H_{N-j}(L)=[z^{N-j}]e^{G_L(z)}\sim\frac{ e^{G_L(r_N)}}{r_N^{N-j}\sqrt{2\pi a_N}},\quad N\to \infty.
\end{equation}
uniformly in $0\le j\le j_N$.
\end{thm}

The proof of the theorem relies on the following technical estimates.

\begin{prop}\label{prop:sub-critical proposition}
In each of the (Sub-Critical) cases:
\begin{enumerate}[label=(\roman{*})]
\item\label{item:part one of subcritical prop}
As $N\to \infty$,
\begin{equation}\label{eq:b_N^2=o(a_N^3)}
b_N^2=o(a_N^3).
\end{equation}
\item\label{item:part two of subcritical prop}
There exists $c' >0$ such that, for large $N$,
\begin{equation}\label{eq:36}
\real \left(G_L\left(r_Ne^{it}\right)\right)\le G_L(r_N)-c'\cdot a_N^{\frac{1}{4}}\sqrt{|t|},\quad a_N^{-\frac{1}{2}}\le |t|\le \pi,
\end{equation}
where $c'$ may depend on $\rho _*$ in the cases (Sub3), (SubConst1) and (SubConst2) (in addition to the usual dependence on $d,\varphi$ and $\theta$).
\end{enumerate}
\end{prop}

Let us first show how Theorem~\ref{thm:sub-critical integral} follows from Proposition~\ref{prop:sub-critical proposition}.

\begin{proof}[Proof of Theorem~\ref{thm:sub-critical integral}]
Fix a sequence $(t_N)_{N=1}^\infty \subseteq (0,\pi )$ so that as $N\to \infty$,
\begin{equation}\label{eq:t_N_relations}
  a_Nt_N^2\to \infty ,\quad b_Nt_N^3\to 0 \mbox{ and }  t_Nj_N\to 0.
\end{equation}
This is possible by part~\ref{item:part one of subcritical prop} of Proposition~\ref{prop:sub-critical proposition} and the assumption on $j_N$, e.g., by setting $t_N:=\min\{a_N^{-1/4}j_N^{-1/2}, b_N^{-1/6}a_N^{-1/4}, \pi\}$. By the Cauchy integral formula with the contour $\gamma $ parametrized as $\gamma (t):=r_Ne^{it}$ for $t\in \left[-\pi ,\pi \right]$,
\[ H_{N-j}(L)=[z^{N-j}]e^{G_L(z)}=\frac{1}{2\pi i}\ointop _\gamma \frac{e^{G_L(z)}}{z^{N-j+1}}dz=\frac{1}{2\pi r_N^{N-j}}\intop _{-\pi }^{\pi }e^{G_L\left(r_Ne^{it}\right)-itN+itj}dt=I_1+I_2,\]
where $I_1$, $I_2$ are the integrals corresponding to $|t|\le t_N$ and $t_N< |t|\le \pi$. We start by estimating $I_1$ via a Taylor expansion of $G_L\left(r_N e^{it}\right)$ around $t=0$. A straightforward calculation yields that
\begin{align*}
  &\frac{d}{dt} G_L(r_N e^{it}) = ir_Ne^{it}G_L'(r_Ne^{it}),\\
  &\frac{d^2}{dt^2} G_L(r_N e^{it}) = -r_N e^{it} G_L'(r_Ne^{it}) - r_N^2 e^{2it} G_L''(r_N e^{it}),\\
  &\frac{d^3}{dt^3} G_L(r_N e^{it}) = -ir_N e^{it} G_L'(r_Ne^{it}) - 3ir_N^2 e^{2it} G_L''(r_N e^{it}) - ir_N^3e^{3it} G_L'''(r_N e^{it}),
\end{align*}
from which we have, by the definition of $r_N$ and as $G_L$ has non-negative Taylor coefficients, that
\begin{equation*}\label{eq:Taylor expantion}
\begin{split}
G_L\left(r_Ne^{it}\right)=G_L(r_N)+it N-\frac{a_N}{2}t^2+O(b_Nt_N^3),\quad |t|\le t_N.
\end{split}
\end{equation*}
Substituting this expansion in the definition of $I_1$ and using the relations \eqref{eq:t_N_relations},
\[ I_1=\frac{e^{G_L(r_N)}}{2\pi r_N^{N-j}}\intop _{-t_N}^{t_N}e^{-\frac{a_N}{2}t^2+o(1)}dt\sim \frac{e^{G_L(r_N)}}{\pi r_N^{N-j}\sqrt{2a_N}}\intop _{-\infty}^\infty e^{-s^2}ds=\frac{e^{G_L(r_N)}}{r_N^{N-j}\sqrt{2\pi a_N}},\quad N\to\infty.\]
We turn to bound $I_2$. By part~\ref{item:part two of subcritical prop} of Proposition~\ref{prop:sub-critical proposition}, for large enough $N$,
\[|I_2|\le \frac{C}{r_N^{N-j}}\intop_{t_N}^{\pi}e^{\real \left(G_L\left(r_Ne^{it}\right)\right)}dt\le \frac{C\cdot e^{G_L(r_N)} }{r_N^{N-j}}\intop_{t_N}^{\infty}e^{-c'\cdot a_N^{\frac{1}{4}}\sqrt{t}}dt= \frac{C\cdot e^{G_L(r_N)} }{r_N^{N-j}\sqrt{a_N}}\intop _{\sqrt{a_N}t_N}^\infty e^{-c'\sqrt{s}}ds,\]
and the last expression is negligible compared to $I_1$ as $\sqrt{a_N}t_N\to \infty$ by \eqref{eq:t_N_relations}.
\end{proof}

The rest of Subsection~\ref{sec:saddle point} is devoted to the proof of Proposition~\ref{prop:sub-critical proposition}.

Recall the definition of $r_N$ in \eqref{eq:def of r_N} and that $\rho =\frac{N}{L^d}$ is the density. In the next lemma we find the asymptotic behavior of~$r_N$.

\begin{lemma}\label{lem:find r_N}
As $N\to \infty$ we have:
\begin{enumerate}[label=(\roman{*})]
\item
In case (Sub1),
\[1-r_N\sim \frac{\theta ^2}{2\var(X)\rho^2},\quad \text{in particular}\quad 1-r_N\to 0 .\]
\item
In case (Sub2),
\[\log \left(\frac{1}{1-r_N}\right)\sim \frac{\rho }{\alpha _c},\quad \text{in particular}\quad \frac{1}{1-r_N}=N^{\frac{\alpha }{\alpha _c}+o(1)}. \]
\item
In cases (Sub3), (SubConst1) and (SubConst2),
\[  r_N\to r_*<1,\]
where $r_*$ is the unique solution of $rg'(r)=\rho _*$ for $r\in(0,1)$.
\end{enumerate}
\end{lemma}

\begin{proof}
We start with the case $d=1$. Corollary~\ref{cor:G_L is close to L^dg} and Lemma~\ref{lem:asymptotics of g at 1} imply that
\begin{equation}\label{eq:r_G_asymptotics 1}
  \left|rG_L'(r) - \frac{r\theta L}{\sqrt{2\var(X)(1-r)}}\right|\le C\left(\frac{r}{1-r} + L r\log\left(\frac{1}{1-r}\right)\right),\quad r\in[0,1),\, L\ge 1.
\end{equation}
We work in the asymptotic regime given by (Sub1). Denote $I:=[\frac{\theta^2}{8\var(X)\rho^2}, \frac{2\theta^2}{\var(X)\rho^2}]$ and recall that $\rho = \frac{N}{L}\to\infty$ as $N\to\infty$. Having the terms in \eqref{eq:r_G_asymptotics 1} in mind, we observe that as $N\to\infty$, uniformly in $1-r\in I$,
\begin{align*}
  &\frac{r\theta L}{\sqrt{2\var(X)(1-r)}} \in \left[\frac{1}{4} N, 2N\right]\quad\text{and}\quad\frac{r}{1-r} + L r\log\left(\frac{1}{1-r}\right) = o(N),
\end{align*}
where we used that $\rho^2 = o(N)$ and $L \log\left(\frac{1}{1-r}\right) \sim 2N \frac{\log(\rho)}{\rho}$.
It follows, since $r_N$ is the unique value in $[0,1)$ for which \eqref{eq:def of r_N} holds and since $rG_L'(r)$ is increasing in $[0,1)$ that
\begin{equation*}
  N = r_N G_L'(r_N) \sim \frac{\theta L}{\sqrt{2\var(X)(1-r_N)}},\quad\text{$N\to\infty$},
\end{equation*}
from which the required asymptotic formula follows.

We continue with the case $d=2$. Corollary~\ref{cor:G_L is close to L^dg} and Lemma~\ref{lem:asymptotics of g at 1} imply that
\begin{equation}\label{eq:r_G_asymptotics 2}
  \left|rG_L'(r) - rL^2\alpha_c\log\left(\frac{1}{1-r}\right)\right|\le C\left(\frac{r}{1-r} + L^2\right),\quad r\in[0,1),\, L\ge 1.
\end{equation}
We work in the asymptotic regime (Sub2), where $\rho = \frac{N}{L^2}\to\infty$ and $\frac{\rho}{\log N}\to\alpha\in[0,\alpha_c)$ as $N\to\infty$. Fix $0<\epsilon<\min\left\{\frac{\alpha_c}{\alpha}-1, 1\right\}$. Denote $I:=\left[\exp\left(-\frac{(1+\epsilon)\rho}{\alpha_c}\right), \exp\left(-\frac{\rho}{2\alpha_c}\right)\right]$. Then, as $N\to\infty$, uniformly in $1-r\in I$,
\begin{align*}
  &rL^2\alpha_c\log\left(\frac{1}{1-r}\right) \in \left[\frac{1}{4}N, 2N\right]\quad\text{and}\quad\frac{r}{1-r} + L^2 = o(N),
\end{align*}
where we used that $\lim_{N\to\infty} \frac{(1+\epsilon)\rho}{\alpha_c \log N} < 1$ and $L^2 = N /\rho$.
It follows, since $r_N$ is the unique value in $[0,1)$ for which \eqref{eq:def of r_N} holds and since $rG_L'(r)$ is increasing in $[0,1)$ that
\begin{equation*}
  N = r_N G_L'(r_N) \sim L^2\alpha_c\log\left(\frac{1}{1-r_N}\right),\quad\text{$N\to\infty$},
\end{equation*}
from which the required asymptotic formula follows.

In cases (Sub3), (SubConst1) and (SubConst2), for sufficierntly large $N$ we have
\[\frac{r_*+1}{2} G_L'\left(\frac{r_*+1}{2}\right)\ge L^d\frac{r_*+1}{2} g'\left(\frac{r_*+1}{2}\right)=N\rho ^{-1}\frac{r_*+1}{2} g'\left(\frac{r_*+1}{2}\right) > N,\]
where the first inequality is by comparison of the Taylor coefficients (see \eqref{eq:W_L_j_def}, \eqref{eq:def of G_L(z)} and \eqref{eq:g_Taylor_expansion}) and the last inequality holds as $\rho \to \rho _*=r_*g'(r_*)$ and as $rg'(r)$ is strictly increasing. As $rG_L'(r)$ is increasing, we conclude that, for sufficiently large $N$, $r_N\le \frac{r_*+1}{2}<1$. Thus, by Corollary~\ref{cor:G_L is close to L^dg},
\[N=r_NG_L'(r_N)\sim N\rho ^{-1}r_Ng'(r_N)\sim  N\rho_* ^{-1}r_Ng'(r_N)\]
so that $r_Ng'(r_N)\to \rho _*$ and therefore $r_N\rightarrow r_*$.
\end{proof}

In the next lemma we find the asymptotic behavior of~$a_N$ and $b_N$.

\begin{lemma}\label{lem:a_N_b_N_estimates}
As $N\to \infty$ we have:
\begin{enumerate}[label=(\roman{*})]
\item
In case (Sub1),
\[\text{$a_N\sim \frac{\var(X) N\rho ^2}{\theta ^2}\quad$ and $\quad b_N\sim \frac{3\var(X)^2N\rho ^4}{\theta ^4}$}.\]
\item
In case (Sub2),
\[\text{$a_N\sim \frac{\alpha _c N}{\rho \left(1-r_N\right)}=N^{1+\frac{\alpha }{\alpha _c}+o(1)}\quad$ and $\quad b_N\sim N^{1+\frac{2\alpha }{\alpha _c}+o(1)}$.}\]
\item\label{item:a_N_b_N fixed rho}
In cases (Sub3), (SubConst1) and (SubConst2),
\[\text{$c_\rho N\le a_N\le C_\rho N\quad$ and $\quad c_\rho N\le b_N\le C_\rho N$.}\]
\end{enumerate}
\end{lemma}
\begin{proof}
In (Sub1), by Corollary~\ref{cor:G_L is close to L^dg}, Lemma~\ref{lem:asymptotics of g at 1} and Lemma~\ref{lem:find r_N}, as $N\to \infty$ we have
\begin{equation}\label{eq:d=1 G''_N(r_N)}
G_L''(r_N)\sim Lg''(r_N)\sim \frac{\theta L}{2\sqrt{2}\sigma }\left(1-r_N\right)^{-\frac{3}{2}}\sim \frac{\sigma ^2N\rho ^2}{\theta ^2},
\end{equation}
where we denoted $\sigma:=\sqrt{\var(X)}$. In a similar manner,
\begin{equation}\label{eq:d=1 G'''_N(r_N)}
G_L'''(r_N)\sim Lg'''(r_N)\sim\frac{3\theta L}{4\sqrt{2}\sigma }\left(1-r_N\right)^{-\frac{5}{2}} \sim  \frac{3\sigma ^4N\rho ^4}{\theta ^4}.
\end{equation}
The last two asymptotic equalities together with the definitions of $a_N$ and $b_N$ from \eqref{eq:def of a_N_b_N} and the definition of $r_N$ imply that $a_N\sim G_L''(r_N)\sim\frac{\sigma ^2N\rho ^2}{\theta ^2}$ and $b_N\sim G_L'''(r_N)\sim\frac{3\sigma ^4N\rho ^4}{\theta ^4}$ as $N\to\infty$.

Similarly in (Sub2),
\begin{equation}\label{eq:d=2 G''_N(r_N)}
a_N\sim G_L''(r_N)\sim \frac{\alpha _c N}{\rho \left(1-r_N\right)}=N^{1+\frac{\alpha }{\alpha _c}+o(1)}, \quad b_N\sim G_L'''(r_N)= N^{1+\frac{2\alpha }{\alpha _c}+o(1)}.
\end{equation}

Similar asymptotic estimates apply also to the cases (Sub3), (SubConst1) and (SubConst2) and show that $G_L''(r_N)$ and $G_L'''(r_N)$ have order of magnitude $N$ as $N\to\infty$. Thus, $a_N$ and $b_N$ are also of the same order of magnitude.
\end{proof}

We now have everything needed to prove Proposition~\ref{prop:sub-critical proposition}.

\begin{proof}[Proof of Proposition~\ref{prop:sub-critical proposition}]
Part~\ref{item:part one of subcritical prop} of the proposition follows immediately from Lemma~\ref{lem:a_N_b_N_estimates}.

We turn to prove part~\ref{item:part two of subcritical prop}. Let $j_0$ be a fixed large integer so that, by Lemma~\ref{lem:asymptotics of kappa _j}, $\varphi ^{*j}(0)>c j^{-\frac{d}{2}}>0$ for all $j\ge j_0$. Thus, using the definitions \eqref{eq:def of G_L(z)} of $G_L$ and \eqref{eq:W_L_j_def} of $W_{L,j}$, we have
\begin{equation}\label{eq:bounds on Re}
\begin{split}
G_L(r_N)-\real&\left(G_L\left(r_Ne^{it}\right)\right)=\sum_{j=1}^\infty \frac{W_{L,j}r_N^j}{j}\left(1-\real\left(e^{itj}\right)\right) \\
& \ge c \frac{N}{\rho }\sum_{j=1}^{\infty } \frac{\varphi^{*j}(0) r_N^j}{j}\left(1-\cos (tj)\right)\ge c \frac{N}{\rho }\sum_{j=j_0}^{\infty } \frac{r_N^j}{j^{\frac{d}{2}+1}}\left(1-\cos (tj)\right)=:S.
\end{split}
\end{equation}
Our goal is to show that there exists $c'>0$ (depending on $\rho_*$ in (Sub3), (SubConst1) and (SubConst2)) such that $S\ge c' a_N^{1/4} \sqrt{|t|}$ for large $N$, uniformly in $a_N^{-1/2}\le |t|\le \pi$. Examining the expression for $S$ reveals that it may be useful to compare the relative sizes of $\frac{N}{\rho}$ and $a_N$. We record the following relations as $N\to\infty$, which follow directly from Lemma~\ref{lem:a_N_b_N_estimates} and Lemma~\ref{lem:find r_N},
\begin{equation}\label{eq:a_N_relations_sub12}
  a_N^{1/4} \le C\frac{N}{\rho}(1-r_N)^{\frac{d-1}{2}}\;\;\text{and}\;\; a_N\le C\frac{N}{\rho}(1-r_N)^{\frac{d}{2}-2},\quad\text{in (Sub1), (Sub2)},
\end{equation}

Suppose first that $|t|\ge \frac{1}{2j_0}$. In this case $\max\{1-\cos(tj_0),1-\cos(t(j_0+1))\}\ge c$. Thus, when $N$ is sufficiently large, by \eqref{eq:a_N_relations_sub12} and part~\ref{item:a_N_b_N fixed rho} of Lemma~\ref{lem:a_N_b_N_estimates},
\[S\ge c N\rho ^{-1} r_N^{j_0+1}\ge c \cdot a_N^{\frac{1}{4}}\ge c \cdot a_N^{\frac{1}{4}}\sqrt{|t|},\]
Second, suppose that $|t|<\frac{1}{2j_0}$. Consider the cases (Sub3), (SubConst1) and (SubConst2). We bound the sum $S$ by the $j=j_0$ element and use the inequality $1-\cos (x)\ge cx^2$ for $|x|\le \frac{1}{2}$ and part~\ref{item:a_N_b_N fixed rho} of Lemma~\ref{lem:a_N_b_N_estimates} to obtain that for $|t|\ge a_N^{-\frac{1}{2}}$ and large $N$,
\[S\ge cN\rho ^{-1}r_N^{j_0}t^2\ge c_{\rho_*}\cdot a_N t^2\ge c_{\rho_*}\cdot a_N^{\frac{1}{4}}\sqrt{|t|}.\]

Now consider the cases (Sub1) and (Sub2). Observe that when
$j\le \min \left(|t|^{-1},\left(1-r_N\right)^{-1}\right)$
we have the bounds $r_N^{j}\ge c$ and $1-\cos (tj)\ge ct^2j^2$. Suppose $N$ is large. If $a_N^{-\frac{1}{2}}\le |t|\le 1-r_N$ then, by \eqref{eq:a_N_relations_sub12},
\[S\ge c\frac{N}{\rho }t^2\sum _{j=j_0}^{\lfloor \left(1-r_N\right)^{-1}\rfloor}j^{1-\frac{d}{2}}\ge c \frac{N}{\rho}t^2\left(1-r_N\right)^{\frac{d}{2}-2}\ge c\cdot a_Nt^2\ge c\cdot a_N^{\frac{1}{4}}\sqrt{|t|},\]
whereas if $|t|\ge 1-r_N$ then, by \eqref{eq:a_N_relations_sub12},
\[S\ge c\frac{N}{\rho }t^2\sum _{j=j_0}^{\lfloor |t|^{-1}\rfloor}j^{1-\frac{d}{2}}\ge c \frac{N}{\rho}|t|^{\frac{d}{2}}\ge c\cdot a_N^{\frac{1}{4}}\sqrt{|t|}.\]

Putting all of the above cases together proves the required lower bound on $S$ and finishes the proof of the proposition.
\end{proof}

\subsection{Proof of the sub-critical parts of the main theorems}\label{sec:proof of sub-critical parts of the main theorem}
The next corollary restates the sub-critical parts of Theorem~\ref{thm:d=1}, Theorem~\ref{thm:d=2} and Theorem~\ref{thm:d=3} in the notation of this section.
\begin{cor}\label{cor:connection to L_1}
In the (Sub-Critical) cases $\nu =0$ and the following holds as $N\to\infty$:
\begin{enumerate}[label=(\roman{*})]
\item\label{item:part_1_sub_critical_restatement}
In case (Sub1),
\[\frac{\theta ^2L_1}{2\var(X)\rho ^2}\overset{d}{\longrightarrow } \gama \left(\frac{1}{2},1\right).\]
\item\label{item:part_2_sub_critical_restatement}
In case (Sub2),
\[\frac{\alpha _c\log L_1}{\rho }\overset{d}{\longrightarrow } U[0,1].\]
\item\label{item:part_3_sub_critical_restatement}
In cases (Sub3), (SubConst1) and (SubConst2),
\[L_1\overset{d}{\longrightarrow } Y,\]
where $Y$ is the integer-valued random variable defined by
\begin{equation*}
  \mathbb P \left(Y=j\right)=\theta \rho _* ^{-1}\varphi ^{*j}(0) r_*^j,\quad j\ge 1
\end{equation*}
and
\begin{equation}\label{eq:r_*_cor_def}
  \text{$r_*$ is the unique number satisfying $0<r_*<1$ and $\sum_{j=1}^\infty \varphi ^{*j}(0) r_*^j=\rho_* \theta^{-1}$.}
\end{equation}
\end{enumerate}
\end{cor}

\begin{proof}
The corollary follows easily by substituting the estimates given in Theorem~\ref{thm:sub-critical integral}. Let $j_N$ be a sequence of integers satisfying $1\le j_N\le N$ and
\begin{equation}\label{eq:condition on j_N}
j_N=o\left(\min \left(\sqrt{a_N},L^2\right)\right),\quad N\to\infty.
\end{equation}
By Lemma~\ref{lem:distribution_of_ell_1,ell_2}, theorem~\ref{thm:sub-critical integral} and Corollary~\ref{cor:corollary on W}, uniformly in $1\le j\leq j_N$ with $\varphi^{*j}(0)\neq 0$,
\begin{equation}\label{eq:prob of cycle}
\mathbb P \left(L_1=j\right)=\frac{W_{L,j}}{N}\cdot \frac{H_{N-j}(L)}{H_N(L)}\sim \frac{\theta L^d \varphi^{*j}(0)}{N}\cdot r_N^j=\theta \rho ^{-1}\varphi^{*j}(0) r_N^j,\quad N\to \infty.
\end{equation}
Note that, by Lemma~\ref{lem:asymptotics of kappa _j}, there are only finitely many $j$ for which $\varphi^{*j}(0)=0$, and that $L^d\varphi^{*j}(0)\ge c$ when $\varphi^{*j}(0)\neq0$ and $j\le L^2$. The above arguments also adapt to show that $\mathbb P \left(L_1=j\right)\to 0$ when $\varphi^{*j}(0)=0$.

In cases (Sub3), (SubConst1) and (SubConst2), one first verifies using \eqref{eq:g_Taylor_expansion} that the definition of $r_*$ given by \eqref{eq:r_*_cor_def} coincides with the one given in Lemma~\ref{lem:find r_N}. It then follows from \eqref{eq:prob of cycle} with fixed $j$ and Lemma~\ref{lem:find r_N} that $L_1\overset{d}{\longrightarrow }Y$, proving part~\ref{item:part_3_sub_critical_restatement}

Consider the case (Sub1). Denote $\sigma:=\sqrt{\var(X)}$. Let $0<a<b<\infty$ and set $j_N=\lfloor \frac{2\sigma ^2\rho ^2b}{\theta ^2}\rfloor$. Observe that $j_N$ satisfies \eqref{eq:condition on j_N} by Lemma~\ref{lem:a_N_b_N_estimates}. Thus, we may use \eqref{eq:prob of cycle} to obtain
\begin{equation}
\begin{split}
\mathbb P \left(a\le \frac{\theta ^2L_1}{2\sigma ^2\rho ^2}\le b\right)=&\sum _{j=\lceil\frac{2\sigma ^2\rho ^2a}{\theta ^2} \rceil}^{\lfloor\frac{2\sigma ^2\rho ^2b}{\theta ^2} \rfloor}\mathbb P \left(L_1=j\right)\sim \theta \rho ^{-1}\sum _{j=\lceil\frac{2\sigma ^2\rho ^2a}{\theta ^2} \rceil}^{\lfloor\frac{2\sigma ^2\rho ^2b}{\theta ^2} \rfloor}\varphi^{*j}(0) r_N^j \\
\sim \frac{\theta ^2}{2\sqrt{\pi}\sigma ^2\rho ^2}&\sum _{j=\lceil\frac{2\sigma ^2\rho ^2a}{\theta ^2} \rceil}^{\lfloor\frac{2\sigma ^2\rho ^2b}{\theta ^2} \rfloor}\left(\frac{\theta ^2j}{2\sigma ^2\rho ^2}\right)^{-\frac{1}{2}}\exp \left(-\frac{\theta ^2j}{2\sigma ^2\rho ^2}\right)\to \frac{1}{\sqrt{\pi}}\intop _a^b x^{-\frac{1}{2}}e^{-x}dx
\end{split}
\end{equation}
where in the second asymptotic equality we used Lemma~\ref{lem:asymptotics of kappa _j}, Lemma~\ref{lem:find r_N} and $1-x=e^{-x+o(x^2)}$ as $x\to 0$ and the final limit is obtained by convergence of the Riemann sum to the integral. This finishes the proof of part~\ref{item:part_1_sub_critical_restatement}.

Consider the case (Sub2). Let $0<a<b<1$ and set $j_N=e^{\frac{b\rho }{\alpha _c}}$. Note that we have $j_N=o\left(\left(1-r_N\right)^{-1}\right)$ by Lemma~\ref{lem:find r_N} and that $j_N$ satisfies \eqref{eq:condition on j_N} by Lemma~\ref{lem:a_N_b_N_estimates}. Thus, we may use \eqref{eq:prob of cycle} to obtain
\[ \mathbb P \left(a\le \frac{\alpha _c\log L_1}{\rho}\le b\right)=\! \sum _{j=\lceil e^{\frac{a\rho}{\alpha _c}}\rceil}^{\lfloor e^{\frac{b\rho}{\alpha _c}}\rfloor}\mathbb P \left(L_1=j\right)\sim \theta  \rho ^{-1}\sum _{j=\lceil e^{\frac{a\rho}{\alpha _c}}\rceil}^{\lfloor e^{\frac{b\rho}{\alpha _c}}\rfloor} \! \varphi^{*j}(0) r_N^j\sim \alpha _c \rho ^{-1}\sum _{j=\lceil e^{\frac{a\rho}{\alpha _c}}\rceil}^{\lfloor e^{\frac{b\rho}{\alpha _c}}\rfloor} \frac{1}{j} \to b-a,\]
where the second asymptotic equality follows from Lemma~\ref{lem:asymptotics of kappa _j} and from $r_N^{j_N}\to 1$. This finishes the proof of part~\ref{item:part_2_sub_critical_restatement}.

The fact that $\nu=0$ follows from the above results as, for any $\epsilon >0$, $\mathbb P \left(L_1\ge  \epsilon
N\right)\to 0$, as $N\to \infty$.
\end{proof}
\begin{remark}\label{remark:asymptotically_IID}
  The analysis in Corollary~\ref{cor:connection to L_1} extends to the study of the joint distribution of $L_1, L_2, \ldots$. The key fact is that an analog of \eqref{eq:prob of cycle} remains valid. By Lemma~\ref{lem:distribution_of_ell_1,ell_2}, for any fixed $m$ as $N\to\infty$,
  \begin{multline*}
    \mathbb P\left(L_1 = j_1,\dots, L_m = j_m\right)\\
    =\frac{H_{N-j_1-\dots-j_m}(L)}{H_N(L)}\cdot \prod_{k=1}^m \frac{W_{L,j_k}}{N-j_1-\dots-j_{k-1}}\sim\mathbb P\left(L_1 = j_1\right)\cdots\mathbb P\left(L_m = j_m\right),
  \end{multline*}
  uniformly in $1\le j_1+\cdots+j_m\le j_N$ with $\varphi^{*j_k}(0)\neq 0$ for all $k$, as follows from theorem~\ref{thm:sub-critical integral} and Corollary~\ref{cor:corollary on W} in the same manner as in \eqref{eq:prob of cycle}. One may then follow the analogous steps to the analysis in Corollary~\ref{cor:connection to L_1} and deduce that the $(L_k)$ become asymptotically independent and identically distributed, in the sense explained in the remark following Theorem~\ref{thm:d=3}.
\end{remark}

\section{The supercritical case}\label{sec:super-sritical}
In this section we prove part~\ref{item:super-critical in d=1} in Theorem~\ref{thm:d=1}, part~\ref{item:super-critical in d=2} in Theorem~\ref{thm:d=2} and part~\ref{item:super-critical in d=3} in Theorem~\ref{thm:d=3}. In particular we show the convergence to the Poisson-Dirchlet distribution.

\setlist[description]{font=\normalfont}
Let $\rho = \rho(N)$ and consider several possibilities for the
dimension $d$ and the asymptotic regime of $\rho (N)$ as
$N\to\infty$ corresponding to the various statements in the
theorems. We give names to these cases to simplify later reference.
By the name (Super-Critical) we refer collectively to any of the following
cases:
\begin{description}[leftmargin=!,labelwidth=75pt]
\item[(Super1)] Dimension $d=1$ and density $\rho$ satisfying $\rho =\omega \left(\sqrt N\right)$ and $\rho\le N$.
\item[(Super2)] Dimension $d=2$ and density $\rho $ satisfying $\frac{\rho}{\log N}\rightarrow \alpha\in(\alpha_c, \infty)$.
\item[(Hyper2)]
Dimension $d=2$ and density $\rho$ satisfying $\rho =\omega \left(\log N\right)$ and $\rho\le N$.
\item[(Super3)]
Dimension $d\ge 3$ and density $\rho >\rho _c$, fixed as $N\to \infty$.
\end{description}
Our assumption that $\rho\le N$ in all cases is equivalent to taking the side length $L$ of $\Lambda$ to be at least $1$, as we assumed in the estimates of Section~\ref{sec:properties}.
Define the auxiliary parameter $\tau=\tau_{d,\rho}$ in the cases above:
\begin{equation}\label{eq:def of tau}
\tau =\tau _{d,\rho }:=\left\{ \begin{array}{ll}
 0 & \mbox{in cases (Super1) and (Hyper2)} \\
\\
\frac{\alpha _c}{\alpha } & \mbox{in case (Super2)}\\
\\
\frac{\rho _c}{\rho } & \mbox{in case (Super3)}\\
\end{array}.
\right.
\end{equation}
We will show that, in all of the (Super-Critical) cases, $\tau$ is the fraction of points in non-macroscopic cycles, or equivalently $\tau=1-\nu  $ where $\nu $ is defined in \eqref{eq:def of nu}.

As in the sub-critical case, our strategy is to apply the Cauchy integral formula to obtain the asymptotic behavior of $H_n(L)=\left[z^n\right]e^{G_L(z)}$ and then use the results of Section~\ref{sec:combinatorics} which connect $H_n(L)$ with the distribution of the cycle lengths $\left(L_1,L_2,\dots \right)$.

\subsection{Singularity analysis}\label{sec:singularity}

In the following theorem we find the asymptotic behavior of $H_{N-j}(L)$. We first introduce an analytic function $F_L$ defined by
\begin{equation}\label{eq:def of F}
F_L(z):=G_L(z)+\theta \log \left(1-z\right)=-\theta \sum _{m\in \frac{1}{L}\Z ^d\setminus\{0\}} \log \left(1-z\hat {\varphi}(m)\right),\quad z\in \Delta _0,
\end{equation}
where the second equality uses \eqref{eq:analytic continuation form G_L} and where we recall from \eqref{eq:Delta_0_def} that
\begin{equation*}
  \Delta_0 = \left\{z\in \mathbb C\setminus\{1\}\,\colon|z|<R_0,
\;\left|\arg(z-1)\right|>\frac{\pi}{8} \right\},
\end{equation*}
with $R_0>1$ determined by Claim~\ref{claim:1-z phi}. For every $L\ge 1$, the sum in \eqref{eq:def of F} converges uniformly in $\overline{\Delta _0}$, by Claim~\ref{claim:1-z phi} and the fact that $\hat{\varphi}$ is a Schwartz function, and therefore $F_L$ is continuous at $z=1$. Moreover, by the definition of $G_L$ in \eqref{eq:def of G_L(z)}, in the closed unit disc $F_L$ is given by the power series
\begin{equation}\label{eq:power series for F_L}
F_L(z)=\sum _{j=1}^\infty \frac{W_{L,j}-\theta }{j}z^j,\quad z\in \overline{\mathbb{D}},
\end{equation}
where the sum converges absolutely in $\overline{\mathbb{D}}$ by Corollary~\ref{cor:corollary on W}. In particular $F_L(1)$ is real.

\begin{thm}\label{thm:supercritical integral}
In the (Super-Critical) cases, for every $\epsilon >0$,
\[H_{N-j}(L)=\left[z^{N-j}\right] e^{G_L(z)} \sim \frac{e^{F_L(1)}N^{\theta-1 }\left(1-\tau -\frac{j}{N}\right)^{\theta -1}}{\Gamma \left(\theta \right)}, \quad N\to \infty,
 \]
uniformly in $j\le \left(1-\tau -\epsilon\right)N$.

\end{thm}

The proof of Theorem~\ref{thm:supercritical integral}, following \cite{bogachev2015asymptotic}, proceeds by expressing $H_{N-j}(L)$ as a contour integral and changing the contour of integration to a `pacman-shaped' contour $\gamma$ defined as follows (see Figure~\ref{fig_gamma}):  $\gamma=\gamma \left(\eta ,\beta,N \right) =\gamma_1+\gamma_2+\gamma_3+\gamma_4$, where $+$ denotes concatenation of contours,
\begin{equation}\label{eq:def of gamma}
\begin{split}
&\gamma _1(t)=1-te^{-i\beta }, \quad t\in \left[-\eta ,-\frac{1}{N}\right]\\
&\gamma_2(t)=1+\frac{1}{N}e^{-it}, \quad t\in \left[\beta , 2\pi -\beta \right]\\
&\gamma_3(t)=1+te^{i\beta }, \quad t\in \left[\frac{1}{N},\eta  \right]\\
&\gamma _4(t)=R_1e^{it}, \quad t\in \left[\beta ',2\pi -\beta '\right]
\end{split}
\end{equation}
and $R_1=\left|1+\eta e^{i\beta }\right| $,  $\beta '=\arg \left(1+\eta e^{i\beta }\right)$.
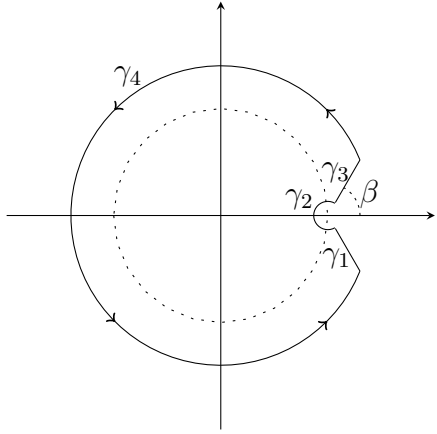
\begin{figure}
\begin{tikzpicture}
\begin{axis}[ticks=none,axis lines=middle,xmin=-2,xmax=2,ymin=-2,ymax=2,axis equal image]
\draw [dash pattern=on 1pt off 3pt, domain=-180:180] plot({cos(\x)},{sin(\x)});
\addplot[domain=1.07:1.3]{tan(60)*x-tan(60)};
\addplot[domain=1.07:1.3]{-tan(60)*x+tan(60)};
\draw [samples =70, domain=atan(0.3*tan(60)/1.3):360-atan(0.3*tan(60)/1.3)] plot({((0.3*tan(60))^2+1.3^2)^0.5*cos(\x)},{((0.3*tan(60))^2+1.3^2)^0.5*sin(\x)});
\draw [domain=60:300] plot({1+0.08*((1+tan(60))^0.5)*cos(\x)},{0.08*((1+tan(60))^0.5)*sin(\x)});
\node[text width=0cm] at (1.3,0.2)
    {$\beta $};
    \draw [dash pattern=on 1pt off 2pt, domain=0:60] plot({1+0.3*cos(\x)},{0.3*sin(\x)});
\node[text width=0cm] at (0.95,-0.4)
{$\gamma _1 $};
\node[text width=0cm] at (0.6,0.15)
{$\gamma _2 $};
\node[text width=0cm] at (0.94,0.4)
{$\gamma _3 $};
\node[text width=0cm] at (-1,1.3)
{$\gamma _4 $};
\draw[->,thick] (-0.99,0.99) -- (-0.99-0.001,0.99-0.001) node [pos=0.66,above]{} ;
\draw[->,thick] (-0.99,-0.99) -- (-0.99+0.001,-0.99-0.001) node [pos=0.66,above]{} ;
\draw[->,thick] (0.99,0.99) -- (0.99-0.001,0.99+0.001) node [pos=0.66,above]{} ;
\draw[->,thick] (0.99,-0.99) -- (0.99+0.001,-0.99+0.001) node [pos=0.66,above]{} ;
\end{axis}
\end{tikzpicture}
\caption{The contour $\gamma=\gamma (\eta ,\beta ,N)$.}
 \label{fig_gamma}
\end{figure}
We will see that the main contribution to the Cauchy integral comes from a small neighborhood of $1$. In the proof of Theorem~\ref{thm:supercritical integral} we estimate this contribution using the following proposition that identifies the behavior of $G_L$ close to $1$ and on the rest of $\gamma$.

\begin{prop}\label{prop:super-critical prop}
In the (Super-Critical) cases, for any $\epsilon >0$ there exists $0<\eta<\frac{1}{10}$ and $N_0>0$ such that $\gamma =\gamma (\eta ,\frac{\pi}{3},N )$ satisfies $\gamma\subseteq \Delta_0$ for $N\ge N_0$ and the following holds:
\begin{enumerate}[label=(\roman{*})]
\item\label{item:part one of super-critical prop}
There exists a sequence $t_N\rightarrow \infty $, $t_N=o(\sqrt{N})$, for which
\[G_L(z)=-\theta \log \left(1-z\right)+F_L(1)+\tau N\left(z-1\right)+o(1),\quad  N\to \infty,\]
uniformly in $z\in \Delta _0$ such that $\frac{1}{N}\le \left|1-z\right|\le \frac{t_N}{N}$.
\item\label{item:part two of super-critical prop}
For any $z\in \gamma _1\cup \gamma _3$ and $N\ge N_0$,
\[\real \left(G_L(z)\right)\le -\theta \log \left|1-z\right|+F_L(1)+\left(\tau +\epsilon \right)N\real \left(z-1\right).\]
\item\label{item:part three of super-critical prop}
For any $z\in \gamma _4$ and $N\ge N_0$,
\[\real \left(G_L(z)\right)\le F_L(1)+(\tau +\epsilon )N\log |z|.\]
\end{enumerate}
\end{prop}

We start by proving Theorem~\ref{thm:supercritical integral} given the proposition. We require also the following well-known Hankel integral \cite[Theorem 8.4b]{henrici1991applied}:
For $\beta \in (0,\pi )$, define the contour  $\gamma '=\gamma '(\beta )=\gamma '_1+\gamma '_2+\gamma '_3$ by
\begin{equation}\label{eq:def of gamma '}
\begin{split}
&\gamma '_1(t):=-te^{-i\beta }, \quad t\in \left(-\infty,-1\right],\\
&\gamma '_2(t):=e^{-it }, \quad t\in \left[\beta ,2\pi -\beta \right],\\
&\gamma '_3(t):=te^{i\beta }, \quad t\in \left[1,\infty\right),
\end{split}
\end{equation}
see Figure~\ref{fig_gamma '}.
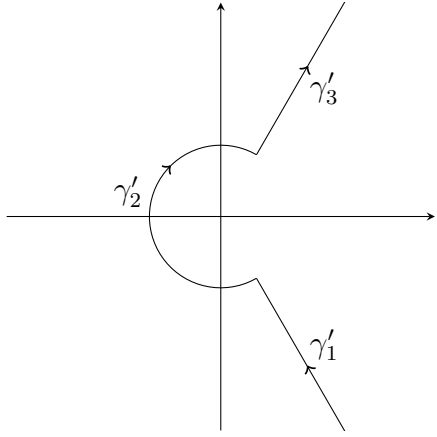
\begin{figure}
\begin{tikzpicture}
\begin{axis}[ticks=none,axis lines=middle,xmin=-3,xmax=3,ymin=-3,ymax=3,axis equal image]
\addplot[domain=0.5:3]{tan(60)*x};
\addplot[domain=0.5:3]{-tan(60)*x};
\draw [samples=50, domain=60:300] plot({cos(\x)},{sin(\x)});
\node[text width=0cm] at (1.25,-1.8)
{$\gamma _1' $};
\node[text width=0cm] at (-1.5,0.35)
{$\gamma _2' $};
\node[text width=0cm] at (1.25,1.8)
{$\gamma _3' $};
\draw[->,thick] (1.217,2.1) -- (1.217+1/200,2.1+3^0.5/200) node [pos=0.66,above]{} ;
\draw[->,thick] (1.217,-2.1) -- (1.217-1/200,-2.1+3^0.5/200) node [pos=0.66,above]{} ;
\draw[->,thick] (-2^0.5/2,2^0.5/2) -- (-2^0.5/2+1/200,2^0.5/2+1/200) node [pos=0.66,above]{} ;
\end{axis}
\end{tikzpicture}
\caption{The contour $\gamma '(\frac{\pi}{3})$.}
\label{fig_gamma '}
\end{figure}
Then, if $\beta \in (0,\frac{\pi }{2})$, we have
\begin{equation}\label{eq:Hankel integral}
\frac{1}{2\pi i}\intop _{\gamma '} (-\omega )^{-s} e^{-\omega } d\omega=\frac{1}{\Gamma (s)},\quad s\in \C.
\end{equation}

\begin{proof}[Proof of Theorem~\ref{thm:supercritical integral}]
Fix $0<\epsilon <\frac{1-\tau }{2}$ where $\tau $ is given in \eqref{eq:def of tau}. Let $\eta $ and $N_0$ be the numbers and $(t_N)$ be the sequence from Proposition~\ref{prop:super-critical prop} corresponding to $\epsilon $. Set $\gamma :=\gamma \left(\eta ,\frac{\pi}{3},N\right)$ where $N\ge N_0$ is taken large enough for the following calculations. We assume throughout this proof that $0\le j\le (1-\tau -2\epsilon )N$.

By the Cauchy integral formula
\[\left[z^{N-j}\right]e^{G_L(z)}=\frac{1}{2\pi i}\ointop _\gamma \frac{e^{G_L(z)}}{z^{N+1-j}}dz=I_1+I_2+I_3,\]
where $I_1$, $I_2$ and $I_3$ are the corresponding integrals over $\gamma \cap \{|1-z|\le \frac{t_N}{N} \}$, $\left(\gamma _1\cup \gamma _3 \right)\cap \{|1-z|\ge \frac{t_N}{N}\}$ and $\gamma _4$ respectively.
We estimate separately the three parts.

We start with $I_1$. The following holds as $N\to\infty$, uniformly in $0\le j\le (1-\tau -2\epsilon )N$ and $z\in \Delta_0 \cap \{\left|1-z\right|\le \frac{t_N}{N}\}$. We have
\begin{equation*}
z^{-\left(N+1-j\right)}=\exp\left[-\left(N+1-j\right)\log z\right]=\exp \left[-\left(N-j\right)\left(z-1\right)+o(1)\right],
\end{equation*}
where we used that $t_N=o(\sqrt{N})$.
Substituting this and part~\ref{item:part one of super-critical prop} of Proposition~\ref{prop:super-critical prop} in $I_1$ we obtain that
\begin{equation}\label{eq:main contribution}
I_1=\frac{e^{F_L(1)}}{2\pi i}\intop _{\gamma \cap \{|1-z|\le \frac{t_N}{N} \} } \left(1-z\right)^{-\theta }\exp\left[\left(\tau N-N+j\right)\left(z-1\right)+o(1)\right]dz,
\end{equation}
We make the change of variables $z=1+\frac{\omega }{N(1-\tau)-j}$, so that $(\tau N-N+j)(z-1)=-\omega $, to obtain
\[I_1= \frac{e^{F_L(1)}\left(N(1-\tau) -j\right)^{\theta -1}}{2\pi i}\intop _{\tilde{\gamma } }\left(-\omega \right)^{-\theta} e^{-\omega+o(1) }d\omega, ,\]
where $\tilde{\gamma }$ is the image of $\gamma \cap \{|1-z|\le \frac{t_N}{N} \}$ under the change of variables, which is a modification of the $\gamma '(\frac{\pi }{3})$ from \eqref{eq:def of gamma '} having the circular arc at radius $1-\tau -\frac{j}{N}\ge 2\epsilon $ and having finite `arms', terminating at radius $t_N(1-\tau -\frac{j}{N})
\ge 2\epsilon t_N\to \infty $. As the integral in \eqref{eq:Hankel integral} with $s=\theta$ and $\beta =\frac{\pi }{3}$ converge to a non-zero quantity, we conclude that
\begin{equation}\label{eq:expression}
I_1\sim \frac{e^{F_L(1)}N^{\theta-1 }\left(1-\tau -\frac{j}{N}\right)^{\theta -1}}{\Gamma \left(\theta \right)}.
\end{equation}

We turn to bound $I_2$. A Taylor expansion shows that
\begin{equation}\label{eq:expanding log}
\left|\real (\log z)-\real (z-1)+\frac{1}{2}\real[(z-1)^2]\right|\le \frac{1}{8}|z-1|^3,\quad |z-1|\le \frac{1}{10}.
\end{equation}
Thus, for $z\in \gamma _1\cup \gamma _3$,
\[\real \left(\log z\right)\ge \real z-1>0.\]
Thus, for such $z$ and using our assumption that $0\le j\le (1-\tau -2\epsilon )N$,
\[\left|z^{-\left(N-j+1\right)}\right|\le    \exp \left[-\left(\tau +2\epsilon \right)N\left(\real z-1\right)\right].\]
Substituting this and part~\ref{item:part two of super-critical prop} of Proposition~\ref{prop:super-critical prop} in $I_2$ we obtain that
\begin{equation}\label{eq:bound I_2 in supercritical}
\begin{split}
|I_2|\le e^{F_L(1)}\!\!\! \!\!\! \intop_{(\gamma _1\cup \gamma _3)\cap \{|1-z|\ge \frac{t_N}{N}\}} \!\!\!\!\! &\left|1-z\right|^{-\theta} \exp \left[-\epsilon \cdot N(\real z-1)\right]|dz| \\
&\le e^{F_L(1)}N^{\theta -1}\intop _{\gamma '\cap \{\left|\omega \right|\ge t_N\}}\left|\omega \right|^{-\theta }e^{-\epsilon \real \omega }=o(I_1),\quad N\to \infty,
\end{split}
\end{equation}
where in the second inequality we made the change of variables $z=1+\frac{\omega }{N}$ and used the contour $\gamma '$ defined in \eqref{eq:def of gamma '} and the final estimate follows from \eqref{eq:expression} and as the integral in \eqref{eq:bound I_2 in supercritical} tends to $0$. We bound $I_3$ using part~\ref{item:part three of super-critical prop} of Proposition~\ref{prop:super-critical prop} by
\[|I_3|\le e^{F_L(1)}\intop_{\gamma _4} |z|^{(\tau +\epsilon-1)N+j-1 }|dz|\le e^{F_L(1)}\intop_{\gamma _4} |z|^{-\epsilon N}|dz|\le 2\pi e^{F_L(1)}R_1^{-\epsilon N} ,\]
which is exponentially smaller than $I_1$ by \eqref{eq:expression}.
\end{proof}

The rest of Section~\ref{sec:singularity} is devoted to the proof of Proposition~\ref{prop:super-critical prop}.

\subsubsection{Near $z=1$} In this section we prove parts~\ref{item:part one of super-critical prop} and \ref{item:part two of super-critical prop} of Proposition~\ref{prop:super-critical prop}.

We first note that, by the definition \eqref{eq:def of F} of $F_L$,
\begin{equation}\label{eq:G=log +F}
G_L(z)=-\theta \log \left(1-z\right)+F_L(z)=-\theta \log \left(1-z\right)+F_L(1)+\intop_1^zF_L'(w)dw,\quad z\in \Delta _0.
\end{equation}
This representation will allow us to prove parts~\ref{item:part one of super-critical prop} and \ref{item:part two of super-critical prop} by suitably estimating $F_L'$. We start by writing $F_L'$ as a series, using again the definition \eqref{eq:def of F},
\begin{equation}\label{eq:F'}
F_L'(z)=\theta \sum_{m\in \frac{1}{L}\Z ^d \setminus \{0\}}\frac{\hat{\varphi}(m)}{1-z\hat{\varphi}(m)},\quad z\in \overline{\Delta _0}.
\end{equation}
We start with the following bound on $F_L'(z)$,
\begin{equation}\label{eq:bound F'}
\begin{split}
\left|F_L'(z)\right|&\le \theta \sum_{\substack{m\in \frac{1}{L}\Z^d  \\ 0<\|m\|\le 1 }}\left|\frac{\hat{\varphi}(m)}{1-z\hat{\varphi}(m)}\right|+\theta \sum_{\substack{m\in \frac{1}{L}\Z^d  \\ \|m\|> 1 }}\left|\frac{\hat{\varphi}(m)}{1-z\hat{\varphi}(m)}\right| \\
&\le C\sum_{\substack{m\in \frac{1}{L}\Z^d  \\ 0<\|m\|\le 1 }}\frac{1}{\|m\|^2}+C\sum_{m\in \frac{1}{L}\Z^d  }\left|\hat{\varphi}(m)\right|\le CL^d\intop _{\frac{1}{L}\le \|t\|\le 1 }\frac{1}{\|t\|^2}+CL^d,
\end{split}
\end{equation}
where in the second inequality we used Claim~\ref{claim:1-z phi}. Now it is relatively straightforward to deduce parts~\ref{item:part one of super-critical prop} and \ref{item:part two of super-critical prop} of Proposition~\ref{prop:super-critical prop} in the cases where $\tau=0$. Indeed, the last expression in (Super1) is $O\left(L^2\right)=O\left(\frac{N^2}{\rho ^2}\right)=o(N)$ and in (Hyper2) it is $O\left(L^2\log L\right)=O\left(\frac{N\log N}{\rho }\right)=o(N)$ as $N\to \infty$.
Therefore, by \eqref{eq:G=log +F}, we get that in both these cases,
\[G_L(z)=-\theta \log \left(1-z\right)+F_L(1)+o(N)\cdot |1-z|,\quad N\to \infty,\]
uniformly in $z\in \Delta _0$. This implies parts~\ref{item:part one of super-critical prop} and \ref{item:part two of super-critical prop} of the proposition in these cases (the fact that the $o(N)$ term is uniform in $z\in \Delta _0$ implies the existence of a sequence $t_N\to \infty$, $t_N=o(\sqrt{N})$ satisfying that $\frac{t_N}{N}$ times the $o(N)$ still tends to $0$).

We turn to prove parts~\ref{item:part one of super-critical prop} and \ref{item:part two of super-critical prop} in the cases (Super2) and (Super3). In these asymptotic regimes, it turns out that $F_L'$ is well-approximated by $L^d g'$, as the next lemma makes precise.
\begin{lemma}\label{lem:sum to integral}
There exists a $C>0$ so that for any $L\ge 1 $ and $z\in \Delta _0$ the following holds:
\begin{enumerate}[label=(\roman{*})]
\item
If $d=2$ then
\begin{equation}\label{eq:d=2 F'-g'}
\left|F_L'(z)-L^2g'(z)\right|\le CL^2 \max \left\{ 1,\log \left(\frac{L^{-2}}{|1-z|}\right) \right\}.
\end{equation}
\item
If $d\ge 3$ then
\begin{equation}\label{eq:d=3 F'-g'}
\left|F_L'(z)-L^dg'(z)\right|\le C\left(L^{d-1}+L^2\log L\right).
\end{equation}
\end{enumerate}
\end{lemma}

\begin{proof}
We approximate the Riemann sum $L^{-d}F_L'(z)$ by the corresponding integral (see \eqref{eq:F'} and \eqref{eq:analytic continuation for g'}). Decomposing the domain of integration to cubes around each $m\in \frac{1}{L}\Z ^d \setminus \{0\}$ we obtain
\begin{equation}\label{eq:each term F' g'}
\begin{split}
\left|F_L'(z)-L^dg'(z)\right| &= \left|\theta \!\! \sum_{m\in \frac{1}{L}\Z ^d \setminus \{0\}}\frac{\hat{\varphi}(m)}{1-z\hat{\varphi}(m)}-\theta L^d\intop _{\R ^d}\frac{\hat{\varphi}(t)}{1-z\hat{\varphi}(t)}dt\right|  \\
\le & \theta L^d\intop _{Q_L}\frac{|\hat{\varphi}(t)|}{|1-z\hat{\varphi}(t)|}dt+\theta L^d \!\! \sum_{\substack{m\in \frac{1}{L}\Z ^d \\ m\neq 0}} \ \intop _{m+Q_L}\!\! \frac{\left|\hat{\varphi}(m)-\hat{\varphi}(t)\right|}{\left|1-z\hat{\varphi}(m)\right|\cdot \left|1-z\hat{\varphi}(t)\right|}dt,
\end{split}
\end{equation}
where we set $Q_L:=\left[-(2L)^{-1}, (2L)^{-1} \right]^d\subseteq \R ^d$. We bound each term in \eqref{eq:each term F' g'} separately and start with the second one. For any $m\in \frac{1}{L}\Z ^d \setminus \{0\}$ and $t\in m+Q_L$ we have
\begin{equation}
\left|\hat{\varphi}(m)-\hat{\varphi}(t)\right|\le CL^{-1} \max _{x\in m+Q_L} \max _i \left|\frac{\partial \hat{\varphi}}{\partial t_i}(x)\right| \le CL^{-1}  \min \{ \|t\|, \|t\|^{-d-1} \},
\end{equation}
where the second inequality follows as $\frac{\partial \hat{\varphi}}{\partial t_i}$ decays fast, $\frac{\partial \hat{\varphi}}{\partial t_i}(0)=0$ and $\frac{\partial ^2 \hat{\varphi}}{\partial t_i^2}$ is bounded. Thus, by part~\ref{item:part 1 of claim 1-z phi} of Claim~\ref{claim:1-z phi} we obtain
\begin{equation*}
\begin{split}
L^d \!\! \sum_{\substack{m\in \frac{1}{L}\Z ^d \\ m\neq 0}} \ \intop _{m+Q_L}\!\! \frac{\left|\hat{\varphi}(m)-\hat{\varphi}(t)\right|}{\left|1-z\hat{\varphi}(m)\right|\cdot \left|1-z\hat{\varphi}(t)\right|}dt &\le CL^{d-1}\!\!\! \intop _{L^{-1}\le \|t\|\le 1} \!\! \|t\|^{-3}dt+CL^{d-1}\intop _{ \|t\|\ge 1} \|t\|^{-d-1}dt \\
&\le CL^{d-1} \intop _{L^{-1}}^{1} r^{d-4}dr+ CL^{d-1 }
\end{split}
\end{equation*}
and we check that the last expression is bounded by the right-hand side of \eqref{eq:d=2 F'-g'} when $d=2$ and by the right-hand side of \eqref{eq:d=3 F'-g'} when $d\ge 3$.

Next, we bound the first term on the right-hand side of \eqref{eq:each term F' g'}, again using Claim~\ref{claim:1-z phi} and using also that $|\hat{\varphi}(t)|\le 1$. When $d= 2$,
\[\intop _{Q_L}\frac{|\hat{\varphi}(t)|}{|1-z\hat{\varphi}(t)|}dt\le C\!\! \intop _{\|t\|\le CL^{-1}}\frac{dt}{|1-z|+\|t\|^2}\le C \intop _0^{CL^{-2}} \frac{dr}{|1-z|+r} \le C\log \left(\frac{|1-z|+CL^{-2}}{|1-z|}\right).\]
When $d\ge 3$,
\[\intop _{Q_L}\frac{|\hat{\varphi}(t)|}{|1-z\hat{\varphi}(t)|}dt\le C\intop _{\|t\|\le CL^{-1}}\frac{1}{\|t\|^2}dt\le C \intop _0^{CL^{-1}}r^{d-3}dr\le \frac{C}{L^{d-2}}.\]
Substituting these bounds in \eqref{eq:each term F' g'} finishes the proof of the lemma.
\end{proof}

We proceed to deduce parts (i) and (ii) of Proposition~\ref{prop:super-critical prop} for the cases (Super3) and (Super2). Consider first the case (Super3). By the definitions of $\tau, g$ and $\rho $ in \eqref{eq:def of tau},\eqref{eq:g_Taylor_expansion} and \eqref{eq:def of rho} respectively, we have that $\tau N=L^dg'(1)$. Thus, using Lemma~\ref{lem:sum to integral} and Lemma~\ref{lem:asymptotics of g at 1} we get that for any $z\in \Delta _0$,
\begin{equation*}
\begin{split}
\left| \intop _1^z F_L'(w)dw -\tau N(z-1) \right|&\le \left| \intop _1^z F_L'(w)dw -\intop _1^z L^dg'(1)dw \right| \\
&\le \intop _1^z |F_L'(w)-L^dg'(w)|\cdot |dw| +L^d\intop _1^z |g'(w)-g'(1)|\cdot |dw| \\
&\le CL^{d-1}\log L \cdot |z-1|+CL^d |z-1|^{\frac{3}{2}}.
\end{split}
\end{equation*}
Parts~\ref{item:part one of super-critical prop} and \ref{item:part two of super-critical prop} of the proposition follows in case (Super3) by substituting the last result in \eqref{eq:G=log +F} and taking $t_N=\sqrt{L}$ and sufficiently small $\eta >0$.

Finally, consider the case (Super2). We estimate the integral on the right-hand side of \eqref{eq:G=log +F} as follows. Let $z\in \Delta _0$ with $|1-z|\ge N^{-1}$ and let $x\in \Delta _0$ be the unique point on the line segment connecting $1$ and $z$ satisfying $|1-x|=(N\log N)^{-1}$. We have,
\begin{equation*}
\begin{split}
&\left|\intop _1^z \! F_L'(w)dw + \alpha _c L^2 (z-1)\log\left(1-z\right) \right| \le \left|\intop _1^z \! F_L'(w)dw + \alpha _c L^2 \!\! \intop_1^z \! \log\left(1-w\right)dw \right|+CL^2|1-z| \\
&\le \! \intop _1^x |F_L'(w)||dw|+CL^2\! \intop _1^x \! |\log (1-w)||dw|
+\! \intop_x^z \! |F_L'(w)+\alpha _c L^2 \log (1-w)||dw|+CL^2|1-z|,
\end{split}
\end{equation*}
where $\alpha _c$ is defined in \eqref{eq:def of alpha _c }. We bound the first and the third integrals separately. To bound the first integral, note that for any $w\in \Delta _0$, we have by \eqref{eq:bound F'}, $|F_L'(w)| \le CL^2\log L\le CN$. To bound the third integral observe that by Lemma~\ref{lem:sum to integral} and Lemma~\ref{lem:asymptotics of g at 1}, for $w\in \Delta _0$ with $|1-w|\ge (N\log N)^{-1}$ and large enough $N$, we have
\begin{equation*}
|F_L'(w)+\alpha _cL^2 \log (1-w)|\le |F_L'(w)-L^2g'(w)|+L^2|g'(w)+\alpha _c \log (1-w)|\le \frac{N}{\sqrt{\log N}}.
\end{equation*}
Therefore, for $z\in \Delta _0$ with $|1-z|\ge N^{-1}$ and sufficiently large $N$,
\begin{equation}
\left|\intop _1^z \! F_L'(w)dw + \alpha _c L^2 (z-1)\log\left(1-z\right) \right| \le \frac{N}{\sqrt{\log N}}|1-z|.
\end{equation}
Part~\ref{item:part one of super-critical prop} of the proposition follows since, as $N\to \infty$,
\[-\alpha _c L^2 (z-1)\log\left(1-z\right) = \alpha _cL^2\log N(z-1)+ o(N)|1-z|= \tau N(z-1)+o(N)|1-z|,\]
uniformly in $z\in \Delta _0$ with $\frac{1}{N}\le |1-z|\le \frac{\log N}{N}$, where in the last asymptotic equality we used that $\frac{\rho }{\log N}\to \alpha $ as $N\to \infty$ and that $\tau =\frac{\alpha_c }{\alpha } $. Part~\ref{item:part two of super-critical prop} of the proposition follows as, for large enough $N$ and $z\in \gamma _1\cup \gamma _2$,
\[\real \left(-\alpha _c L^2 (z-1)\log\left(1-z\right)\right)\le \alpha _c L^2 \log N \real (z-1) +CL^2|1-z|\le (\tau +\epsilon )N\real (z-1).\]

\subsubsection{Far from $z=1$} In this section we prove the third part of Proposition~\ref{prop:super-critical prop}.

We start with dimensions $1$ and $2$ whose treatment is relatively straightforward. We claim that
\begin{equation}\label{eq:third_part_stronger_estimate_dim_1_2}
|F_L(z)|\le CL^d,\quad z\in \overline{\Delta _0}.
\end{equation}
This implies part (iii) of the proposition in cases (Super1), (Super2) and (Hyper2) as, for any $z\in \gamma _4$,
\[\real G_L(z)\le \left|G_L(z)\right|=\left|-\theta \log \left(1-z\right)+F_L(z)\right|\le C L^d\le F_L(1)+C L^d \le F_L(1)+\epsilon N\log |z|,\]
where the last inequality holds for large enough $N$ since $L^d=o(N)$ in these cases.

To see \eqref{eq:third_part_stronger_estimate_dim_1_2}, fix $C_0>0$ such that
\[|\hat{\varphi}(t)|\le \frac{C_0}{R_0 \|t\|^{d+1}},\quad t\in \R ^d,\]
where $R_0$ is the radius from the definition of $\Delta _0$ in \eqref{eq:Delta_0_def}. We have, for $z\in \overline{\Delta _0}$,
\begin{equation*}
\left|F_L(z)\right| \le \sum _{m\in \frac{1}{L}\Z ^d\setminus\{0\}}\left| \log \left(1-z\hat {\varphi}(m)\right)\right|\le C \sum _{\substack{m\in \frac{1}{L}\Z ^d \\ 0<\|m\|\le 2C_0 }}(1+\left|\log (\|m\|)\right|)+\sum _{\substack{m\in \frac{1}{L}\Z ^d \\ \|m\|\ge 2C_0 }} \frac{C}{\|m\|^{d+1}},
\end{equation*}
where in the second inequality we used that $c\|m\|^2\le |1-z\hat{\varphi}(m)|\le C$ for $\|m\|\le 2C_0$ by Claim~\ref{claim:1-z phi} and that $|\log (1+x)|\le C|x|$ for $|x|\le \frac{1}{2}$. One can easily check that the last expression is at most $CL^d$
which finishes the proof of part (iii) of Proposition~\ref{prop:super-critical prop} in the cases (Super1), (Super2) and (Hyper2).

We turn to prove part (iii) of Proposition~\ref{prop:super-critical prop} in the remaining case (Super3). We start with the following lemma which bounds the real part of the function $g$.

\begin{lemma}\label{lem:bound on Re g(z)}
Let $d\ge 3$. There is an $\eta >0$ such that for any $z\in \gamma_4=\gamma _4\left(\eta ,\frac{\pi}{3}\right)$ we have
\begin{equation}\label{eq:Re g(z)<g(1)+g'(1)+log|z|}
\real g(z)\le g(1)+g'(1)\log |z|=g(1)+g'(1)\log R_1.
\end{equation}
\end{lemma}

\begin{proof}
We use different arguments for $z$ close to $1$ and far from $1$. First we claim that there exists a $\delta >0$ such that the inequality \eqref{eq:Re g(z)<g(1)+g'(1)+log|z|} holds in
\[\Omega _\delta =\left\{z\in \Delta _0 \ \big|\  \left|1-z\right|\le \delta , \ \left|\arg \left(z-1\right)\right|\ge \frac{\pi }{3},\ |z|\ge 1\right\}.\]
For this we expand $g(z)$ around $z=1$. We have
\[g(z)=g(1)+g'(1)(z-1)+\intop_1^z \intop _1^w g''(x)dxdw\]
and therefore, by Lemma~\ref{lem:asymptotics of g at 1}, we have the following asymptotics as $z\to 1$ according to the dimension~$d$,
\[g(z)=g(1)+g'(1)\left(z-1\right)+c_1\left(1-z\right)^{\frac{3}{2}}\left(1+o(1)\right),\quad d=3,\]
\[g(z)=g(1)+g'(1)\left(z-1\right)-c_2\log \left(1-z\right)\left(z-1\right)^2\left(1+o(1)\right),\quad d=4,\]
\[g(z)=g(1)+g'(1)\left(z-1\right)+c_3\left(z-1\right)^2\left(1+o(1)\right),\quad d\ge 5,\]
for some positive constants $c_1,c_2$ and $c_3$. Thus, for sufficiently small $\delta >0$ and $z\in \Omega _\delta $ we have
\[\real g(z)\le g(1)+g'(1)\left(\real z-1\right)\le g(1)+g'(1)\log |z|,\]
where in the second inequality we use that
\[\log |z|=\real \left(\log z\right)= \real z-1-\frac{1}{2}\real \left(z-1\right)^2+O\left(z-1\right)^3,\quad z\in \Delta _0.\]

Fix $\delta >0$ such that \eqref{eq:Re g(z)<g(1)+g'(1)+log|z|} holds in $\Omega _\delta $. As $g$ has non-negative Taylor coefficients (see \eqref{eq:g_Taylor_expansion}), positive with finitely many exceptions by Lemma~\ref{lem:asymptotics of kappa _j}, it follows that for any $z\neq 1$ with $|z|\le 1$,
\begin{equation}\label{eq:Re g(z)<g(1)}
\real g(z)<g(1).
\end{equation}
Thus, by continuity and compactness arguments, there is $R _\delta >1$ such that the inequality \eqref{eq:Re g(z)<g(1)} holds in
\[D  :=\left\{z\in \Delta _0 \ \big| \ |z|\le R_\delta , \ |1-z|\ge \delta \right\}.\]
Now, the fact that for sufficiently small $\eta >0$,
\[\gamma _4=\gamma _4(\eta ,\frac{\pi}{3})\subseteq \Omega _\delta \cup D,\]
completes the proof of the lemma.
\end{proof}

We can now finish the proof of Proposition~\ref{prop:super-critical prop} for the case (Super3). We note that $F_L(0)=G_L(0)=g(0)=0$. By Lemma~\ref{lem:sum to integral}, for large enough $N$ and $z\in \overline{\Delta _0}$, we have
\begin{equation}\label{eq:F_L-L^dg}
\left|F_L(z)-L^dg(z)\right|\le \intop _0^z\left|F_L'(w)-L^dg'(w)\right|\cdot |dw|\le C(L^{d-1}+L^2 \log L)<\frac{\epsilon }{3}\log (R_1) N.
\end{equation}
Thus, recalling that $\tau N=L^dg'(1)$, we get that for large enough $N$ and $z\in \gamma _4$,
\begin{equation}
\begin{split}
\real G_L(z)&\le \real F_L(z)+C_\eta \le L^d \real g(z)+\frac{2\epsilon }{3}\log (R_1)N \\ &\le L^dg(1)+L^dg'(1)\log |z|+\frac{2\epsilon }{3}\log (R_1)N \le F_L(1)+(\tau+\epsilon ) N \log |z|,
\end{split}
\end{equation}
where in the second and fourth inequalities we used \eqref{eq:F_L-L^dg} and in the third one we used Lemma~\ref{lem:bound on Re g(z)}.
This implies the third part of Proposition~\ref{prop:super-critical prop} in the case (Super3) and finishes the proof of all parts of the proposition.

\subsection{Convergence to Poisson-Dirichlet}\label{sec:poisson dirichlet}

In this section we establish the convergence of the cycle length to the Poisson-Dirichlet distribution in the (Super-Critical) cases. Recall that $\tau $ is defined in \eqref{eq:def of tau}. We denote by $\betadist (\alpha,\beta ) $ the beta distribution with shape parameters $\alpha $ and $\beta $ whose density with respect to Lebesgue measure on $[0,1]$ is given by
\[\frac{\Gamma (\alpha +\beta )x^{\alpha -1}\left(1-x\right)^{\beta -1}}{\Gamma (\alpha )\Gamma (\beta )},\quad x\in [0,1].\]

\begin{definition}[Modified stick-breaking process]\label{def:def of X_i}
We define a sequence of random variables $\left(X_1,X_2,\dots\right) $ inductively: Let $Y_1,Y_2,\dots $ be independent random variables with distribution $\betadist \left(1,\theta \right)$ and let $U_1,U_2,\dots $ be independent random variables with distribution $ U\left[0,1\right]$ which are independent of $Y_1,Y_2,\dots $. Set $S_0:=0$ and, inductively for $k\ge 0$,
\begin{align*}
  &A_{k+1}=\left\{U_{k+1}<1-\frac{\tau }{1-S_k}\right\},\quad X_{k+1}:=\left(1-S_k-\tau \right)Y_{k+1}\mathds{1}_{A_{k+1}},\quad S_{k+1}:=\sum _{j=1}^{k+1} X_j.
\end{align*}
\end{definition}

In other words, $X_1$ equals $0$ with probability $\tau $ and is otherwise distributed as a $\betadist \left(1,\theta \right)$ fraction of $1-\tau$. Conditioned on $X_1$, $X_2$ equals $0$ with probability $\frac{\tau }{1-X_1}$ and is otherwise distributed as a $\betadist \left(1,\theta \right)$ fraction of $1 -X_1 -\tau$, etc.

The name of the process arises from the intuitive idea of taking a stick (interval) of length $1$ for which the sub-interval of length $\tau$ in its beginning is deemed the `unbreakable' part. Then, one iteratively samples a point uniformly in the remaining part of the stick and, if the point does not fall in the unbreakable part, one removes (breaks) the part of the stick from the sampled point until the end of the stick. The existence of the unbreakable part is the reason that the process is called a `modified' stick breaking process.

Note that $X_k\in \sigma \left(Y_1,\dots ,Y_k,U_1,\dots ,U_k\right)$ and that $\sum _{k=1}^{\infty}X_k=1-\tau $, a.s.

\begin{claim}\label{cl:rearrangement_PD}
The sequence $\frac{1}{1-\tau }(X_1,X_2,\dots )$ rearranged in decreasing order (i.e. the largest element of the sequence, the second largest element and so on) has the $\PD(\theta)$ distribution.
\end{claim}

The reader unfamiliar with the $\PD(\theta )$ distribution may take the above as its definition, in which case the content of Claim~\ref{cl:rearrangement_PD} is that the distribution of the rearranged sequence does not depend on $\tau$ (the proof of the claim does not use the specific value of $\tau$ given by \eqref{eq:def of tau}).

\begin{thm}\label{thm:convergence to Poisson-Dirichlet}
In all the  (Super-Critical) cases,
\[\frac{1}{N}\left(L_1,L_2,\dots \right)\overset{d}{\longrightarrow}\left(X_1,X_2,\dots\right).\]
\end{thm}
The theorem together with Claim~\ref{cl:rearrangement_PD} imply the statements on convergence to the Poisson-Dirichlet distribution in the (Super-critical) parts of our main theorems. This is a consequence of the fact that for any $\epsilon>0$,
\begin{equation*}
  \lim_{k_0\to\infty}\lim_{N\to\infty}\mathbb{P}(\text{there exists $k\ge k_0$ for which $L_k\ge \epsilon N$}) = 0,
\end{equation*}
where $\rho$ can depend on $N$ in an arbitrary fashion in the limit. This can be proved using similar arguments to those used in Remark~\ref{remark:connection}. A similar argument appears in \cite[Theorem 5.9]{bogachev2015asymptotic}.

In order to prove Theorem~\ref{thm:convergence to Poisson-Dirichlet}, we prove the convergence in distribution of the finite-dimensional marginals by induction. We assume that
\begin{equation}\label{eq:induction hypothesis for convergence to PD}
\frac{1}{N}\left(L_1,\dots ,L_k\right)\overset{d}{\longrightarrow}\left(X_1,\dots,X_k\right),
\end{equation}
and prove that
\begin{equation}\label{eq:what we need to prove}
\frac{1}{N}\left(L_1,\dots ,L_{k+1}\right)\overset{d}{\longrightarrow}\left(X_1,\dots,X_{k+1}\right).
\end{equation}

We need the following lemma in order to prove \eqref{eq:what we need to prove}. Before the lemma let us define
\[S_k^N:=\frac{1}{N}\sum_{j=1}^k L_j,\]
and observe that, by the induction hypothesis \eqref{eq:induction hypothesis for convergence to PD}, $S_k^N\overset{d}{\longrightarrow }S_k$.
\begin{lemma}\label{lem:psi_N is close to psi}
Let $f:\R\to \R $ be a bounded and continuous function. Then
\begin{enumerate}[label=(\roman{*})]
\item
There is a bounded and continuous function $\psi =\psi _f:\R\to \R$  so that
\begin{equation}\label{eq:def of psi}
\mathbb E \left(f\left(X_{k+1}\right) \ | \ X_1,\dots,X_k \right)=\psi \left(S_k \right),\quad a.s.
\end{equation}
\item
There are functions $\psi _N=\psi _{N,f}$ so that
\begin{equation}\label{eq:def of psi_N}
\mathbb E\left(f\left(\frac{L_{k+1}}{N}\right) \ \Big| \ L_1,\dots,L_k\right)=\psi_N\left(S_k^N\right),\quad a.s.
\end{equation}
\item
The functions $\psi $ and $\psi _N$ from the previous parts
satisfy
\[\mathbb E \left|\psi_N\left(S_k^N\right)-\psi\left(S_k^N\right)\right|\to 0,\quad N\to \infty.\]
\end{enumerate}
\end{lemma}

We first explain how Lemma~\ref{lem:psi_N is close to psi} implies \eqref{eq:what we need to prove}, and therefore Theorem~\ref{thm:convergence to Poisson-Dirichlet}.

\begin{proof}[Proof of Theorem~\ref{thm:convergence to Poisson-Dirichlet}]
In order to prove \eqref{eq:what we need to prove}, it suffices to show that for any bounded and continuous functions $f_1,\dots ,f_{k+1}:\R \to \R$,
\begin{equation}\label{eq:claim in poisson dirchlet}
\mathbb E \left(f_1\left(\frac{L_1}{N}\right)\cdots f_{k+1}\left(\frac{L_{k+1}}{N}\right)\right)\to \mathbb E \left(f_1\left(X_1\right)\cdots f_{k+1}\left(X_{k+1}\right)\right),\quad N\to \infty.
\end{equation}
Using Lemma~\ref{lem:psi_N is close to psi} with $f=f_{k+1}$ we get as $N\to \infty$ that
\begin{equation}
\begin{split}
\mathbb E & \left(f_1\left(\frac{L_1}{N}\right)\cdots f_{k+1}\left(\frac{L_{k+1}}{N}\right)\right)=\mathbb E \left(\mathbb E \left(f_1\left(\frac{L_1}{N}\right)\cdots f_{k+1}\left(\frac{L_{k+1}}{N}\right) \ \Big| \ L_1,\dots ,L_k\right)\right) \\
&= \mathbb E \left(f_1\left(\frac{L_1}{N}\right)\cdots f_k\left(\frac{L_k}{N}\right)\psi_N(S_k^N)\right)=
\mathbb E \left(f_1\left(\frac{L_1}{N}\right)\cdots f_k\left(\frac{L_k}{N}\right)\psi(S_k^N)\right)+o(1),
\end{split}
\end{equation}
where in the last equation we also used that $f_i$, $1\le i\le k$, are bounded.
The function,
\[\left(y_1,\dots,y_k\right)\mapsto f_1(y_1)\cdots f_k(y_k)\psi\left(\sum _{j=1}^k y_j\right)\]
is bounded and continuous and therefore, by the induction hypothesis \eqref{eq:induction hypothesis for convergence to PD} and Lemma~\ref{lem:psi_N is close to psi}, as $N\to \infty$,
\begin{equation}
\begin{split}
\mathbb E \left(f_1\left(\frac{L_1}{N}\right)\cdots f_k\left(\frac{L_k}{N}\right)\psi(S_k^N)\right) \to & \mathbb E \left(f_1\left(X_1\right)\cdots f_k\left(X_k\right)\psi (S_k)\right) \\
=& \mathbb E \left(f_1\left(X_1\right)\cdots f_{k+1}\left(X_{k+1}\right)\right).
\end{split}
\end{equation}
This implies \eqref{eq:claim in poisson dirchlet} and therefore the theorem.
\end{proof}

We turn to prove Lemma~\ref{lem:psi_N is close to psi}. The following claim proves, in particular, the first part of the lemma.

\begin{claim}
For a bounded and continuous function $f:\R\to \R$, define $\psi=\psi_f$ on $[0,1-\tau)$ by
\begin{equation}
\begin{split}
\psi\left(s\right)&=\frac{\tau }{1-s}f(0)+\frac{\theta }{1-s}\intop _{0}^{1-s-\tau}f\left(x\right)\left(1-\frac{x}{1-s-\tau }\right)^{\theta -1}dx \\
&=\frac{\tau }{1-s}f(0)+\frac{1-s-\tau  }{1-s}\intop _{0}^{1}f\left(\left(1-s-\tau\right)y\right)\theta \left(1-y\right)^{\theta -1}dy.
\end{split}
\end{equation}
Then, $\psi$ extends to a continuous and bounded function on $\R$ and we have
\[\mathbb E \left(f\left(X_{k+1}\right)|X_1,\dots,X_k \right)=\psi \left(S_k \right),\quad a.s.\]
\end{claim}

The proof follows directly from the definition of $X_k$ and we leave it to the reader. Note that $\psi $ can be extended to a continuous and bounded function on $\R $ since the limit as $s\to 1-\tau $ exists.

We need the following claim for the proof of the third part of Lemma~\ref{lem:psi_N is close to psi}. Before the claim we fix a sequence of integers $(a_N)$ such that as $N\to \infty$,  $\max \left(L^2,\frac{N}{\log N}\right)= o(a_N)$ and $a_N=o(N)$ (note this is possible in all of the (Super-Critical) cases).

\begin{claim}\label{claim:technical claim for poisson}
In all of the (Super-Critical) cases we have,
\[\underset{N\to \infty}{\liminf }\  \frac{\theta }{\rho} \sum _{j=1}^{a_N} \varphi ^
{*j}(0)\ge \tau. \]
\end{claim}

\begin{proof}
The claim holds trivially when $\tau =0$. Thus it remains to prove the claim in cases (Super2) and (Super3).

In case (Super3),
\begin{equation}\label{eq:under assumption (4)}
\underset{N\to \infty}{\lim} \frac{\theta}{\rho } \sum _{j=1}^{a_N} \varphi^{*j}(0)=\frac{\theta }{\rho }\sum _{j=1}^{\infty}\varphi^{*j}(0)=\frac{g'(1)}{\rho }=\frac{\rho _c}{\rho }=\tau.
\end{equation}
In case (Super2), by Lemma~\ref{lem:asymptotics of kappa _j}, for any $\delta >0$,
\begin{equation}
\begin{split}
\underset{N\to \infty}{\liminf}\ \frac{\theta }{\rho }\sum _{j=1}^{a_N} \varphi^{*j}(0)\ge \underset{N\to \infty}{\lim}\frac{\theta}{\rho }\sum _{j=L^\delta }^{L^2} \varphi^{*j}(0)&=\underset{N\to \infty}{\lim} \ \frac{1}{\alpha \log N }\sum _{j=L^\delta }^{L^2}\frac{\alpha _c}{j} \\
&=\underset{N\to \infty}{\lim}\  \frac{\alpha _c \left(\log \left(L^2\right)-\log L^\delta \right)}{\alpha \log N} = \tau - \frac{\delta }{2}\tau.
\end{split}
\end{equation}
The claim follows as we can take arbitrarily small $\delta $.
\end{proof}

\begin{proof}[Proof of Lemma~\ref{lem:psi_N is close to psi}]
Denote by $\supp (S_k^N)$ the support of the discrete random variable $S_k^N$ defined by
\[ \supp (S_k^N):=\{s\in \R \ | \  \mathbb{P} (S_k^N=s)>0\}.\]
Define
\[\psi_N(s):=\sum _{l=1}^{N-sN} f\left(\frac{l}{N}\right)\cdot \frac{W_{L,l}}{N\left(1-s\right)}\cdot \frac{H_{N-sN-l}(L)}{H_{N-sN}(L)}, \quad s\in \supp (S_k^N)\]
and note that \eqref{eq:def of psi_N} holds by Lemma~\ref{lem:distribution_of_ell_1,ell_2}. Let
\[A_{N,\epsilon }=\supp(S_k^N)\cap \{s\in \R \,|\, s<1-\tau-\epsilon \}.\]
It suffices to prove that for any $\epsilon >0$,
\begin{equation}\label{eq:converge to 0 on support}
\underset{s\in A_{N,\epsilon }}{\max}\left|\psi_N(s)-\psi(s)\right|\to 0,\quad N\to \infty,
\end{equation}
since, using the induction hypothesis \eqref{eq:induction hypothesis for convergence to PD} and the fact that $\psi _N,\psi$ are bounded, we get
\begin{equation}
\begin{split}
\underset{N\to \infty}{\limsup }\ \mathbb E &\left(\left|\psi_N\left(S_k^N\right)-\psi\left(S_k^N\right)\right|\mathds{1}_{\{S_k^N\ge 1-\tau -\epsilon\}}\right)\le \\
&\underset{N\to \infty}{\limsup }\ C\cdot \mathbb P \left(S_k^N \ge 1-\tau -\epsilon \right) =C\cdot \mathbb P\left(S_k\ge 1-\tau -\epsilon \right)\overset{\epsilon \to 0}{\longrightarrow }0.
\end{split}
\end{equation}
Let us first prove that
\begin{equation}\label{eq:probability of atom}
\underset{s\in A_{N,\epsilon }}{\max} \left|\sum _{l=1}^{a_N} \frac{W_{L,l}}{N\left(1-s\right)}\cdot \frac{H_{N-sN-l}(L)}{H_{N-sN}(L)}- \frac{\tau }{1-s}\right|\to 0,\quad N\to \infty.
\end{equation}
We start with an upper bound on the sum in \eqref{eq:probability of atom}. For any positive  $\delta <\epsilon $ and $s\in \supp (S_k^N)$ we have
\begin{equation}\label{eq:1-p}
\sum _{l=1}^{a_N} \frac{W_{L,l}}{N\left(1-s\right)}\cdot \frac{H_{N-sN-l}(L)}{H_{N-sN}(L)}\le 1-\sum _{l=a_N+1}^{\lfloor \left(1-s-\tau -\delta \right)N\rfloor } \frac{W_{L,l}}{N\left(1-s\right)}\cdot \frac{H_{N-sN-l}(L)}{H_{N-sN}(L)},
\end{equation}
since
\[\sum _{l=1}^{N-S_k^NN} \frac{W_{L,l}}{N\left(1-S_k^N\right)}\cdot \frac{H_{N-S_k^NN-l}(L)}{H_{N-S_k^NN}(L)}=\mathbb E \left(1\,|\,L_1,\dots ,L_k\right)=1,\quad a.s.\]
Using Theorem~\ref{thm:supercritical integral} (with  $\delta$ as $\epsilon $) and substituting $W_{L,l}\sim \theta $ (which follows from Corollary~\ref{cor:corollary on W} for $l\ge a_N$) we obtain as $N\to \infty$,
\begin{equation}\label{eq:Riemann sum}
\begin{split}
\sum _{l=a_N+1}^{\lfloor \left(1-s-\tau -\delta \right)N\rfloor }\!\! \frac{W_{L,l}}{N\left(1-s\right)}&\cdot \frac{H_{N-sN-l}(L)}{H_{N-sN}(L)}= \frac{\theta +o (1)}{1-s}\sum _{l=a_N+1}^{\lfloor \left(1-s-\tau -\delta \right)N\rfloor }\frac{1}{N} \left(1-\frac{\frac{l}{N}}{1-s-\tau}\right)^{\theta -1} \\
=\frac{\theta +o (1)}{1-s}\intop _0^{1-s-\tau -\delta }& \left(1-\frac{x}{1-s-\tau}\right)^{\theta -1}dx= \frac{1-s-\tau}{1-s}-\frac{\left(1-s-\tau \right)^{1-\theta }\delta ^\theta }{1-s}+o (1),
\end{split}
\end{equation}
where all the estimates hold uniformly in $s\in A_{N,\epsilon }$.
Substituting the last estimate in \eqref{eq:1-p} and choosing $\delta >0$ arbitrarily small, we get
\[\sum _{l=1}^{a_N} \frac{W_{L,l}}{N\left(1-s\right)}\cdot \frac{H_{N-sN-l}(L)}{H_{N-sN}(L)}\le  \frac{\tau }{1-s}+o(1),\quad N\to \infty,\]
uniformly in $s\in A_{N,\epsilon }$.

On the other hand, using the inequality $W_{L,l}\ge \theta L^d \varphi ^{*l}(0)$, Theorem~\ref{thm:supercritical integral} and Claim~\ref{claim:technical claim for poisson}, we get
\begin{equation*}
\begin{split}
\sum _{l=1}^{a_N} \frac{W_{L,l}}{N\left(1-s\right)}\cdot \frac{H_{N-sN-l}(L)}{H_{N-sN}(L)}\ge \frac{\theta +o(1)}{1-s}\cdot \frac{1}{\rho }\sum _{l=1}^{a_N} \varphi^{*l}(0) \ge \frac{\tau }{1-s}+o(1),\quad N\to \infty ,
\end{split}
\end{equation*}
uniformly in $s\in A_{N,\epsilon }$. This completes the proof of \eqref{eq:probability of atom}.

From the last equation and \eqref{eq:Riemann sum} we also get that
\begin{equation}\label{eq:bound S_3}
\underset{N\to \infty}{\limsup} \underset{s\in A_{N,\epsilon }}{\max} \sum _{\lceil l=\left(1-s-\tau -\delta \right)N \rceil}^{N-sN} \frac{W_{L,l}}{N\left(1-s\right)}\cdot \frac{H_{N-sN-l}(L)}{H_{N-sN}(L)} \le C_\epsilon \delta^\theta.
\end{equation}

Now, for $s\in A_{N,\epsilon }$ and $\delta >0$, we write
\[\psi_N(s)=\sum _{l=1}^{N-sN} f\left(\frac{l}{N}\right)\cdot \frac{W_{L,l}}{N\left(1-s\right)}\cdot \frac{H_{N-sN-l}(L)}{H_{N-sN}(L)}=S_1+S_2+S_3,\]
where $S_1,S_2$ and $S_3$ are the sums corresponding to $l\le a_N$,\\ $a_N<l\le \left(1-s-\tau -\delta \right)N$ and $\left(1-s-\tau -\delta \right)N<l\le N-sN$ respectively. By the continuity of $f$ and \eqref{eq:probability of atom},
\[\underset{s\in A_{N,\epsilon }}{\max} \left|S_1-\frac{\tau}{1-s}f(0)\right|\to 0,\quad N\to \infty.\]
By the same arguments as in \eqref{eq:Riemann sum},
\[\underset{N\to \infty}{\limsup} \underset{s\in A_{N,\epsilon }}{\max} \left|S_2-\frac{\theta }{1-s}\intop _{0}^{1-s-\tau-\delta }f\left(x\right)\left(1-\frac{x}{1-s-\tau }\right)^{\theta -1}dx\right|\le C_{\epsilon,f} \delta ^{\theta }.\]
Finally, as $f$ is bounded and by \eqref{eq:bound S_3},
\[\underset{N\to \infty}{\limsup} \underset{s\in A_{N,\epsilon }}{\max} \left( \left|S_3\right|+\bigg|\frac{\theta }{1-s}\intop _{1-s-\tau-\delta }^{1-s-\tau }f\left(x\right)\left(1-\frac{x}{1-s-\tau }\right)^{\theta -1}dx\bigg| \right) \le C_{\epsilon,f} \delta ^\theta. \]
Equation \eqref{eq:converge to 0 on support} follows using the triangle inequality and taking arbitrarily small $\delta $.
\end{proof}

\subsection{The length of a cycle containing a given point}
To finish the proof of our main theorems in the (Super-critical) cases it remains to establish the second limit law in part~\ref{item:super-critical in d=2} of Theorem~\ref{thm:d=2} and the limit \eqref{eq:small cycles in supr critical} in part~\ref{item:super-critical in d=3} of Theorem~\ref{thm:d=3} (as convergence to the Poisson-Dirichlet law follows from Theorem~\ref{thm:convergence to Poisson-Dirichlet}). We start with Theorem~\ref{thm:d=2}. The proof that the limiting distribution has a constant density on $(0,1)$ is almost identical to the proof of part~\ref{item:part_2_sub_critical_restatement} of Corollary~\ref{cor:connection to L_1}. Indeed, using Lemma~\ref{lem:distribution_of_ell_1,ell_2}, Theorem~\ref{thm:supercritical integral}, Corollary~\ref{cor:corollary on W} and Lemma~\ref{lem:asymptotics of kappa _j}, we get that for any $0<a<b< 1$,
\begin{equation*}
\mathbb P \left(a\le \frac{\log L_1}{\log N }\le b\right)=\sum _{j=\lceil N^a \rceil} ^{j=\lfloor N^b \rfloor} \frac{W_{L,j}}{N}\cdot \frac{H_{N-j}(L)}{H_N(L)}\sim \frac{\alpha _c}{\rho }\sum _{j=\lceil N^a \rceil} ^{j=\lfloor N^b \rfloor} \frac{1}{j} \to \frac{\alpha _c}{\alpha } (b-a),\quad N\to \infty,
\end{equation*}
where in the last limit we used that $\frac{\rho }{\log N}\to \alpha $.
On the other hand, by Theorem~\ref{thm:convergence to Poisson-Dirichlet}, we have that $\frac{L_1}{N}\to X_1$ and therefore for any $\epsilon >0$,
\[\underset{N\to \infty}{\liminf}\ \mathbb P \left(\left|\frac{\log L_1}{\log N }-1\right|\le \epsilon \right)\ge \mathbb P \left(X_1>0\right)=1-\tau =1-\frac{\alpha _c}{\alpha },\]
from which the limit law follows.

We proceed to prove \eqref{eq:small cycles in supr critical} in part~\ref{item:super-critical in d=3} of Theorem~\ref{thm:d=3}. Under the corresponding assumptions, by Lemma~\ref{lem:distribution_of_ell_1,ell_2}, Theorem~\ref{thm:supercritical integral} and Corollary~\ref{cor:corollary on W}, for any fixed $j\in \N$,
\[\mathbb P (L_1=j)=\frac{W_{L,j}}{N}\cdot \frac{H_{N-j}(L)}{H_N(L)}\to \theta \rho ^{-1}\varphi ^{*j}(0),\quad N\to \infty.\]

\section{The critical case in dimensions $d\ge 2$}\label{sec:critical in dimension 5}

In this section we find the asymptotic distribution of $L_1$ in the critical regimes in dimensions $d\ge 2$, establishing the case $\alpha=\alpha_c$ in part~\ref{item:sub-critical and critical in d=2} of Theorem~\ref{thm:d=2} and proving part~\ref{item:critical in d=3} in Theorem~\ref{thm:d=3}. The results are deduced as a consequence of the estimates in the sub-critical regime, by proving a type of monotonicity statement as the density increases.

An approach via contour integrals, as done in the sub-critical and super-critical cases, is also possible, though the estimates seem quite involved in dimensions $d\in\{2,3,4\}$ where $g''(1)=\infty$. Although it is not used in the proof of our main theorems, we also follow this route and find the asymptotics of the partition function in dimensions $d\ge 5$ (see Theorem~\ref{thm:critical integral} below) as it may be useful for obtaining further information on the random permutation (e.g., the probability-generating function of the number of cycles as discussed in Remark~\ref{remark:generating_function_of_number_of_cycles}).

\subsection{Distribution of $L_1$}
We proceed to find the limiting distribution of $L_1$ in the critical regimes in dimensions $d\ge 2$. In the following claim we prove  monotonicity of the cycle weights.
\begin{claim}\label{claim:claim on W}
For each $0<\epsilon<1$ there is $j_0(\epsilon )$ such that for any  $j\ge j_0(\epsilon )$ and $L\ge 1$,
\begin{equation}\label{eq:inequality for W}
W_{L,j}\le (1+\epsilon ) W_{L,j-i},\quad 0\le i\le \frac{j}{2}.
\end{equation}
\end{claim}

\begin{proof}
First, we claim that for sufficiently large $j_0$,  \begin{equation}\label{eq:lower bound on W}
\underset{L\ge 1, \ j\ge j_0}{\inf }W_{L,j}>0.
\end{equation}
When $j_0\le j\le L^2$, this inequality follows as $W_{L,j}\ge \theta L^d \varphi ^{*j}(0)$ and by Lemma~\ref{lem:asymptotics of kappa _j}. When $j\ge \max \{L^2,j_0\}$ it follows from Lemma~\ref{lem:bound with fourier} noting that $\hat{\psi }(0)=1$.

We turn to prove \eqref{eq:inequality for W}. We consider separately the cases $j\le L^{2-\frac{1}{2d}}$ and $j> L^{2-\frac{1}{2d}}$. In the first case the inequality follows from part~\ref{item:cor on W part 1} of Corollary~\ref{cor:corollary on W} and Lemma~\ref{lem:asymptotics of kappa _j}, whereas in the second case it follows from Lemma~\ref{lem:bound with fourier}. Indeed, when $j> L^{2-\frac{1}{2d}}$,
\begin{equation*}
\begin{split}
W_{L,j}\le \theta \sum _{m\in \frac{1}{L} \mathbb Z ^d}\hat{\psi}^j(m)+CL^{-\frac{1}{2}}e^{-c\frac{j}{L^2}}&\le \theta \sum _{m\in \frac{1}{L} \mathbb Z ^d}\hat{\psi}^{j-i}(m)+CL^{-\frac{1}{2}}e^{-c\frac{j}{L^2}} \\
&\le W_{L,j-i}+CL^{-\frac{1}{2}}e^{-c\frac{j}{L^2}} \le  (1+\epsilon ) W_{L,j-i},
\end{split}
\end{equation*}
where in the last inequality we used \eqref{eq:lower bound on W}.
\end{proof}

The following lemma is the main step in the proof of the results of this section.

\begin{lemma}\label{lem:bound on partition function}
For each $0<\epsilon <1$ there is $N_0(\epsilon )$ such that:
\begin{enumerate}[label=(\roman{*})]
\item
Suppose that $d\ge 3 $ and $\rho \to \rho _c$ as $N\to \infty$. Then, for $N\ge N_0(\epsilon )$
\begin{equation}\label{eq:ratio of H d=3}
\frac{H_{N-1}(L)}{H_N(L)}\ge 1-\epsilon
\end{equation}
\item
Suppose that $d=2$ and that $\frac{\rho }{\log N}\to \alpha _c$ as $N\to \infty$. Then, for $N\ge N_0(\epsilon )$
\begin{equation}\label{eq:ratio of H d=2}
\frac{H_{N-j}(L)}{H_N(L)}\ge 1-\epsilon\quad\text{for all}\quad 0\le j\le N^{1-\epsilon}.
\end{equation}
\end{enumerate}
\end{lemma}

\begin{proof}
Using \eqref{eq:lower bound on W}, fix some $j_0>1$ for which  $\underset{L\ge 1, \ j\ge j_0}{\inf }W_{L,j}>0$. Let $L\ge 1$ be sufficiently large for the following arguments (as a function of $\epsilon$).

We start with the first part. Let $r_1:= (1+\epsilon )^{-1}$ and $\rho _1 :=r_1g'(r_1)$. Note that $\rho _1<g'(1)=\rho _c$. Let $N_1:=\lceil \rho _1 L^d\rceil $ and $j_1:= \lceil N_1^{\frac{1}{3}}\rceil$. We have that $r_{N_1,L}\to r_1$ as $L\to \infty $ by Lemma~\ref{lem:find r_N}, where $r_{N_1,L}$ is defined in \eqref{eq:def of r_N}. Therefore, using Theorem~\ref{thm:sub-critical integral} and the bounds on $a_N$ in Lemma~\ref{lem:a_N_b_N_estimates} we get that for sufficiently large $L$,
\begin{equation}\label{eq:using the sub-critical theorem}
H_{i+1}(L)\le (1+2\epsilon )H_i(L),\quad N_1-j_1\le i\le N_1-1
\end{equation}
and
\begin{equation}\label{eq:using the sub-critical theorem2}
H_{N_1-j_1}(L)\le \left(1-c_\epsilon  \right)^{j_1-j_0} H_{N_1-j_0}(L).
 \end{equation}

We proceed by proving that
\begin{equation}\label{eq:induction hypothesis}
H_K(L)\le \left(1+2\epsilon +Ke^{-c_\epsilon N_1^\frac{1}{3}}\right)H_{K-1}(L)
\end{equation}
using induction on $N_1\le K\le 2\rho _c L^d$. Note that \eqref{eq:ratio of H d=3} easily follows from this for sufficiently large $N$ as $\epsilon $ is arbitrary. We already proved the case $K=N_1$ in \eqref{eq:using the sub-critical theorem}. We make the induction hypothesis that \eqref{eq:induction hypothesis} holds for all $N_1\le K'< K$. Using Lemma~\ref{lem:distribution_of_ell_1,ell_2} we get that
\begin{equation*}
\sum _{j=1}^{K} \frac{W_{L,j}}{K}\cdot \frac{H_{K-j}(L)}{H_{K}(L)}=\mathbb P \left(1\le L_1 \le K\right)=1,
\end{equation*}
and therefore
\begin{equation}\label{eq:sum to estimate}
H_{K}(L)=\sum _{i=0}^{K-1} \frac{W_{L,K-i}}{K}\cdot H_{i}(L).
\end{equation}
Now we break the sum in \eqref{eq:sum to estimate} to three parts and estimate each one separately. First, observe that by Claim~\ref{claim:claim on W},
\begin{equation*}
\sum _{i=0}^{N_1-j_1-1} \frac{W_{L,K-i}}{K}\cdot H_{i}(L) \le (1+\epsilon )\sum _{i=0}^{N_1-j_1-1} \frac{W_{L,K-1-i}}{K-1}\cdot H_{i}(L).
\end{equation*}
Second, by \eqref{eq:using the sub-critical theorem2} and the bounds in Corollary~\ref{cor:corollary on W} and Lemma~\ref{lem:asymptotics of kappa _j},
\[\frac{W_{L,K-N_1+j_1}}{K}\cdot H_{N_1-j_1}(L)\le C(1-c_\epsilon )^{j_1-j_0} H_{N_1-j_0}(L)\le e^{-c_\epsilon N_1^{\frac{1}{3}}} H_{K-1}(L),\]
where in here we used that
\begin{equation}\label{eq:upper bound on H}
H_{K-1}(L)=\sum _{i=0}^{K-2} \frac{W_{L,K-1-i}}{K-1}\cdot H_{i}(L)\ge \frac{W_{L,K-1-N_1+j_0}}{K-1}\cdot H_{N_1-j_0}(L)\ge \frac{c}{K}H_{N_1-j_0}(L).
\end{equation}
Third, by \eqref{eq:using the sub-critical theorem} and the induction hypothesis
\begin{equation}
\begin{split}
\sum _{i=N_1-j_1+1}^{K-1 } \frac{W_{L,K-i}}{K}\cdot H_{i}(L)&=\sum _{i=N_1-j_1}^{K-2} \frac{W_{L,K-1-i}}{K}\cdot H_{i+1}(L) \\
 &\le \left(1+2\epsilon +(K-1)e^{-c_\epsilon N_1^\frac{1}{3}}\right) \sum _{i=N_1-j_1}^{K-2} \frac{W_{L,K-1-i}}{K-1}\cdot H_{i}(L).
\end{split}
\end{equation}
Adding up the contributions and using \eqref{eq:sum to estimate} we obtain
\begin{equation*}
\begin{split}
H_K(L)&\le e^{-c_\epsilon N_1^{\frac{1}{3}}}H_{K-1}(L)+\left(1+2\epsilon +(K-1)e^{-c_\epsilon N_1^\frac{1}{3}}\right)\sum _{i=0}^{K-2} \frac{W_{L,K-1-i}}{K-1}\cdot H_{i}(L)\\
&= \left(1+2\epsilon +Ke^{-c_\epsilon N_1^\frac{1}{3}}\right)H_{K-1}(L)
\end{split}
\end{equation*}
establishing \eqref{eq:induction hypothesis}.

We turn to prove the second part. Let $\alpha _1:=(1-5\epsilon )\alpha _c$, $N_1=N_1(L)$ be the smallest integer for which $\frac{N_1}{\log N_1}\ge \alpha _1 L^2$ and $j_1:=\lceil N_1^{1-3\epsilon }\rceil $.
By Lemma~\ref{lem:find r_N}, for sufficiently large $L$,
\begin{equation}
1-N_1^{-\left(1-6\epsilon \right)}< r_{N_1,L}< 1-N_1^{-\left(1-4\epsilon \right)}.
\end{equation}
Thus, by Theorem~\ref{thm:sub-critical integral} (with $N_1$ replacing $N$) and the estimate for $a_N$ in part (ii) of Lemma~\ref{lem:a_N_b_N_estimates}, we have for sufficiently large $L$ and any $0\le j\le N_1^{1-7\epsilon }$ ,
\begin{equation}\label{eq:using the sub-critical theorem d=2}
H_{i+j}(L)\le (1-N_1^{6\epsilon-1 })^{-j} H_i(L)\le (1+\epsilon)H_i(L) ,\quad N_1-j_1\le i\le N_1-j
\end{equation}
and
\begin{equation}\label{eq:using the sub-critical theorem2 d=2}
H_{i}(L)\le (1-N_1^{4\epsilon -1})^{\frac{j_1}{3}} H_{N_1-j-j_0}(L)\le e^{-\frac{1}{3} N_1^\epsilon } H_{N_1-j-j_0}(L),\quad N_1-j_1\le i \le  N_1-\frac{j_1}{2},
\end{equation}
where we recall that $j_0$ was fixed in the beginning of the proof. We proceed by proving that for any $0\le j\le N_1^{1-7\epsilon }$,
\begin{equation}\label{eq:induction hypothesis d=2}
H_K(L)\le \left(1+\epsilon +Ke^{-\frac{1}{4} N_1^\epsilon }\right) H_{K-j}(L),
\end{equation}
using induction on $N_1\le K\le N_2$, where $N_2$ is the largest integer for which $\frac{N_2}{\log N_2}\le 2\alpha _cL^2$. Note that \eqref{eq:ratio of H d=2} easily follows from this as $\epsilon $ is arbitrary. The case $K=N_1$ is given by \eqref{eq:using the sub-critical theorem d=2}. Suppose that \eqref{eq:induction hypothesis d=2} holds for $N_1\le K'< K$ and fix $0\le j\le N_1^{1-7\epsilon }$. We bound the sum in \eqref{eq:sum to estimate} as before. First, by Claim~\ref{claim:claim on W} using that $j\le N_1^{1-7\epsilon }\le \frac{1}{2}j_1$ for sufficiently large $L$,
\begin{equation*}
\sum _{i=0}^{N_1-j_1-1} \frac{W_{L,K-i}}{K}\cdot H_{i}(L) \le (1+\epsilon )\sum _{i=0}^{N_1-j_1-1} \frac{W_{L,K-j-i}}{K-j}\cdot H_{i}(L).
\end{equation*}
Second, by the bounds in Corollary~\ref{cor:corollary on W} and \eqref{eq:using the sub-critical theorem2 d=2},
\[\sum _{i=N_1-j_1}^{N_1-j_1+j-1} \frac{W_{L,K-i}}{K}\cdot H_{i}(L)\le CL^d e^{-\frac{1}{3} N_1^{\epsilon }}H_{N_1-j-j_0}(L)\le e^{-\frac{1}{4} N_1^{\epsilon }}H_{K-j}(L),\]
where in the last inequality we use the same arguments as in \eqref{eq:upper bound on H}. Third, by the induction hypothesis and \eqref{eq:using the sub-critical theorem d=2},
\begin{equation*}
\begin{split}
\sum _{i=N_1-j_1+j}^{K-1} \frac{W_{L,K-i}}{K}\cdot H_{i}(L)&=\sum _{i=N_1-j_1}^{K-j-1} \frac{W_{L,K-j-i}}{K}\cdot H_{i+j}(L) \\
 &\le \left(1+\epsilon +(K-1)e^{-cN_1^\epsilon }\right) \sum _{i=N_1-j_1}^{K-j-1} \frac{W_{L,K-j-i}}{K-j}\cdot H_{i}(L).
\end{split}
\end{equation*}
Adding up the contributions and using \eqref{eq:sum to estimate} yields \eqref{eq:induction hypothesis d=2}.
\end{proof}

\begin{cor}\label{cor:limiting distrubution in critical d>1}
We have the following limit laws as $N\to \infty$:
\begin{enumerate}[label=(\roman{*})]
\item
Suppose that $d\ge 3$ and $\rho =\rho _c$ is fixed as $N\to \infty$. Then
\begin{equation}
L_1\overset{d}{\longrightarrow }Y,
\end{equation}
where $Y$ is the integer-valued random variable defined by
\begin{equation}
\mathbb P \left( Y=j \right)=\frac{\theta \varphi^{*j}(0)}{\rho _c},\quad j\in \N.
\end{equation}
\item
Suppose that $d=2$ and $\frac{\rho }{\log N}\to \alpha _c$ as $N\to \infty$. Then
\begin{equation}
\frac{\log L_1}{\log N} \overset{d}{\longrightarrow } U[0,1].
\end{equation}
\end{enumerate}
\end{cor}

\begin{proof}
We start with part (i). Fix $j\in \N$. By Lemma~\ref{lem:distribution_of_ell_1,ell_2} and Lemma~\ref{lem:bound on partition function}, for any $\epsilon >0$ we have for sufficiently large $N$,
\begin{equation*}
\mathbb P (L_1=j)=\frac{W_{L,j}}{N}\frac{H_{N-j}(L)}{H_N(L)}\ge (1-\epsilon)^j \frac{W_{L,j}}{N}\ge (1-\epsilon)^{j}\frac{\theta \varphi ^{*j}(0)}{\rho _c}
\end{equation*}
where the last inequality follows from the definition of $W_{L,j}$ in \eqref{eq:W_L_j_def}, using that $N=\rho _cL^d$. Thus, as $\epsilon $ is arbitrary,
\begin{equation}\label{eq:for fix j}
\underset{N\to \infty}{\liminf } \ \mathbb P (L_1=j)\ge \frac{\theta \varphi ^{*j}(0)}{\rho _c} =\mathbb P (Y=j).
\end{equation}
Since \eqref{eq:for fix j} holds for any $j\in \N$, we conclude that $L_1\overset{d}{\longrightarrow }Y$ as $N\to \infty$.

%
We turn to prove part (ii). For any $0<a<b<1$, $\epsilon >0$ and sufficiently large $N$,
\begin{equation*}
\begin{split}
\mathbb P \left(a\le \frac{ \log L_1}{\log N}\le b\right)&= \sum _{j=\lceil N^{a}\rceil}^{\lfloor N^{b} \rfloor} \mathbb P (L_1=j)\ge (1-\epsilon )\sum _{j=\lceil N^{a}\rceil}^{\lfloor N^{b} \rfloor} \frac{W_{L,j}}{N} \ge  (1-\epsilon )\sum _{j=\lceil N^{a}\rceil}^{\lfloor N^{b} \rfloor} \frac{L^2\theta \varphi ^{*j}(0)}{N} \\
&\ge (1-2\epsilon )\frac{1}{\rho } \sum _{j=\lceil N^{a}\rceil}^{\lfloor N^{b} \rfloor} \frac{\alpha _c}{j}\ge (1-3\epsilon )\frac{\alpha _c \log N}{\rho }(b-a)\ge (1-4\epsilon)(b-a),
\end{split}
\end{equation*}
where in the first inequality we used Lemma~\ref{lem:distribution_of_ell_1,ell_2} and Lemma~\ref{lem:bound on partition function}, in the second inequality we used \eqref{eq:W_L_j_def}
 and in the third inequality we used Lemma~\ref{lem:asymptotics of kappa _j}. The limit law follows from this as $\epsilon $ is arbitrary.
\end{proof}

\begin{remark}\label{remark:asymptotically_IID_critical_d_ge_5}
  Similarly to the sub-critical cases (see Remark~\ref{remark:asymptotically_IID}), the analysis in Corollary~\ref{cor:limiting distrubution in critical d>1} extends in a straightforward manner to the study of the joint distribution of $L_1, L_2, \ldots$. Indeed, starting with Lemma~\ref{lem:distribution_of_ell_1,ell_2}, one obtains that when $d\ge 3 $ and $\rho =\rho _c$,
 \begin{multline*}
    \mathbb P\left(L_1 = j_1,\dots, L_m = j_m\right)\\
    =\frac{H_{N-j_1-\dots-j_m}(L)}{H_N(L)}\cdot \prod_{k=1}^m \frac{W_{L,j_k}}{N-j_1-\dots-j_{k-1}}\ge (1-\epsilon ) \prod_{k=1}^m \frac{\theta \varphi ^{*j_k}(0)}{\rho _c},
  \end{multline*}
  for any $\epsilon >0$, $j_1,\cdots j_m\in \N $ and sufficiently large $N$ (depending on $\epsilon, j_1,\cdots ,j_m $). And that, when $d=2$ and $\frac{\rho }{\log N} \to \alpha _c$,
\begin{multline*}
    \mathbb P\left(L_1 = j_1,\dots, L_m = j_m\right)
    =\frac{H_{N-j_1-\dots-j_m}(L)}{H_N(L)}\cdot \prod_{k=1}^m \frac{W_{L,j_k}}{N-j_1-\dots-j_{k-1}}\ge (1-\epsilon ) \prod_{k=1}^m \frac{\alpha _c}{\rho\cdot j_k},
  \end{multline*}
for any $\epsilon >0$, $N^{\epsilon} \le j_1,\cdots ,j_m \le N^{1-\epsilon }$ and sufficiently large $N$ (depending on $\epsilon $).

One may then follow the analogous steps to the analysis in Corollary~\ref{cor:limiting distrubution in critical d>1} and deduce that the $(L_k)$ become asymptotically independent and identically distributed, in the sense explained in the remark following Theorem~\ref{thm:d=3}.
\end{remark}

\subsection{Asymptotics of the partition function in dimensions $d\ge 5$}

\begin{thm}\label{thm:critical integral}
Suppose that $d\ge 5$ and that $\rho =\rho _c$ is fixed as $N\to \infty$. Then, uniformly in $j\le N^{\frac{1}{3}}$, we have
\begin{equation}\label{eq:H_{N-j}}
H_{N-j}(L)=\left[z^{N-j}\right]e^{G_L(z)}\sim \frac{e^{F_L(1)}N^{\frac{\theta -1}{2}}\left(\frac{g''(1)}{2g'(1)}+\frac{1}{2}\right)^{\frac{\theta -1}{2}}}{2\Gamma \left( \frac{\theta +1}{2}\right)},\quad N\to \infty.
\end{equation}
\end{thm}

In order to prove Theorem~\ref{thm:critical integral}, we use the contour
\begin{equation}\label{eq:gamma in critical case}
\gamma =\gamma \left(\eta ,\beta (\eta ),\sqrt{N}\right),
\end{equation}
where $\gamma \left(\eta ,\beta ,N\right)$ is defined in \eqref{eq:def of gamma}, $\eta $ is sufficiently small and $\beta (\eta )> \frac{\pi }{2}$ is the corresponding angle such that $\gamma _4$ is contained in the unit circle i.e., $\left|1+\eta e^{i\beta (\eta )}\right|=1$ (see Figure~\ref{fig_critical contour}). Note that, unlike the super-critical case, the distance from the singularity $z=1$ in this case is $N^{-\frac{1}{2}}$ and $\gamma $ is contained in the closed unit disc $\overline{\mathbb{D}}$.
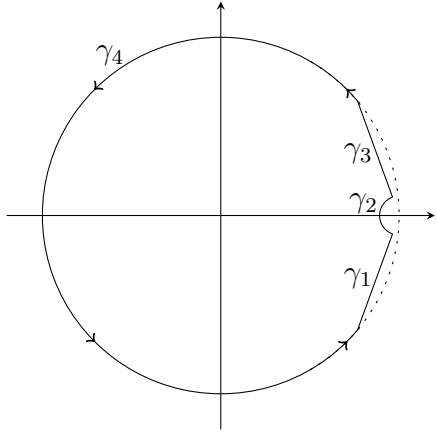
\begin{figure}
\begin{tikzpicture}
\begin{axis}[ticks=none,axis lines=middle,xmin=-1.2,xmax=1.2,ymin=-1.2,ymax=1.2,axis equal image]
\draw [dash pattern=on 1pt off 3pt, domain=-60:60] plot({cos(\x)},{sin(\x)});
\addplot[domain=0.766:0.9629]{tan(70)*x-tan(70)};
\addplot[domain=0.766:0.9629]{-tan(70)*x+tan(70)};
\draw [samples=70, domain=39:321] plot({cos(\x)},{sin(\x)});
\draw [domain=110:250] plot({1+0.11*cos(\x)},{0.11*sin(\x)});
\node[text width=0cm] at (0.69,-0.35)
{$\gamma _1 $};
\node[text width=0cm] at (0.72,0.07)
{$\gamma _2 $};
\node[text width=0cm] at (0.69,0.35)
{$\gamma _3 $};
\node[text width=0cm] at (-0.7,0.9)
{$\gamma _4 $};
\draw[->,thick] (-1/2^0.5,1/2^0.5) -- (-1/2^0.5-0.001,1/2^0.5-0.001) node [pos=0.66,above]{} ;
\draw[->,thick] (-1/2^0.5,-1/2^0.5) -- (-1/2^0.5+0.001,-1/2^0.5-0.001) node [pos=0.66,above]{} ;
\draw[->,thick] (1/2^0.5,1/2^0.5) -- (1/2^0.5-0.001,1/2^0.5+0.001) node [pos=0.66,above]{} ;
\draw[->,thick] (1/2^0.5,-1/2^0.5) -- (1/2^0.5+0.001,-1/2^0.5+0.001) node [pos=0.66,above]{} ;
\end{axis}
\end{tikzpicture}
\caption{The contour $\gamma $ in the critical case.}
\label{fig_critical contour}
\end{figure}
\begin{prop}\label{prop:critical prop}
Suppose that $d\ge 5$ and that $\rho =\rho _c$ is fixed as $N\to \infty$. Then there is $0<\eta <\frac{1}{10}$ such that the following holds for the contour $\gamma =\gamma \left(\eta ,\beta (\eta ),\sqrt{N}\right)$:
\begin{enumerate}[label=(\roman{*})]
\item
As $N \to \infty$, uniformly in $z\in D$ with $|1-z|\le N^{-\frac{4}{9}}$, we have
\[G_L(z)=-\theta \log \left(1-z\right)+F_L(1)+N\left(z-1\right)+\frac{g''(1)}{2g'(1)}N\left(z-1\right)^2+o(1).\]
\item
There is $c>0$ such that for any $z\in \gamma _1\cup \gamma _3$ and large enough $N$,
\[\real \left(G_L(z)\right)\le -\theta \log \left|1-z\right|+F_L(1)+N\real \left(z-1\right)+cN\real \left(z-1\right)^2.\]
\item
There is $c>0$ such that for any $z\in \gamma _4$ and large enough $N$,
\[\real \left(G_L(z)\right)\le F_L(1)-c N.\]
\end{enumerate}
\end{prop}

\begin{proof}
First we bound the difference between $F_L'$ and $L^dg'$.
Substituting the Taylor expansions of $F_L'$ and $g'$ that follow from \eqref{eq:power series for F_L} and \eqref{eq:g_Taylor_expansion} respectively, we obtain
\begin{equation}\label{eq:F_L'-L^dg' in critical case}
\begin{split}
\left|F_L'(z)-L^dg'(z)\right| & \le \sum _{j=1}^{\infty} \left|W_{L,j}-\theta -\theta L^d\varphi ^{*j}(0)\right| \\
&\le \sum_{j=1}^{L^2}C +C\sum_{j=L^2}^\infty (e^{-c\frac{j}{L^2}}+L^dj^{-\frac{d}{2}})\le CL^2, \quad z\in  \overline{\mathbb{D}},
\end{split}
\end{equation}
where in the second inequality we used Corollary~\ref{cor:corollary on W}, Corollary~\ref{cor:G_L is close to L^dg} and Lemma~\ref{lem:asymptotics of kappa _j}.
Moreover, by Lemma~\ref{lem:asymptotics of g at 1}, when $d\ge 5$ we have
\begin{equation}\label{eq:g'(z) near z=1}
g'(z)=g'(1)+g''(1)(z-1)+O(|z-1|^{\frac{3}{2}}),\quad z\to 1.
\end{equation}
Now, using \eqref{eq:F_L'-L^dg' in critical case} and \eqref{eq:g'(z) near z=1} and noting that $N=\rho _c L^d=g'(1)L^d$ we get
\begin{equation*}
\begin{split}
\left|\intop _1^z F_L(w)dw-N(z-1)-\frac{g''(1)}{2g'(1)}N(z-1)^2\right|=&\left|\intop _1^z \! F_L(w)dw-L^d \! \intop _1^z  \! g'(1)-g''(1)(w-1)dw\right|\\
\le \intop _1^z|F_L(w)-L^dg'(w)||dw|&+L^d\intop _1^z |g'(w)-g'(1)-g''(1)(w-1)||dw| \\
 \le CL^2|z-1|+CN|z-1|^{\frac{5}{2}}&,\quad z\in  \overline{\mathbb{D}}.
\end{split}
\end{equation*}
The first and second parts of the proposition follow by substituting the last equation in \eqref{eq:G=log +F}.

We turn to prove the third part. Take $0<\eta <\frac{1}{10}$ so that the first and second parts of the proposition hold. There is $c_\eta >0$ such that for any $z\in \gamma _4$ (that is $|z|=1$) and large enough $N$,
\[\real G_L(z)\le F_L(z)+C_\eta \le L^d\real g(z)+CL^{d-1}\le L^dg(1)-c_\eta  N\le F_L(1)-c_\eta  N,\]
where in the second and fourth inequalities we used \eqref{eq:F_L-L^dg} and the third inequality follows from the fact that $g$ has non-negative Taylor coefficients which are positive with finitely many exceptions.
\end{proof}

Now, we turn to prove Theorem~\ref{thm:critical integral}.

\begin{proof}[Proof of Theorem~\ref{thm:critical integral}]
Let $\gamma $ be the contour from \eqref{eq:gamma in critical case} with $0<\eta <\frac{1}{10}$ such that Proposition~\ref{prop:super-critical prop} holds. Note that we also have $\frac{\pi }{2}< \beta (\eta )<\frac{5\pi }{8}$. By the Cauchy integral formula
\[\left[z^{N-j}\right] e^{G_L(z)}=\frac{1}{2\pi i}\ointop _{\gamma }\frac{e^{G_L(z)}}{z^{N-j+1}}dz=I_1+I_2+I_3,\]
where $I_1$, $I_2$ and $I_3$ are the corresponding integrals over $\gamma \cap \{ |1-z|\le N^{-\frac{4}{9}}\}$,\\ $\left(\gamma _1 \cup \gamma _3 \right)\cap \{ |1-z|\ge N^{-\frac{4}{9}}\}$ and $\gamma _4 $ respectively.

We start by estimating $I_1$. The following holds as $N\to\infty$, uniformly in $0\le j\le N^{\frac{1}{3}}$ and $z\in \overline{\mathbb{D}}\cap \{N^{-\frac{1}{2}}\le \left|1-z\right|\le N^{-\frac{4}{9}}\}$. Expanding $\log z$, we obtain
\begin{equation}\label{eq:Taylor expansion of log}
\begin{split}
z^{-\left(N-j+1\right)}=\exp \left[-N\left(z-1\right)+\frac{N}{2}\left(z-1\right)^2+O\left(N|z-1|^3+N^{\frac{1}{3}}|1-z|\right)\right]\\
=\exp \left[-N\left( z-1\right) +\frac{N}{2}\left(z-1\right)^2+o(1)\right].
\end{split}
\end{equation}
Therefore, by the first part of Proposition~\ref{prop:critical prop},
\[I_1=\frac{e^{F_L(1)}}{2\pi i}\!\! \intop_{\gamma \cap \{ |1-z|\le N^{-\frac{4}{9}}\}} \!\! \left(1-z\right)^{-\theta }\exp \left[\left(\frac{g''(1)}{2g'(1)}+\frac{1}{2}\right)
N\left(z-1\right)^2+o(1)\right]dz.\]
We make the change of variables
\[z=1-\left(\frac{g''(1)}{2g'(1)}+\frac{1}{2}\right)^{-\frac{1}{2}} N^{-\frac{1}{2}}\sqrt{-\omega },\quad \omega =-\left(\frac{g''(1)}{2g'(1)}+\frac{1}{2}\right)
N\left(z-1\right)^2,\]
to obtain
\[I_1=\frac{e^{F_L(1)}\left(\frac{g''(1)}{2g'(1)}+\frac{1}{2}\right)^{\frac{\theta -1}{2}}
N^{\frac{\theta -1}{2}}}{4\pi i}\intop _{\tilde{\gamma }}\left(-\omega \right)^{\frac{-\theta -1}{2}}e^{-\omega +o(1)}\]
where $\tilde{\gamma }$ is the image of $\gamma \cap \{ |1-z|\le N^{-\frac{4}{9}}\}$ under the change of variables, which is a modification of the $\gamma '\left(2\beta (\eta )-\pi\right)$ from \eqref{eq:def of gamma '} having the circular arc at radius $\big(\frac{g''(1)}{2g'(1)}+\frac{1}{2}\big) $ and having finite `arms', terminating at radius $\big(\frac{g''(1)}{2g'(1)}+\frac{1}{2}\big)N^{\frac{1}{9}} \to \infty $. As the integral in \eqref{eq:Hankel integral} with $s=\frac{\theta +1}{2}$ and $\beta =2\beta (\eta )-\pi <\frac{\pi }{2}$ converge to a non-zero quantity, we conclude that
\begin{equation}\label{eq:estimate on I_1 critical}
I_1\sim \frac{e^{F_L(1)}N^{\frac{\theta -1}{2}}\left(\frac{g''(1)}{2g'(1)}+\frac{1}{2}\right)^{\frac{\theta -1}{2}}}{2\Gamma \left( \frac{\theta +1}{2}\right)}.
\end{equation}

We turn to bound the rest of the integral. As $\eta <\frac{1}{10}$, and by the same arguments as in \eqref{eq:expanding log}, we have
\begin{equation}\label{eq:for small eta}
\left|z^{-N+j-1}\right|\le |z|^{-N-1}\le Ce^{-N\real (\log z)}\le C e^{-N\real \left(z-1\right)},\quad z\in \gamma _1\cup \gamma _3.
\end{equation}
Thus, by the second part of Proposition~\ref{prop:critical prop},
\begin{equation}
\begin{split}
|I_2|&\le Ce^{F_L(1)}\!\!\!\! \intop _{\left(\gamma _1 \cup \gamma _3 \right)\cap \{ |1-z|\ge N^{-\frac{4}{9}}\}}\!\!\!\! \left|1-z\right|^{-\theta }e^{cN\real \left(z-1\right)^2}|dz|\le \\
&\le Ce^{F_L(1)}N^{\frac{\theta -1}{2}}\!\!\!\! \intop _{\gamma '\left(2\beta (\eta )-\pi \right)\cap \{|\omega |\ge N^{\frac{1}{9}}\}}\!\!\!\!|\omega |^{\frac{-\theta -1}{2} }e^{-c\real\omega }|d\omega |=o(I_1),\quad N\to \infty,
\end{split}
\end{equation}
where in the second inequality we changed the variables by $\omega =-N(z-1)^2$.

It remains to bound the integral over $\gamma _4$. By the third part of Proposition~\ref{prop:critical prop},
\[|I_3|\le \intop _{\gamma _4}\left|\frac{e^{G_L(z)}}{z^{N-j+1}}\right||dz|=\intop _{\gamma _4}e^{\real (G_L(z))}|dz|\le Ce^{F_L(1)-\epsilon N}, \]
which is exponentially smaller than $I_1$ by \eqref{eq:estimate on I_1 critical}.
\end{proof}

\section{The critical case in dimension 1}\label{sec:critical in one dimension}

In this section we prove part~\ref{item:critical in d=1} in Theorem~\ref{thm:d=1}. Thus, throughout this section we assume that $d=1$ and $\frac{\rho }{\sqrt{N}}\to \alpha \in (0,\infty)$ as $N\to \infty$.

In the following theorem we find the asymptotic behavior of $H_{N-j}(L)$. In the theorem and what follows we set $\sigma ^2=\var (X)$. Recall that we use the standard branches of the functions $\log (z)$, $\sqrt{z}$ (see Section~\ref{sec:preliminaries}).

\begin{thm}\label{thm:complex integral in critical dimension 1}
For any $\epsilon >0$ we have
\[H_{N-j}(L)\sim \frac{C_0e^{G_L(1-L^{-2})}}{N} \left(1-\frac{j}{N}\right)^{-\frac{3}{2}} \sum _{n=0}^{\infty } (-1)^n {{-2\theta }\choose{n}} (\theta +n)\exp \left(\frac{-(\theta +n)^2}{2 \alpha ^2\sigma ^2 \left(1-\frac{j}{N}\right)}\right),\]
as $N\to \infty $ uniformly in $0\le j\le \left(1-\epsilon \right)N$, where
\[C_0:=\frac{e^{\sqrt{2}\sigma ^{-1}\theta }\left(1-e^{-\sqrt{2}\sigma ^{-1}}\right)^{2\theta }}{\sqrt{2\pi }\alpha \sigma }.\]
\end{thm}

In order to prove Theorem~\ref{thm:complex integral in critical dimension 1} we need the following integral calculation.

\begin{claim}\label{claim:complex analysis}
 For any $a>0$,
\begin{equation}\label{eq:complex integral with residue}
\intop _{\gamma '\left(\frac{2\pi }{3}\right)}\frac{\omega e^{ a\omega ^2+2\theta \omega }}{\left(1-e^{2\omega }\right)^{2\theta }}d\omega =- \frac{i\sqrt{\pi}}{a^{\frac{3}{2}}} \sum _{n=0}^{\infty } (-1)^n {{-2\theta }\choose{n}} (\theta +n)e^{- \frac{(\theta +n)^2}{a}},
\end{equation}
where the contour $\gamma '(\beta )$ is defined in \eqref{eq:def of gamma '}.
\end{claim}

\begin{proof}
First, we note that the integral on the left converges absolutely. Using the Taylor expansion of $\left(1-x\right)^{-2\theta }$ we obtain
\[ \intop _{\gamma '\left(\frac{2\pi }{3}\right)}\frac{\omega e^{ a\omega ^2+2\theta  \omega }}{\left(1-e^{2\omega }\right)^{2\theta }}d\omega = \sum _{n=0}^\infty (-1)^n {{-2\theta }\choose{n}}  \intop _{\gamma '\left(\frac{2\pi }{3}\right)} \omega e^{a\omega ^2 +2(\theta +n)\omega }d\omega ,\]
where in here we used that
\[ \intop _{\gamma '\left(\frac{2\pi }{3}\right)}\left|\omega e^{ a\omega ^2+2\theta  \omega }\right|\sum _{n=0}^\infty \Big|(-1)^n {{-2\theta }\choose{n}}e^{2n\omega }\Big|\left|d\omega \right| = \intop _{\gamma '\left(\frac{2\pi }{3}\right)}\frac{\big| \omega e^{ a\omega ^2+2\theta  \omega }\big| }{\left(1-\left|e^{2\omega }\right|\right)^{2\theta }}\left|d\omega \right|<\infty.\]
Now, by deforming the contour, we see that it suffices to evaluate
\[\intop _{\gamma ^{(n)}} \omega e^{a\omega ^2 +2(\theta +n)\omega }d\omega=ie^{-\frac{(\theta +n)^2}{a}}\intop _{\R } \left(-\frac{\theta +n}{a}+it\right)e^{-at^2}dt=-\frac{i\sqrt{\pi}}{a^{\frac{3}{2}}}(\theta +n)e^{-\frac{(\theta +n)^2}{a}},\]
where $\gamma ^{(n)}(t):=-\frac{\theta +n}{a}+it$ for  $t\in \R$.
\end{proof}

For the proof of Theorem~\ref{thm:complex integral in critical dimension 1} we also require the following lemma.

\begin{lemma}\label{lem:sinh}
We have
\begin{equation*}
\begin{split}
G_L(z)&=c_0+G_L(1-L^{-2})-\sqrt{2}\sigma ^{-1}\theta L\sqrt{1-z}\\
 &-2\theta \log \left(1-e^{-\sqrt{2}\sigma ^{-1}L\sqrt{1-z}}\right) +O\left(|1-L^{-2}-z|L\log L\right),\quad z\in \Delta _0,
\end{split}
\end{equation*}
where $c_0=\sqrt{2}\sigma ^{-1}\theta +2\theta \log \left(1-e^{-\sqrt{2}\sigma ^{-1}}\right)$. In particular,
\begin{equation}\label{eq:bound on re critical dimension 1}
\real (G_L(z))\le C+G_L(1-L^{-2})+C|1-z|L\log L,\quad z\in \Delta _0,\  |1-z|\ge L^{-2}.
\end{equation}
\end{lemma}

\begin{proof}
We start be proving that
\begin{equation}\label{eq:G_L' in critical d=1}
\left|G_L'(z)-\frac{\theta L}{\sqrt{2}\sigma \sqrt{1-z}}\coth \left(\frac{L\sqrt{1-z}}{\sqrt{2}\sigma }\right)\right|\le CL\log L,\quad z\in \Delta _0.
\end{equation}
Using \eqref{eq:analytic continuation for G_L'}, we get that for any $z\in \Delta _0$,
\begin{equation*}
\begin{split}
& \left|G_L'(z)-\sum_{m\in \frac{1}{L} \mathbb Z}\frac{\theta }{1-z+2\pi ^2\sigma ^2m^2}\right| \\
& \le \sum_{\substack{m\in \frac{1}{L}\Z \\ |m|\le 1}}\left|\frac{\theta \hat {\varphi}(m)}{1-z\hat {\varphi}(m)}-\frac{\theta }{1-z+2\pi ^2\sigma ^2m^2}\right|
+\sum_{\substack{m\in \frac{1}{L}\Z \\ |m|\ge 1}}\left|\frac{\theta \hat {\varphi}(m)}{1-z\hat {\varphi}(m)}\right|+\sum_{\substack{m\in \frac{1}{L}\Z \\ |m|\ge 1}}\left|\frac{\theta }{1-z+2\pi ^2\sigma ^2m^2}\right| \\
&\le  C \sum_{\substack{m\in \frac{1}{L}\Z \\ 0<|m|\le 1}}\frac{|m|}{\left|1-z+2\pi ^2\sigma ^2m^2\right|}+C \sum_{\substack{m\in \frac{1}{L}\Z \\ \left|m\right|\ge 1}}\left|\hat {\varphi}(m)\right|+C\sum_{\substack{m\in \frac{1}{L}\Z \\ \left|m\right|\ge 1}}\frac{1}{m^2}\\
&\le C L \sum _{m=1}^L \frac{m}{m^2}+CL\le CL\log L,
\end{split}
\end{equation*}
where in the second inequality we used \eqref{eq:6} and \eqref{eq:taylor of phi} to bound the first sum and \eqref{eq:3} to bound the second sum, and in the last inequality we used that $\hat {\varphi}$ is Schwartz. Equation \eqref{eq:G_L' in critical d=1} follows from the last bound and the identity (see \cite[page 351, example (2.1)]{gamelin2003complex}),
\begin{equation}\label{eq:complex identity}
\sum_{k\in \mathbb Z}\frac{1}{w^2+k^2}=\frac{\pi \coth (\pi w)}{w},\quad w\in \mathbb C \setminus \{ik : k\in \Z\}
\end{equation}
with $w=\frac{L\sqrt{1-z}}{\sqrt{2}\pi \sigma }$.

Next, by a straightforward calculation we have, for any $z\in \Delta _0$,
\[\frac{d}{dz} \left(\sqrt{2}\sigma ^{-1}\theta L\sqrt{1-z} +2\theta \log \left(1-e^{-\sqrt{2}\sigma ^{-1}L\sqrt{1-z}}\right) \right)= \frac{- \theta L}{\sqrt{2}\sigma \sqrt{1-z}}\coth \left(\frac{L\sqrt{1-z}}{\sqrt{2}\sigma }\right).\]
Thus, using \eqref{eq:G_L' in critical d=1}, we get that for any $z\in \Delta _0$,
\begin{equation*}
\begin{split}
G_L(z)&=G_L\left(1-L^{-2}\right)+\intop _{1-L^{-2}}^z G_L'(w)dw=c_0+G_L(1-L^{-2})\\
&-\sqrt{2}\sigma ^{-1}\theta L\sqrt{1-z}
 -2\theta \log \left(1-e^{-\sqrt{2}\sigma ^{-1}L\sqrt{1-z}}\right) +O\left(|1-L^{-2}-z|L\log L\right).\qedhere
\end{split}
\end{equation*}
\end{proof}

\begin{proof}[Proof of Theorem~\ref{thm:complex integral in critical dimension 1}]
Fix $0<\epsilon <1$ and assume throughout the proof that $0\le j\le (1-\epsilon )N$ and that $N$ is large enough. Take $0<\eta <\frac{1}{10}$ such that $\gamma :=\gamma \big(\eta ,\frac{\pi }{3} ,\frac{L^2}{2\sigma ^2 }\big)\subseteq \Delta _0$, where $\gamma \left(\eta ,\beta ,N\right)$ is defined in \eqref{eq:def of gamma}. We choose a sequence $t_N\le N^{\frac{1}{3}}$ such that as $N\to \infty$ we have $t_N\to \infty$ and $t_N\big(\frac{N }{L^2}-\alpha ^2 \big)\to 0$. By the Cauchy integral formula
\[H_{N-j}(L)=\frac{1}{2\pi i}\ointop_\gamma \frac{e^{G_L(z)}}{z^{N-j+1}}dz=I_1+I_2+I_3,\]
where $I_1,I_2$ and $I_3$ are the corresponding integrals over $\gamma \cap \{|1-z|\le \frac{t_N}{N}\}$,\\ $\gamma \cap \{\frac{t_N}{N}\le |1-z|\le N^{-\frac{1}{3}}\}$ and $\gamma \cap \{|1-z|\ge N^{-\frac{1}{3}}\}$.

We start by evaluating $I_1$. The following holds as $N\to\infty$, uniformly in $0\le j\le (1-\epsilon)N$ and $z\in \Delta _0$ such that $|1-z|\le \frac{t_N}{N}$. First, by Lemma~\ref{lem:sinh} we have that
\begin{equation}\label{eq:e^G_L d=1}
e^{G_L(z)}=e^{c_0+G_L(1-L^{-2})}e^{2\theta \omega }\left(1-e^{2\omega }\right)^{-2\theta }\left(1+o(1)\right),
\end{equation}
where $\omega :=-2^{-\frac{1}{2}}\sigma ^{-1}L\sqrt{1-z}$. Second, expanding $\log z$ at $z=1$ we get
\begin{equation}\label{eq:z^N-j d=1}
z^{-\left(N-j+1\right)}=\exp \left[-N\Big(1-\frac{j}{N}\Big)\left(z-1\right)+o(1)\right]=\exp \left[2\alpha ^2\sigma ^2\Big(1-\frac{j}{N}\Big)\omega ^2+o(1)\right],
\end{equation}
where the second equality relies on the fact that $t_N\left( \frac{N}{L^2}-\alpha ^2 \right)\to 0$. Substituting \eqref{eq:e^G_L d=1} and \eqref{eq:z^N-j d=1} in $I_1$ and changing variables from $z$ to $\omega $ we obtain that
\begin{equation*}
I_1= \frac{-2\sigma ^2 e^{c_0+G_L(1-L^{-2})}}{\pi i L^2}\intop_{\tilde{\gamma }}\frac{\omega e^{2\theta \omega}}{\left(1-e^{2\omega }\right)^{2\theta }}\exp \left(2\alpha ^2\sigma ^2\Big(1-\frac{j}{N}\Big)\omega ^2+o(1)\right)d\omega,
\end{equation*}
where $\tilde{\gamma }$ is the image of $\gamma \cap \{|1-z|\le \frac{t_N}{N}\}$ under the change of variables, which is a modification of the contour $\gamma '(\frac{2\pi }{3})$ from Claim~\ref{claim:complex analysis} having finite `arms', terminating at radius $\frac{L}{\sqrt{2}\sigma \sqrt{N}}\sqrt{t_N}\to \infty$. As the integral in \eqref{eq:complex integral with residue} with $a=2\alpha ^2 \sigma ^2\big(1-\frac{j}{N}\big)$ converges to a non-zero quantity, one may verify that
\begin{equation}\label{eq:final asymptotic I_1}
I_1\sim \frac{C_0e^{G_L(1-L^{-2})}}{N} \left(1-\frac{j}{N}\right)^{-\frac{3}{2}} \sum _{n=0}^{\infty } (-1)^n {{-2\theta }\choose{n}} (\theta +n)\exp \left(\frac{-(\theta +n)^2}{2 \alpha ^2\sigma ^2 \left(1-\frac{j}{N}\right)}\right)
\end{equation}
uniformly in $0\le j\le (1-\epsilon)N$.

We turn to bound $I_2$. By the same arguments as in \eqref{eq:expanding log}, we have
\[ \left|z^{-\left(N-j+1\right)}\right|= e^{-\left(N-j+1\right)\real (\log z)}\le e^{-\epsilon N (\real z-1)}\le e^{-c_\epsilon  N | z-1|} ,\quad z\in \gamma , \ \frac{2\sigma ^2}{L^2}\le |1-z|\le \eta .\]
Therefore, using \eqref{eq:bound on re critical dimension 1}, we get as $N\to \infty$,
\[\left|I_2\right|\le Ce^{G_L(1-L^{-2})}\!\!\!\! \intop_{\gamma \cap \{\frac{t_N}{N}\le |1-z|\le N^{-\frac{1}{3}}\}}\!\!\!\! e^{-c_\epsilon  N |z-1|}|dz |\le \frac{Ce^{G_L(1-L^{-2})}}{N}\!\!\! \intop _{ \gamma' (\frac{\pi}{3})\cap \{|s|\ge t_N\}}\!\!\! e^{-c_\epsilon  |s| }|ds|=o(I_1),\]
where $\gamma '(\beta )$ is defined in \eqref{eq:def of gamma '} and where in the second inequality we made the change of variables $z=1+N^{-1}s$.

Finally, we bound $I_3$ by
\[|I_3|\le \intop_{\gamma \cap \{|1-z|\ge N^{-\frac{1}{3}}\}} \left| \frac{e^{G_L(z)}}{z^{N-j+1}}\right| |dz|\le Ce^{CL}\big(1+cN^{-\frac{1}{3}}\big)^{-\epsilon  N}\le e^{-c_\epsilon N^{\frac{2}{3}}}=o(I_1), \quad N\to \infty,\]
where in the second inequality and in the final estimate we used that $|G_L(z)|\le CL$ for any $z\in \Delta _0$ with $|1-z|\ge L^{-2}$. This fact follows, for example, from  \eqref{eq:third_part_stronger_estimate_dim_1_2} and the definition of $F_L$ in \eqref{eq:def of F}.
\end{proof}

\begin{cor}
We have
\[\frac{L_1}{N}\overset{d}{\longrightarrow }\mu, \quad N\to \infty, \]
where $\mu $  is a continuous probability measure supported on $\left(0,1\right)$ with density function proportional to
\[\left( \sum _{m\in \Z} e^{-2\pi ^2 \sigma ^2 \alpha ^2 m^2 x}  \right)\left(1-x\right)^{-\frac{3}{2}} \sum _{n=0}^{\infty } (-1)^n {{-2\theta }\choose{n}} (\theta +n)\exp \left(\frac{-(\theta +n)^2}{2 \alpha ^2\sigma ^2 \left(1-x\right)}\right),\quad x\in \left(0,1\right). \]
\end{cor}

\begin{proof}
Let $0<a<b<1$. By Lemma~\ref{lem:bound with fourier} and \eqref{eq:fourier of psi} we have that
\[W_{L,j}\sim \theta \sum _{m\in \Z} e^{-2\pi ^2 \sigma ^2 m^2 \frac{j}{L^2}}\sim \theta \sum _{m\in \Z} e^{-2\pi ^2 \sigma ^2 \alpha ^2 m^2 \frac{j}{N}},\quad N\to \infty\]
uniformly in $aN\le j\le bN$, where the second estimate follows as the function $\sum _{m\in \Z} e^{-2\pi ^2 \sigma ^2 m^2 x}$ is continuously differentiable on $[a,b]$ and as $\frac{N}{L^2}\to \alpha ^2 $. Thus, by Lemma~\ref{lem:distribution_of_ell_1,ell_2}, and Theorem~\ref{thm:complex integral in critical dimension 1} we have as $N\to \infty$,
\begin{equation*}\label{eq:probability in critical dimension 1}
\begin{split}
&\mathbb P\left(a\le \frac{L_1}{N}\le b\right)=\sum _{j=\lceil aN \rceil }^{\lfloor bN \rfloor }\frac{W_{L,j}}{N}\frac{H_{N-j}(L)}{H_N(L)} \\
&\sim \frac{1}{Z} \!\! \sum _{j=\lceil aN \rceil }^{\lfloor bN \rfloor }\frac{1}{N} \left( \sum _{m\in \Z} e^{-2\pi ^2 \sigma ^2 \alpha ^2 m^2 \frac{j}{N}}  \right)\left(1-\frac{j}{N}\right)^{-\frac{3}{2}} \sum _{n=0}^{\infty } (-1)^n {{-2\theta }\choose{n}} (\theta +n)\exp \left(\frac{-(\theta +n)^2}{2 \alpha ^2\sigma ^2 \left(1-\frac{j}{N}\right)}\right)\\
&\to \frac{1}{Z}\intop _a^b \left( \sum _{m\in \Z} e^{-2\pi ^2 \sigma ^2 \alpha ^2 m^2 x}  \right)\left(1-x\right)^{-\frac{3}{2}} \sum _{n=0}^{\infty } (-1)^n {{-2\theta }\choose{n}} (\theta +n)\exp \left(\frac{-(\theta +n)^2}{2 \alpha ^2\sigma ^2 \left(1-x\right)}\right)dx,
\end{split}
\end{equation*}
where
\[Z=\frac{1}{\theta } \sum _{n=0}^{\infty } (-1)^n {{-2\theta }\choose{n}} (\theta +n)\exp \left(\frac{-(\theta +n)^2}{2 \alpha ^2\sigma ^2 }\right)\]
and where the last limit follows as the integrand is continuous on $\left[a,b\right]$. Thus, it suffices to show that
\[\underset{\epsilon \to 0}{\lim } \ \underset{N\to \infty}{\lim } \ \mathbb{P}\left(\epsilon  \le \frac{L_1}{N}\le (1-\epsilon)  \right) =1.\]

First, note that by the bounds given in Corollary~\ref{cor:corollary on W} and the asymptotics in Lemma~\ref{lem:asymptotics of kappa _j}, we have
\begin{equation}\label{eq:simple bound on W}
W_{L,j}\le C+C\frac{L}{\sqrt{j}},\quad L\ge 1, \  j\in \N.
\end{equation}
Thus, using Theorem~\ref{thm:complex integral in critical dimension 1} we obtain that
\begin{equation}
\mathbb P\left( L_1 \le \epsilon N \right)=\sum _{j=1}^{ \lfloor \epsilon N\rfloor }\frac{W_{L,j}}{N}\frac{H_{N-j}(L)}{H_N(L)} \le \frac{C}{N} \sum _{j=1}^{ \lfloor \epsilon N\rfloor } \left(1+\frac{L}{\sqrt{j}}\right)\le C \epsilon + C_\alpha \sqrt{\epsilon } \overset{\epsilon \to 0}{\longrightarrow } 0.
\end{equation}

Next, by \eqref{eq:generating_h(v)} we have for large enough $N$,
\begin{equation}
\sum _{n=0}^{ \lfloor \epsilon N\rfloor } H_{n}(L)\le C \sum _{n=0} ^{\infty} H_{n}(L) \left(1-\frac{1}{\epsilon N}\right)^{n}=Ce^{G_L(1-\frac{1}{\epsilon N})}.
\end{equation}
Therefore, by Theorem~\ref{thm:complex integral in critical dimension 1} with $j=0$ and \eqref{eq:simple bound on W}, we have
\begin{equation*}
\begin{split}
\mathbb P\left( L_1 \ge (1-\epsilon ) N \right)&=\sum _{j=\lceil (1-\epsilon ) N\rceil }^{N} \frac{W_{L,j}}{N}\frac{H_{N-j}(L)}{H_N(L)}\le \frac{C_\alpha }{NH_N(L)} \sum _{n=0}^{ \lfloor \epsilon N\rfloor } H_{n}(L)  \\
&\le C_\alpha e^ {G_L\left(1-\frac{1}{\epsilon N}\right) -G_L\left(1-L^{-2}\right)} \overset{N \to \infty }{\longrightarrow } C_\alpha  e^{c_0-\frac{\sqrt{2}\theta }{\sigma \alpha \sqrt{\epsilon }}}\left(1-e^{-\frac{\sqrt{2} }{\sigma \alpha \sqrt{\epsilon }}}\right)^{-2\theta } \overset{\epsilon  \to 0 }{\longrightarrow }0,
\end{split}
\end{equation*}
where the limit as $N\to \infty$ follows from Lemma~\ref{lem:sinh}.
\end{proof}

\bibliographystyle{amsplain}
\bibliography{biblicomplete}

\end{document}